\documentclass[10pt]{article}
\usepackage{amsmath,amssymb,amsthm,amscd,esint}

\newtheorem{theorem}{Theorem}[section]
\newtheorem{lemma}[theorem]{Lemma}

\newtheorem{corollary}[theorem]{Corollary}
\newtheorem{proposition}[theorem]{Proposition}

\theoremstyle{definition}

\theoremstyle{remark}
\newtheorem{remark}[theorem]{Remark}

\numberwithin{equation}{section}
\numberwithin{theorem}{section}

\DeclareMathOperator{\tr}{tr} 
\DeclareMathOperator{\diam}{diam}

\DeclareMathOperator{\Ric}{Ric}

\begin{document}

\title{Finite Time Noncollapsing K\"ahler-Ricci Flow: Gaussian Type Heat Kernel, Gromov-Hausdorff Limits}

\author{Lei Zhang\thanks{ Email: leizizhang1992@163.com}\\
School of Mathematics and Statistics, Jiangsu Normal University\\
Zhenlei Zhang\thanks{Supported partially by NSFC 11431009 and NSFC 11771301. Email: zhleigo@aliyun.com}\\
School of Mathematical Sciences, Capital Normal University}
\date{}

\maketitle




\begin{abstract}
In this paper, we study heat kernel estimates along the K\"ahler-Ricci flow with finite time noncollapsing singularities over a compact K\"ahler manifold. The starting point is the type-II scalar-curvature estimate, which gives the natural parabolic scale $r\lesssim T-t$ near the singular time. By combining the distance distortion and space-time cut-off estimates of Bamler--Zhang with entropy, mean-value inequality, and conjugate heat-kernel estimates, we obtain two-sided Gaussian bounds and gradient bounds that are uniform on this scale. Under the global type-I scalar-curvature estimate, the same estimates hold on every scale.

We then use these estimates and the monotonicity of the $W_1$-distance between conjugate heat flows to construct a canonical singular-time pseudo-distance. The distances $d_t$ (induced by $\omega(t)$) converge uniformly to this pseudo-distance. Its metric quotient is the unique Gromov--Hausdorff limit. 

Finally, we apply Bamler's theory and the recent work of Hallgren--Zhang to obtain the following. Every fixed-base-point type-I blow-up admits a subsequential metric-flow limit whose time slices are singular shrinking K\"ahler--Ricci soliton spaces with normal complex-analytic and log-terminal structure, unique tangent cones, and a polarized Fano fibration. For moving base points, the limiting metric flow, equipped with its limiting conjugate heat flow, is a metric soliton precisely under an entropy-pinching condition. 

\end{abstract}

\tableofcontents



\section{Introduction}
Let $(X^n,\omega_0)$ be a compact K\"ahler manifold of complex dimension $n.$ The K\"ahler-Ricci flow on $X$ is defined by
\begin{equation}\label{KRF}
  \begin{cases}
  \frac{\partial\omega}{\partial t}&=-Ric(\omega)\\
  \omega(0)&=\omega_0
  \end{cases}
\end{equation}
The flow has a solution for short positive time \cite{Cao85} and the maximal existence time $T$ of (\ref{KRF}) is uniquely determined by a result of Tian-Zhang in \cite{TZh06}, namely,
$$T=\sup\{t>0:~[\omega_0]-2\pi tc_1(X)>0\}.$$
The K\"ahler class $[\omega(t)]$ at time $t$ is precisely given by $[\omega_0]-2\pi tc_1(X).$

In this paper, we focus on the finite time noncollapsing of K\"ahler-Ricci flow (\ref{KRF}). That is, $T<\infty,$ the cohomology class $[\beta]:=[\omega_0]-2\pi Tc_1(X)$ is a limit of K\"ahler classes which is nef and big (Indeed, it is itself no longer K\"ahler). The study of K\"ahler-Ricci flow (\ref{KRF}) on compact K\"ahler manifolds with finite time singularity has a long history and many spectacular results have appeared during the years (see \cite{Bam18,Cao85,SesTi08,So14,SSW13,SW13,ST07,ST12,ST17,STZ19,TZhZ16,TZhZ21,TZZZZ,TZh06,Tsu88,TosZh18,ZhZ10} and the references given therein). La Nave-Tian \cite{LT16} introduced a new continuity method which provides an alternative way in carrying out the Analytic Minimal Model Program proposed in \cite{LT16}. In the case of finite time continuity equation, the differential geometric properties of the limit have been studied in many papers, see \cite{LTZ17,Zha24,ZhaZh21,ZZ19,ZZ20} and the references therein.

Let us motivate the study in our case. Consider a Ricci flow $(M,g(t)_{t\in[0,T)})$ that develops a singularity at time $T<\infty.$ It is shown by Hamilton \cite{Ham82} that the maximum of the norm of the Riemannian curvature tensor diverges as $t\to T.$ Later, Sesum \cite{Ses05} showed that also the norm of the Ricci tensor has to become unbounded as $t\to T$ (see \cite{ZhZ14} for K\"ahler case). In the K\"ahler case, Zhang \cite{ZhZ10} showed that the scalar curvature must be unbounded as $t\to T,$ and that it is bounded from above by $\frac{C}{(T-t)^2}$ (a type-II bound) under an additional assumption on the limiting class. Jian-Song-Tian \cite{JST23} proved the local type-I scalar curvature bound near the Ricci vertex and they also proved that every solution of K\"ahler-Ricci flow on projective bundle over an $n$-dimensional K\"ahler-Einstein manifold with Calabi symmetry must develop type-I curvature singularities. The general scalar type-I estimate has now been proved by Jian--Song \cite{JS26} (see Theorem \ref{JSthm} in Section \ref{sec8}). It also implies the type-II bound \eqref{type2-scalar} after changing the constant. Thus the short-scale arguments below are conditional only on that weaker scalar bound, while their application to a general compact flow uses \cite{JS26}.

Another motivation for Gaussian type heat kernel estimates arises in the study of Ricci flow with bounded scalar curvature and its application to Fano K\"ahler-Ricci flow \cite{Bam18,Bam20,BamZh17,BamZh19,CW12,CW13,Wan12,TZZZZ,ZhQ12}. We will use a result obtained by Bamler in \cite{Bam20} which gives us upper bounds on heat kernels and its gradient for Ricci flows in terms of the pointed Nash entropy. Furthermore, combining our uniform $\nu$-entropy bound (Theorem \ref{ZhZh22}) with Gaussian type heat kernel bounds for Ricci flow with type-I scalar curvature bounds, we can provide an alternative proof of the Gaussian type bound and its gradient, which may be of independent interest and applications.

\subsection{Distance distortion, cut-off function}

The first main result of this paper is a local bound on the distortion of the distance function between two points. We remark that previously, distance distortion bounds have been obtained under certain growth conditions on the curvature tensor \cite{Ham93,Per02}, boundedness of parts of the Ricci curvature \cite{SiTop24,STZ19,TWan15,TZhZ21} and the bounded scalar curvature \cite{Bam18,Bam20,BamZh17,ZhQ12}.

Along a finite time solution of (\ref{KRF}) the scalar curvature is not bounded. On the scales $r\lesssim T-t$ this is precisely a bounded scalar curvature bound, and it is on these scales that we obtain the following results: distance distortion estimates (Theorem \ref{localdistortion}), a space-time cut-off function (Theorem \ref{cutoffftn}) and Gaussian type heat kernel bounds (Theorem \ref{hke1}). Under a type-I scalar curvature bound, the distance and cut-off estimates hold on the larger scales $r\lesssim\sqrt{T-t}$, and the heat kernel estimates hold for all $0\leq s<t<T$ (Theorem \ref{hke2}).

\begin{theorem}{\rm (Short-time distance distortion estimate)}\label{localdistortion}
Let $(X,\omega(t))$ be a K\"ahler-Ricci flow \eqref{KRF} defined on a compact K\"ahler manifold $(X,\omega_0)$ of complex dimension $n$ and $t\in[0,T),~T<\infty.$ Assume that the limiting class $[\beta]:=[\omega_0]-2\pi Tc_1(X)$ satisfies
$$\big([\omega_0]-2\pi Tc_1(X)\big)^n=\int_X\beta^n>0.$$ Then there are constants $0<\alpha<1,~c_0>0$ that only depend on $n,\omega_0$ such that the following holds:

Suppose that $t_0\in[0,T),~0<r_0\leq\min\{\sqrt{t_0},c_0(T-t_0)\},$ let $x_0,~y_0\in X$ be points with $d_{t_0}(x_0,y_0)\geq r_0$ and let $t\in[t_0-\alpha r_0^2,t_0+\alpha r_0^2].$ Then
$$\alpha d_{t_0}(x_0,y_0)<d_t(x_0,y_0)<\alpha^{-1}d_{t_0}(x_0,y_0).$$
If in addition the scalar curvature satisfies the type-I bound \eqref{typeIscalar} below, then the same conclusion holds for all $0<r_0\leq\min\{\sqrt{t_0},c_0\sqrt{T-t_0}\},$ with constants $\alpha,c_0$ depending on $C_0$ in {\rm Theorem \ref{hke2}}.
\end{theorem}

On the parabolic neighborhood $X\times[t_0-\alpha r_0^2,t_0+\alpha r_0^2]$ this bound gives $R\leq r_0^{-2}$ whenever $r_0\leq c_0(T-t_0)$ with $c_0$ sufficiently small, and the results of Bamler-Zhang \cite{BamZh17} apply on such neighborhoods after a parabolic rescaling. Under a type-I bound the natural scale is $r_0\lesssim\sqrt{T-t_0}.$ The rescaling proofs use the scalar upper bound and the finite-scale entropy lower bound. We stress that the proofs of Theorems \rm{\ref{localdistortion}}, \rm{\ref{cutoffftn}} and \rm{\ref{hke1}} below only use the scalar curvature lower bound \rm{(\ref{scalarlb2})}, the bound \rm{(\ref{type2-scalar})} and the finiteness of $\nu[g_0,2T]$. In particular they hold for every finite time solution of \rm{(\ref{KRF})}, whereas the noncollapsing assumption $\int_X\beta^n>0$ supplies the diameter and volume bounds \eqref{diameterbd}--\eqref{noncollapse2} used in the later geometric arguments.

\begin{corollary}{\rm(Long-time distance distortion estimate)}\label{globaldistortion}
Under the assumption of {\rm Theorem \ref{localdistortion}}, let $c_0$ be the constant of {\rm Theorem \ref{localdistortion}}. Then there is a constant $A<\infty$ that depends only on $n,\omega_0$ such that the following holds:

Suppose that $t_0\in[0,T)$, let $x_0,~y_0\in X$ be points with $d_{t_0}(x_0,y_0)=r_0>0$. For any $t\in[r_0^2,T)$ with $A^2(r_0^2+|t-t_0|)\leq\min\{t_0,c_0^2(T-\max\{t_0,t\})^2\}$ we have
$$d_t(x_0,y_0)<A\sqrt{r_0^2+|t-t_0|}.$$
\end{corollary}

\begin{theorem}{\rm(Cut-off function)}\label{cutoffftn}
Under the assumption of {\rm Theorem \ref{localdistortion}}, let $c_0$ be the constant of Theorem {\rm\ref{localdistortion}}. Then there is a constant $0<\rho<1$ that depends only on $n,\omega_0$ such that the following holds:

Let $(x_0,t_0)\in X\times[0,T)$ and $0<r_0\leq\min\{\sqrt{t_0},c_0(T-t_0)\}$ and let $0<\tau\leq\rho^2r^2_0.$ Then there is a function $\phi\in C^{\infty}(X\times[t_0-\tau,t_0])$ with the following properties:
\begin{itemize}
\item[(a)] $0\leq\phi<1$ everywhere;
\item[(b)] $\phi>\rho$ on $P(x_0,t_0,\rho r_0,-\tau);$
\item[(c)] $\phi=0$ on $(X\setminus B(x_0,t_0,r_0))\times[t_0-\tau,t_0];$
\item[(d)] $|\nabla\phi|<r_0^{-1}$ and $|\frac{\partial}{\partial t}\phi|+|\Delta\phi|<r^{-2}_0$ everywhere.
\end{itemize}
Under the type-I bound \eqref{typeIscalar} the same holds for all $$0<r_0\leq\min\{\sqrt{t_0},c_0\sqrt{T-t_0}\},$$
with $\rho$ depending on $c_0$ in {\rm Theorem \ref{localdistortion}}.
\end{theorem}

The proofs of Theorem \ref{localdistortion}, Corollary \ref{globaldistortion} and Theorem \ref{cutoffftn} can be found in Section \ref{sec3}. They are obtained from the corresponding results of Bamler-Zhang \cite{BamZh17} by a parabolic rescaling, using the type-II bound (\ref{type2-scalar}).

Let us now proceed to the second part of this paper.

\subsection{Heat kernel and its gradient}

We will derive Gaussian type heat kernel estimates (\ref{gthub}) for K\"ahler-Ricci flow (\ref{KRF}) which are similar to the results obtained by Bamler-Zhang \cite{BamZh17} under the assumption that the scalar curvature is (locally) bounded. We analyze the heat kernel $H(x,t;y,s)$ for the heat equation coupled with K\"ahler-Ricci flow and establish Gaussian type upper bounds and its gradient. Over the last decades, similar questions have been subject to active research, especially after Li-Yau's celebrated paper \cite{LiYau86}. Among numerous useful papers, let us mention \cite{Bam20,BamZh17,CaoNi08,CaoHam09,CaoZh11,ChHam97,Dav90,Gri97,Gri09,GriHu08,Per02,ZhQ06} and the references therein. Previous bounds usually relied on boundedness assumptions of the (scalar, Ricci, holomorphic bisectional) curvature and distance functions. Inspired by Bamler-Zhang \cite{BamZh17}, our proof makes use of an integral bound for heat kernel obtained by Hein-Naber (Theorem 1.13 and (1.20) of \cite{HNab14}) as well as our distance distortion bounds (Theorem \ref{localdistortion}), the cut-off function (Theorem \ref{cutoffftn}) and mean value inequalities for the conjugate heat equation (Section \ref{sec4}). Since Theorems \ref{localdistortion} and \ref{cutoffftn} are available on the scales $r_0\lesssim T-t_0$, the resulting heat kernel bounds are uniform in the range $\sqrt{t-s}\lesssim T-t.$ More precisely, we have the following. For more details see Sections \ref{sec5} and \ref{sec6}.

\begin{theorem}{\rm(Heat Kernel Estimates)}\label{hke1}
Under the assumption of {\rm Theorem \ref{localdistortion}}, there are constants $c_1>0$ and $C_1,~C_2<\infty$ that only depend on $n,\omega_0$ such that for all $0\leq s<t<T$ with
\begin{equation}\label{scalerange}
t-s\leq c_1(T-t)^2
\end{equation}
we have
\begin{align}\label{gthub}
\frac{1}{C_1(t-s)^n}\exp(-\frac{C_2d_s^2(x,y)}{t-s})\leq H(x,t;y,s)&\leq\frac{C_1}{(t-s)^n}\exp(-\frac{d_s^2(x,y)}{C_2(t-s)})\\ \label{tiduub}
|\nabla_xH(x,t;y,s)|_{\omega(t)}&\leq\frac{C_1}{(t-s)^{n+\frac{1}{2}}}\exp(-\frac{d_s^2(x,y)}{C_2(t-s)})\\ \label{gradientup1}
\frac{|\nabla_xH(x,t;y,s)|_{\omega(t)}}{H(x,t;y,s)}&\leq\frac{C_2}{(t-s)^{1/2}}\Big(1+\frac{d_s(x,y)}{(t-s)^{1/2}}\Big).
\end{align}
\end{theorem}

We remark that, for \(t-s\ge\varepsilon_0>0\), all estimates in Theorem \ref{hke1} hold with constants depending on \(\varepsilon_0\) (see Remark \ref{bddgeom}). Consequently, the content of Theorem \ref{hke1} is the uniformity of the constants within (\ref{scalerange}) as \(t\to T\). This range of scales is precisely the one in which the cut-off function of Theorem \ref{cutoffftn} is available at scale \(\sqrt{t-s}\) around \((x,t)\); see the discussion following Theorem \ref{localdistortion}.

For Ricci flows with type-I scalar curvature bounds, Hallgren \cite{Hall23} (see \cite[Theorem 1.4]{BamZh17} and \cite[Theorem 2.1]{CaoZh11}) proved the following Gaussian type heat kernel estimates
\begin{align}\label{Gaussuplow}
\frac{1}{C_1(t-s)^n}\exp(-\frac{C_2d_s^2(x,y)}{t-s})\leq H(x,t;y,s)&\leq\frac{C_1}{(t-s)^n}\exp(-\frac{d_s^2(x,y)}{C_2(t-s)}),
\end{align}
for $\frac{3T}{4}\leq s<t<T,$ where $C_1,C_2$ are uniform constants which depend only on $n,\nu[g_0,2T]$ and the type-I constant (see Lemma \ref{type1-hk-ublb} for the precise statement). In particular, combining (\ref{Gaussuplow}) with Bamler's estimates for heat kernel \cite[Theorem 7.5]{Bam20} or (\ref{graheatuppernash}), we have
\begin{align}\label{gradientup}
\frac{|\nabla_xH(x,t;y,s)|_{\omega(t)}}{H(x,t;y,s)}&\leq\frac{C_2}{(t-s)^{1/2}}\Big(1+\frac{d_s(x,y)}{(t-s)^{1/2}}\Big).
\end{align}
The factor $1+\frac{d_s(x,y)}{\sqrt{t-s}}$ on the right hand side of (\ref{gradientup}) cannot be omitted, as one already sees on the Euclidean heat kernel, for which $\frac{|\nabla_xH|}{H}=\frac{|x-y|}{2(t-s)}.$ 

In Section \ref{sec6}, we will give a second proof of the Gaussian type upper bound in (\ref{gthub}), which uses neither the Hein-Naber inequality nor the mean value inequalities of Section \ref{sec4}, and which may be of independent interest. It relies on the integral lower bound of Lemma \ref{hklb} for the conjugate heat kernel measure and on Bamler's upper bounds for the heat kernel and its gradient in terms of the pointed Nash entropy \cite[Theorem 7.1, Theorem 7.2, Theorem 7.5]{Bam20}. The Gaussian lower bound in (\ref{gthub}) depends on Perelman's Harnack inequality, and the bound (\ref{gradientup1}) on the gradient then follows from Bamler's gradient estimate. The restriction (\ref{scalerange}) in Theorem \ref{hke1} comes only from the range of scales on which Theorems \ref{localdistortion} and \ref{cutoffftn} are available. Under a type-I scalar curvature bound this restriction disappears: the estimates of Bamler-Zhang are then available on the scale $\sqrt{T-t},$ and on larger scales one can use the recent work of Hallgren \cite[Lemma 3.1]{Hall23} and the semigroup property of the heat kernel. We obtain the following.

\begin{theorem}{\rm(Heat Kernel Estimates II)}\label{hke2}
Under the assumption of {\rm Theorem \ref{localdistortion}}, assume in addition that the scalar curvature satisfies the type-I bound
\begin{equation}\label{typeIscalar}
R(x,t)\leq\frac{C_0}{T-t}~~~~\text{on}~~X\times[0,T)
\end{equation}
for some $C_0<\infty.$ Then there are constants $C_1,~C_2$ that only depend on $\omega_0,~n,~C_0$ such that for any $0\leq s<t<T$, we have
\begin{align}\label{Gaussuplow2}
\frac{1}{C_1(t-s)^n}\exp(-\frac{C_2d_s^2(x,y)}{t-s})\leq H(x,t;y,s)&\leq\frac{C_1}{(t-s)^n}\exp(-\frac{d_s^2(x,y)}{C_2(t-s)})\\\label{gradientup2}
\frac{|\nabla_xH(x,t;y,s)|_{\omega(t)}}{H(x,t;y,s)}&\leq\frac{C_2}{(t-s)^{1/2}}\Big(1+\frac{d_s(x,y)}{(t-s)^{1/2}}\Big).
\end{align}
\end{theorem}

\begin{remark}\label{typeIremark}
Let us explain the role of the two scalar curvature bounds. The Gaussian lower bound depends on Perelman's estimate \eqref{heatkelb2}, which requires a control of the quantity $$\frac{1}{\sqrt{t-s}}\int_s^t\sqrt{t-\tau}R(y,\tau)d\tau.$$
The type-II bound \eqref{type2-scalar} controls this quantity on \eqref{scalerange}, whereas the type-I bound \eqref{typeIscalar} controls it for all $s<t$; see \eqref{typeI-Lintegral}. Likewise, the local distance and cutoff estimates use scales $r_0\lesssim T-t_0$ under type II and $r_0\lesssim\sqrt{T-t_0}$ under type I. Upon completing the manuscript, the authors discovered that Jian-Song \cite{JS26} had already proved that the type-I scalar curvature bound \eqref{typeIscalar} holds in general. \qed 
\end{remark}

Let us now proceed to the third part of this paper.

\subsection{Distance distortion II and the structure of GH limit}

Bamler--Zhang \cite{BamZh19} studied the behavior of the distance function for Ricci flows with uniformly bounded scalar curvature. They showed that, on short time intervals, the distance function is \(\frac{1}{2}\)-H\"older continuous in a uniform way. Hence it can be extended continuously up to the singular time.

In Section \ref{sec7} we present the consequences of Theorem \ref{hke1} for the distance function, in the spirit of Theorem 1.1 in \cite{BamZh19}.

\begin{theorem}
Under the assumption of {\rm Theorem \ref{localdistortion}}, there is a uniform constant $C$ such that
\begin{equation}\label{impdde}
d_s(x,y)-C\sqrt{t-s}\leq d_t(x,y)
\end{equation}
for all $x,y\in X$ and $0\leq s<t<T$ satisfying \eqref{scalerange}. Under \eqref{typeIscalar}, this lower estimate holds for every $0\leq s<t<T$. For each fixed $s<T$, there is $C(s)<\infty$ such that
\begin{equation}\label{impdde-upper}
d_t(x,y)\leq d_s(x,y)+C(s)\sqrt{t-s},\qquad s<t\leq\frac{T+s}{2}.
\end{equation}
\end{theorem}

The lower estimate follows from the Gaussian upper bound (\ref{gthub}) and Bamler's volume bound valid within (\ref{scalerange}) and, under condition (\ref{typeIscalar}), for all $0\leq s<t<T.$ For the right hand inequality we only obtain a local statement, with a constant $C(s)$ which is not asserted to remain bounded as $s\to T.$ Motivated by Hallgren--Jian--Song--Tian \cite{HJST23}, the heat-equation approach to a uniform constant would require a gradient estimate for the backward heat equation, with constants uniform up to the singular time. If \eqref{impdde-upper} holds with constants uniform in \(s\), then the argument of Bamler-Zhang \cite{Bam18,BamZh19} can be used to analyze the structure of the limit space. The uniform convergence in Section \ref{sec8} will instead follow from localization, Wasserstein monotonicity, and packing.

We also combine Theorem \ref{hke2} with Bamler's theory \cite{Bam20,Bam20b,Bam20c}, with Hallgren \cite{Hall23}, and with Hallgren--Zhang \cite{HZ26} in order to describe the singular time slice of the flow and its scalar type I blow-up limits. We obtain the following statements.

\begin{theorem}
Under the assumption of {\rm Theorem \ref{localdistortion}}, we have
\begin{itemize}
\item[(i)] As \(t\to T\), the distance functions \(d_t\) converge uniformly on \(X\times X\) to a pseudo-distance \(d_T\). The Gromov--Hausdorff limit of \((X,d_t)\) is unique. It equals the metric identification \((X/\!\sim,d_T)\), and is isometric to both the asymptotic boundary \(\mathcal{M}_T\) of Bamler {\rm \cite[Section 2.6]{Bam20c}} and the time-zero slice \(Z_0\) of Li--Zhang {\rm \cite{LZ26} (Theorem \ref{thmB} in Section \ref{sec8})}.
\item[(ii)] The regular part of the limit is the K\"ahler manifold \((X\setminus\mathrm{Null}(\alpha),\omega_T)\) of Collins--Tosatti {\rm\cite{CT15}}, with \([\omega_T]=\alpha:=[\beta]\). Under assumption \eqref{omega0-omegay}, the map \(\pi\) descends to a Lipschitz surjection from the limit onto \((Y,d_{\omega_Y})\). It is a homeomorphism if and only if the fibres of \(\pi\) collapse {\rm(Proposition \ref{propY} in Section \ref{sec8})}.
\item[(iii)] At a fixed base point, the Type I blow-up limits \((X,(T-t_i)^{-1}g(t_i),x)\) are time slices of tangent flows in the sense of Bamler, which are singular shrinking K\"ahler--Ricci solitons, and carry the structure of normal complex analytic varieties with log terminal singularities and a polarized Fano fibration {\rm (Theorem \ref{thmA} in Section \ref{sec8})}.
\end{itemize}
\end{theorem}

The arguments differ from those announced in \cite[Section 1]{JS26} for the sequel of that paper (Ricci potentials and partial $C^0$-estimates, in the spirit of \cite{JST23,HJST23}), and from those of Li--Zhang \cite{LZ26} (the spacetime completion of Fang--Li \cite{FL25} and curvature radius estimates). The only ingredient taken from \cite{JS26} is the scalar curvature bound (\ref{typeIscalar}). Instead, the heat kernel estimates of Theorem \ref{hke2}, together with the monotonicity of the $W_1$-distance of conjugate heat flows \cite[Lemma 2.7]{Bam20}, replace the partial $C^0$-estimates and the curvature radius estimates, respectively. The soliton property of the blow-up limits follows from the monotonicity of the Nash entropy \cite[Theorem 2.37]{Bam20c}, while the complex analytic structure comes from Hallgren-Zhang \cite{HZ26}.

{\bf Declaration on the use of AI.} AI tools were used in developing the argument in Section \ref{sec8}. AI tools were also used for English-language editing. The authors take full responsibility for all mathematical claims and for any errors that remain.

\section{Preliminary}\label{preliminary}
Let $(M,g(t))$ be a Ricci flow defined on a compact $m$-manifold $M$ and $t\in[0,T),~T\leq\infty.$ For any $(x_0,t_0)\in M\times[0,T),r>0,$ we will frequently use the notation
\begin{align*}
Q^+(x_0,t_0,r):&=\{(x,t)\in M\times[0,T)~~:~d_t(x_0,x)<r,~~t_0\leq t\leq t_0+r^2\}\\
Q^-(x_0,t_0,r):&=\{(x,t)\in M\times[0,T)~~:~d_t(x_0,x)<r,~~t_0-r^2\leq t\leq t_0\}
\end{align*}
to denote the ``forward'' and ``backward'' parabolic cubes. For any $(x,t)\in M\times[0,T),~r\geq0$ and $\triangle t\in \mathbb{R}$ we denote by
$$P(x,t,r,\triangle t):=B(x,t,r)\times[\min\{t,t+\triangle t\},\max\{t,t+\triangle t\}]$$
the parabolic neighborhood around $(x,t)$, whenever this time interval is contained in $[0,T)$.

{\bf Conventions.} Throughout the paper $m=2n$. For a K\"ahler form $\omega$, let $g_\omega$ denote the usual associated Riemannian metric, for which $\Delta_{g_\omega}=2\Delta_\omega$ and $R_{g_\omega}=2R_\omega$. We use the fixed normalization
$$g(t):=2g_{\omega(t)},\qquad dg_t=2^n\frac{\omega(t)^n}{n!}.$$
Then the original time variable satisfies
$$\partial_tg=-2\Ric_g,\qquad \Delta_g=\Delta_\omega,
\qquad R_g=R_\omega=\tr_{\omega(t)}\Ric(\omega(t)).$$
The symbols $d_t$, $dg_t$, and $H$ refer to this Riemannian metric, its volume, and its heat kernel; $\delta_y$ in \eqref{heatkernel} is interpreted with this volume. For scalar functions, the complex gradient norm $|\partial u|_\omega$ agrees with $|du|_g$. Distances denoted by $d_\omega$, including $d_{\omega_Y}$ and $d_{\omega_T}$, use the same normalization. The usual Riemannian distances differ by the fixed factor $\sqrt2$. In particular all Perelman functionals, Nash entropies, and the shrinker equation use the standard real Ricci-flow convention, with no implicit time change. The noncollapsed finite-time singular setting has $n\geq2$. In complex dimension one, a nef class of positive volume is K\"ahler and cannot be the limiting boundary class.

Note that the evolution equation for the scalar curvature reads
$$\frac{\partial}{\partial t}R=\Delta R+2|Ric|^2\geq\Delta R+\frac{2}{m}R^2.$$
Applying the maximum principle, we can deduce the following lower bound on the scalar curvature: if $R(\cdot,t_0)\geq R_{min}$ with $R_{min}\leq0$, then for all $t\geq t_0,$ we have
\begin{equation}\label{scalarlb}
R(\cdot,t)\geq\frac{m}{2}\frac{R_{min}}{\frac{m}{2}-R_{min}(t-t_0)}.
\end{equation}
In particular, if $0\leq t_0<t<T,$ then
\begin{equation}\label{scalarlb1}
R(\cdot,t)\geq-\frac{m}{2(t-t_0)}
\end{equation}
In the case of the K\"ahler-Ricci flow (\ref{KRF}) with $T<\infty$ the bound (\ref{scalarlb}) can be improved to a bound which is uniform on the whole of $X\times[0,T),$ and which will be used repeatedly below.
\begin{lemma}\label{scalarlbunif}
Let $(X,\omega(t))$ be a K\"ahler-Ricci flow \rm{(\ref{KRF})} with $t\in[0,T),$ $T<\infty.$ Then
\begin{equation}\label{scalarlb2}
R(\cdot,t)\geq-R_0:=\min\{0,\min_XR(\cdot,0)\}~~~~\text{on}~~X\times[0,T).
\end{equation}
In particular $R_0=R_0(\omega_0),$ and $dg_t\leq e^{R_0(t-s)}dg_s$ for $0\leq s\leq t<T.$
\end{lemma}

\begin{proof}
Apply (\ref{scalarlb}) with $t_0=0$ and $R_{min}=\min\{0,\min_XR(\cdot,0)\}.$ If $R_{min}=0$ there is nothing to prove. If $R_{min}<0,$ then $\frac{m}{2}-R_{min}t\geq\frac{m}{2}$ for $t\geq0,$ hence
$$R(\cdot,t)\geq\frac{m}{2}\frac{R_{min}}{\frac{m}{2}-R_{min}t}\geq\frac{\frac{m}{2}R_{min}}{\frac{m}{2}}=R_{min}.$$
The last assertion follows from $\frac{\partial}{\partial t}dg_t=-Rdg_t.$
\end{proof}

\begin{remark}\label{bddgeom}
Since $T$ is the maximal existence time of \rm{(\ref{KRF})}, the flow is smooth with uniformly bounded curvature on $X\times[0,T']$ for every $T'<T,$ and the metrics $\omega(t),~t\in[0,T'],$ are uniformly equivalent to $\omega_0.$ On such a compact time interval all the estimates below (distance distortion, cut-off functions, two sided Gaussian bounds for the heat kernel and upper bounds for its gradient) hold by the standard theory of the heat kernel of a smooth family of metrics with bounded geometry (see e.g. \cite[Chapter 26]{CCGGIIKLLN10}; alternatively one may apply \cite{BamZh17} with $R_0:=\sup_{X\times[0,T']}|R|$), with constants depending on $n,\omega_0$ and $T'.$ We shall use this with $T'=\frac{T}{2}$ and $T'=\frac{7T}{8}$. Since $T$ is determined by $\omega_0,$ the resulting constants depend only on $n,\omega_0.$ Moreover, by the diameter bound \rm{(\ref{diameterbd})}, if $t-s\geq\varepsilon_0$ for some fixed $\varepsilon_0>0,$ then the Gaussian factors $\exp\big(\pm\frac{d^2_{\bullet}(x,y)}{C(t-s)}\big)$ occurring below are bounded from above and from below by positive constants depending only on $\varepsilon_0,n,\omega_0$. For completeness, if $t-s\geq\varepsilon_0$, choose $$b=s+\varepsilon_0/2\leq T-\varepsilon_0/2.$$
Smooth bounded geometry gives
$$H(z,b;y,s)\geq c(\varepsilon_0)>0.$$
The reproduction formula and $$\int_XH(x,t;z,b)dg_b=1,$$
then give
$$H(x,t;y,s)\geq c(\varepsilon_0).$$
The on-diagonal upper bound and the forward gradient estimate applied on the later half of $[s,t]$ give the corresponding upper and gradient bounds. The bounded Gaussian factors can then be inserted by changing constants. We shall use these two observations without further comment, in order to reduce all the statements below to the range in which $t-s$ is small compared with $\min\{s,T-s\}.$ \qed
\end{remark}

Recall Perelman's $\mathcal{W}$-functional (cf. \cite{Per02}):
$$\mathcal{W}[g,f,\tau]:=\int_M\big[\tau(|\nabla f|^2+R)+f-m\big](4\pi\tau)^{-\frac{m}{2}}e^{-f}dg.$$
Here $g$ is a Riemannian metric, $f\in C^1(M)$ and $\tau>0.$ Recall that the derived functionals
$$\mu[g,\tau]:=\inf\{\mathcal{W}[g,f,\tau]~:~f\in C^1(M)~\text{satisfying}~\int_M(4\pi\tau)^{-\frac{m}{2}}e^{-f}dg=1~\}$$
and
$$\nu[g,\tau]:=\inf_{0<\tau'\leq\tau}\mu[g,\tau'],~~~~\nu[g]:=\inf_{\tau>0}\nu[g,\tau]$$
satisfy $\nu[g,\tau],\nu[g]\leq0.$ Along the Ricci flow, the quantities
$$\mu[g(t),\tau-t],\qquad\nu[g(t),\tau-t],\qquad\nu[g(t)]$$
are nondecreasing wherever they are defined, with $\tau-t>0$ for the first two quantities.

Next, we discuss heat equations coupled with the Ricci flow. The forward heat equation
\begin{equation}\label{heateqa}
  \begin{cases}
  \frac{\partial}{\partial t}g(x,t)&=-2Ric(g(x,t))\\
  \frac{\partial}{\partial t}u(x,t)&=\Delta_{x,t}u(x,t)
  \end{cases}
\end{equation}
has a conjugate,
\begin{equation}\label{cheateqa}
  \begin{cases}
  \frac{\partial}{\partial t}g(x,t)&=-2Ric(g(x,t))\\
  \frac{\partial}{\partial t}u(x,t)&=-\Delta_{x,t}u(x,t)+R(x,t)u(x,t)
  \end{cases}
\end{equation}
which evolves backwards in time. The following differential Harnack inequality for a positive solution $u\in C^{\infty}(M\times(0,T))$ of the heat equation (\ref{heateqa}) is obtained in \cite[Theorem 3.3]{ZhQ06} or \cite[Theorem 5.1]{CaoHam09}:
\begin{equation}\label{diffhar1}
\frac{|\nabla u(x,t)|}{u(x,t)}\leq\sqrt{\frac{1}{t}}\sqrt{\log\frac{J}{u(x,t)}}
\end{equation}
for $(x,t)\in M\times(0,T],$ where $J:=\sup_{M\times(0,t]}u$. This bound does not assume an upper bound on the scalar curvature. For the heat kernel we apply it after shifting the initial time to a positive distance from its pole, so that $J<\infty$. In particular, the following interpolation inequality holds for any $\delta>0$, $x,~y\in M$ and $0<t\leq T:$
\begin{equation}\label{diffhar2}
u(y,t)\leq c_1u(x,t)^{1/(1+\delta)}J^{\delta/(1+\delta)}\exp(c_2d_t(x,y)^2/t),
\end{equation}
where $c_1,c_2$ are positive constants depending only on $\delta.$

For any two compactly supported functions $u,v\in C^{\infty}(M\times(0,T))$ we have
$$\int_0^T\int_M\left(\frac{\partial}{\partial t}-\Delta_{x,t}\right)u(x,t)\cdot vdg_t(x)dt=\int_0^T\int_Mu\cdot\left(-\frac{\partial}{\partial t}-\Delta_{x,t}+R(x,t)\right)v(x,t)dg_t(x)dt.$$
We will denote by $H(x,t;y,s),$ where $x,y\in M,0\leq s<t<T,$ the heat kernel associated with the heat equation (\ref{heateqa}), meaning that for any fixed $(y,s)\in M\times[0,T)$
\begin{equation}\label{heatkernel}
\left(\frac{\partial}{\partial t}-\Delta_{x,t}\right)H(x,t;y,s)=0~~\text{and}~~\lim_{t\searrow s}H(x,t;y,s)=\delta_y(x).
\end{equation}
Then $H(x,t;y,s)$ is the kernel associated to the conjugate heat equation (\ref{cheateqa}), meaning that for any fixed $(x,t)\in M\times[0,T)$
\begin{equation}\label{cheatkernel}
\left(-\frac{\partial}{\partial s}-\Delta_{y,s}+R(y,s)\right)H(x,t;y,s)=0~~\text{and}~~\lim_{s\nearrow t}H(x,t;y,s)=\delta_x(y).
\end{equation}
Lastly, we recall two heat kernel bounds from \cite{ZhQ12}; see also \cite[Theorems 7.1 and 7.5]{Bam20} and \cite[Theorem 2.1]{CaoZh11}. The first is a global on-diagonal upper bound. Assume that $x,y\in M$, $0\leq s<t<T$, and $R\geq-R_0$ on $M\times[s,t]$. Then
\begin{equation}\label{ondiagheatub}
H(x,t;y,s)\leq \frac{C}{(t-s)^{\frac{m}{2}}},
\end{equation}
where $C$ only depends on $\nu[g_0,2T],~n,$ and $(t-s)R_0.$

The second estimate is a distance dependent lower bound on the heat kernel. This bound is a consequence of (\ref{diffhar1}),(\ref{diffhar2}) and Perelman's Harnack inequality. Let $x,y\in M,$ $0\leq s<t<T$ and consider a smooth curve $\gamma:[s,t]\to M$ between $(y,s)$ and $(x,t),$ i.e. $\gamma(s)=y$ and $\gamma(t)=x.$ Its $\mathcal{L}$-length is defined as
$$\mathcal{L}(\gamma):=\int_s^t\sqrt{t-\tau}(|\gamma'(\tau)|^2_{\tau}+R(\gamma(\tau),\tau))d\tau.$$
The reduced distance between $(x,t)$ and $(y,s)$ is defined as
$$\ell_{x,t}(y,s):=\frac{1}{2\sqrt{t-s}}\inf\left\{\mathcal{L}(\gamma)|\gamma:[s,t]\to M~\text{between}~(y,s)~\text{and}~(x,t)\right\}.$$
Then by \cite[Corollary 9.5]{Per02}, we have
\begin{equation}\label{heatkelb1}
H(x,t;y,s)\geq\frac{1}{(4\pi(t-s))^{\frac{m}{2}}}\exp\Big(-\ell_{x,t}(y,s)\Big).
\end{equation}
Applying (\ref{heatkelb1}) for $x=y$ at time $t$ yields
\begin{equation}\label{heatkelb2}
H(x,t;x,s)\geq\frac{1}{(4\pi(t-s))^{\frac{m}{2}}}\exp\left(-\frac{1}{2\sqrt{t-s}}\int_s^t\sqrt{t-\tau}R(x,\tau)d\tau\right).
\end{equation}
Put $a=(s+t)/2$. On the later interval $[a,t]$, the on-diagonal bound gives
$$J:=\sup_{M\times[a,t]}H(\cdot,\cdot;y,s)\leq C(t-s)^{-m/2}.$$
Apply \eqref{diffhar2}, with $\delta=1$, to this interval, whose elapsed time is $(t-s)/2$, and combine it with \eqref{heatkelb2} at $y$. After enlarging the constants, this yields
\begin{equation}\label{heatke3}
H(x,t;y,s)\geq\frac{c_3}{(t-s)^{m/2}}
\exp\left(-C\frac{d_t(x,y)^2}{t-s}
-\frac{1}{\sqrt{t-s}}\int_s^t\sqrt{t-\tau}R(y,\tau)d\tau\right),
\end{equation}
which is our desired lower bound, $c_3$ only depends on $\nu[g_0,2T],~n,$ and $(t-s)R_0.$ We point out that the scalar curvature integral in (\ref{heatke3}) is taken along the worldline of the pole $(y,s)$ and not along that of $x;$ the heat kernel is not symmetric under the exchange of its two spatial variables at different times. In particular, if the scalar curvature satisfies $R\leq\frac{C}{T-t}$ on $X\times[0,T),$ then
$$\frac{1}{\sqrt{t-s}}\int_s^t\frac{\sqrt{t-\tau}}{T-\tau}\,d\tau\leq2.$$
Thus, by (\ref{heatke3}), after changing the constants, which now also depend on the type-I constant, we have
\begin{equation}\label{heatke4}
H(x,t;y,s)\geq\frac{c_3}{(t-s)^{\frac{m}{2}}}\exp\Big(-2c_2\frac{d_t(x,y)^2}{t-s}\Big)
\end{equation}
for any $x,y\in X$ with $[s,t]\subset[0,T).$

Recently, Guo-Phong-Song-Sturm in \cite{GPSS24-1,GPSS23} made significant progress in studying the geometry of K\"ahler metrics (see also \cite{ZhZh26+}). They proved the following estimates, which will be used frequently in the later sections.
\begin{theorem}{\rm (Theorem 2.1 and Section 8 in \cite{GPSS24-1})}\label{GPSS22}
Under the assumption of {\rm Theorem \ref{localdistortion}}, there exist uniform constants $\theta=\theta(n)>2n,$ $C_1=C_1(X,\omega_0,n)$ and $C_2=C_2(X,\omega_0,n,\theta)$ such that
\begin{align}\label{diameterbd}
diam(X,\omega(t))&\leq C_1\\\label{noncollapse2}
\frac{|B(x,r,t)|_t}{[\omega(t)^n]}&\geq C_2r^{\theta}
\end{align}
for any $x\in X$ and $r\in(0,1].$
\end{theorem}


Only the finiteness of the exponent $\theta$ is needed in the later packing argument.
The Sobolev inequality is a powerful tool that conveys deep analytical and geometric information, such as non-collapsing and isoperimetric estimates. In this context, we strengthen the monotonicity of Perelman's W-entropy to derive a uniform Sobolev inequality along the K\"ahler-Ricci flow. 

\begin{theorem}{\rm (Uniform Entropy Estimates)}\label{ZhZh22}
Under the assumption of {\rm Theorem \ref {localdistortion}}, along the K\"ahler-Ricci flow \eqref{KRF} we have the uniform entropy estimate
\begin{align}\label{u-lowerbd}
\mu(g(t),\tau)\geq\log vol_{g(t)}(X,g(t))-C(n,\omega_0),~\forall~0<\tau\leq1,
\end{align}
and the uniform Sobolev inequality
\begin{align}\label{Soboleviq}
\Big(\int_X|f|^{\frac{2n}{n-1}}dg_t\Big)^{\frac{n-1}{n}}\leq C_S\cdot\int_X\Big(|\nabla f|^2+(R(t)+R_0+1)f^2\Big)dg_t,~~f\in C^{\infty}(X,\mathbb{R})
\end{align}
where $R_0$ is the constant of {\rm Lemma \ref{scalarlbunif}} and $C_S$ is a uniform constant which depends only on $\omega_0,n.$
\end{theorem}

Note that $R+R_0+1\geq1$ on $X\times[0,T),$ and that \eqref{Soboleviq} is invariant under the parabolic rescaling $\omega\mapsto r^{-2}\omega$ for $r\leq1$ (the zero order coefficient $R_0+1$ being replaced by $r^2(R_0+1)$).

\begin{proof}
{\bf Step 1:} Since $T<\infty$ and $\omega_0$ is a fixed K\"ahler metric on the compact manifold $X,$ Perelman's $\nu$-functional satisfies $\nu[g_0,T+1]=:-C_0>-\infty.$ By the monotonicity of the $\mu$-functional along the Ricci flow we have, for every $t\in[0,T)$ and every $0<\tau\leq1,$
$$\mu(g(t),\tau)\geq\mu(g_0,\tau+t)\geq\nu[g_0,T+1]=-C_0,$$
since $\tau+t\leq T+1.$ Note that this step does not use the noncollapsing assumption. On the other hand, we have
$$vol_{g(t)}(X,g(t))=\frac{2^n}{n!}[\omega(t)]^n=\frac{2^n}{n!}\big([\omega_0]-2\pi tc_1(X)\big)^n\leq C_1$$
and the desired estimate (\ref{u-lowerbd}) follows. We stress that under the assumption of Theorem \ref{localdistortion} the volume is in addition bounded from below: the polynomial $t\mapsto\big([\omega_0]-2\pi tc_1(X)\big)^n$ is positive on $[0,T),$ being the volume of a K\"ahler class, and positive at $t=T$ by assumption. Hence by continuity, we have
$$vol_{g(t)}(X,g(t))\geq\frac{2^n}{n!}\min_{0\leq t\leq T}\big([\omega_0]-2\pi tc_1(X)\big)^n>0.$$
Hence, (\ref{u-lowerbd}) is in fact equivalent to the uniform entropy bound $\mu(g(t),\tau)\geq-C$ for $0<\tau\leq1.$

{\bf Step 2:} By the classical argument (see \cite[Theorem 6.2.1]{ZhQ16}), the logarithmic Sobolev inequalities $\mu(g(t),\tau)\geq-C,~0<\tau\leq1,$ imply a Sobolev inequality of the form
$$\big(\int_X|f|^{\frac{2m}{m-2}}dg_t\big)^{\frac{m-2}{m}}\leq A\int_X\big(4|\nabla f|^2+Rf^2\big)dg_t+B\int_Xf^2dg_t$$
with constants $A,B$ depending only on $C$ and $m,$ where the Laplacian and the scalar curvature are those of the Ricci flow (\ref{heateqa}). Since $R+R_0\geq0$ by Lemma \ref{scalarlbunif}, then the right hand side is bounded by
$$\max\{4A,A,B\}\int_X\big(|\nabla f|^2+(R+R_0+1)f^2\big)dg_t,$$
and (\ref{Soboleviq}) follows after translating into the K\"ahler conventions ($m=2n,$ so that $\frac{2m}{m-2}=\frac{2n}{n-1}$). 
\end{proof}
\begin{remark}\label{nashentropyremark}
Step 1 above uses in an essential way that $T<\infty.$ For a long time solution of the K\"ahler-Ricci flow, when the canonical bundle is semiample or nef, the quantity $\nu[g_0,T+1]$ is of no use and one has to argue differently \cite{ZhZhl22}. Let us only indicate the identities which are relevant in that case, in the notation of \cite{Bam20}. For a conjugate heat kernel measure
$$d\nu_{x,t;t-\tau}=(4\pi\tau)^{-n}e^{-f}dg_{t-\tau}$$ based at $(x,t),$
the pointed Nash entropy
$$\mathcal{N}_{x,t}(\tau)=\int_Xf\,d\nu_{x,t;t-\tau}-n$$
and the entropy
$$\mathcal{W}_{x,t}(\tau)=\int_X\big[\tau(R+|\nabla f|^2)+f-2n\big]d\nu_{x,t;t-\tau}$$ satisfy, by \cite[Proposition 5.2]{Bam20} (see also \cite{CMZ22}),
$$\frac{d}{d\tau}\mathcal{W}_{x,t}(\tau)\leq0,~~~~\frac{d}{d\tau}\big(\tau\mathcal{N}_{x,t}(\tau)\big)=\mathcal{W}_{x,t}(\tau),~~~~\text{i.e.}~~\frac{d}{d\tau}\mathcal{N}_{x,t}(\tau)=\frac{\mathcal{W}_{x,t}(\tau)-\mathcal{N}_{x,t}(\tau)}{\tau},$$
together with
$$\mathcal{N}_{x,t}(0)=0,~~\mathcal{N}_{x,t}(\tau)\leq0~ \text{and}~\mathcal{N}_{x,t}(\tau)\geq\mu[g_{t-\tau},\tau].$$
Combining these identities with Bamler's gradient estimate for the pointed
Nash entropy \cite[Corollary 5.11]{Bam20}, Bamler's upper volume bound for
distance balls \cite[Theorem 8.1]{Bam20}, and the diameter bound
\eqref{diameterbd}, we can bound $\mu$ from below in terms of the volume in
the long-time case; we refer to \cite[Theorem 3.4]{ZhZhl22} for the details, which will not be needed here. \qed
\end{remark}

We end this section with the following ultracontractivity property of positive solutions to the heat equation (\ref{heateqa}) which is obtained in \cite{ZhZhl22}.
\begin{theorem}{\rm(Theorem 3.6 in \cite{ZhZhl22})}\label{ZhZl21}
Let $q>1$ be any constant. Let $(M,g(t))$ be a solution of Ricci flow on a compact Riemannian manifold of dimension $m$ on $M\times[0,t_0].$ Assume that
\begin{equation}\label{mu0lb}
\mu(g(0),\tau)\geq-A,~~\forall~0<\tau\leq t_0
\end{equation}
for some $A\geq0,$ and that $R\geq-R_0$ on $M\times[0,t_0]$ for some $R_0\geq0.$ Then for any nonnegative solution $u(t)$ of heat equation \eqref{heateqa},
\begin{equation}
||u(t_0)||_{L^{\infty}}\leq C(m)\cdot e^{A+R_0t_0}\cdot t_0^{-\frac{m}{2q}}\cdot||u(0)||_{L^q(g(0))}.
\end{equation}
\end{theorem}

In particular, letting $q\searrow1$ (the constant in the proof remains bounded, see \cite[Theorem 3.6]{ZhZhl22}), we recover Zhang Qi's ultracontractivity. Note that only the lower bound $R\geq-R_0$ is used. Along the K\"ahler-Ricci flow \eqref{KRF} the hypotheses are satisfied, by Theorem {\rm\ref{ZhZh22}} and Lemma {\rm\ref{scalarlbunif}}, and we obtain the on-diagonal bound \eqref{ondiagheatub} for all $0\leq s<t<T.$ For time intervals longer than one, the same monotonicity argument uses $\nu[g_0,2T]$ instead of $\nu[g_0,T+1]$.

\section{Distance distortion estimates and the construction of cut-off function}\label{sec3}
In this section we first set up the Ricci potential along the K\"ahler-Ricci flow \eqref{KRF} and establish a Li-Yau type inequality for positive solutions of the heat equation inspired by Bamler-Zhang \cite{BamZh17}, Hallgren-Jian-Song-Tian \cite{HJST23}. This inequality is independent of the Ricci curvature, but it depends on the Ricci potential. We then prove Theorems \ref{localdistortion} and \ref{cutoffftn}. The key observation is that, at scales $r_0\lesssim T-t_0$, the type-II
scalar curvature bound \eqref{type2-scalar} implies the bounded
scalar curvature bound $R\lesssim r_0^{-2}$ on the corresponding time interval. Hence, at these scales, the
results of Bamler--Zhang \cite{BamZh17} apply after a parabolic rescaling.

\subsection{Set up and preliminary results} We consider the following setting:

Let $(X^n,\omega_0)$ be a compact K\"ahler manifold of dimension $n$ and $(Y^k,\omega_Y)$ be a compact connected K\"ahler manifold of dimension $k$, $k\in\{0,1,\cdots,n\}$. Assume in addition that $\pi:X\to Y$ is a holomorphic map which satisfies
\begin{equation}\label{omega0-omegay}
[\omega_0]-2\pi T c_1(X)=[\pi^*\omega_Y].
\end{equation}
Furthermore, we focus on the K\"ahler-Ricci flow (\ref{KRF}) with initial metric $\omega_0$ and (\ref{omega0-omegay}). We remark that under the noncollapsing assumption $\int_X\beta^n>0$ of Theorem \ref{localdistortion} the class $[\pi^*\omega_Y]$ is big, so that necessarily $k=n$ and $\pi$ is generically finite; the case $k<n$ corresponds to a finite time collapsing of the flow, which is not considered in this paper. The assumption (\ref{omega0-omegay}) is used here, in Theorem \ref{LiYauthm}, and in Proposition \ref{propY}$(b),(c)$; the proofs of Theorems \ref{localdistortion}--\ref{hke2} do not depend on it.

Recall that the K\"ahler-Ricci flow (\ref{KRF}) is equivalent to the following parabolic complex Monge-Amp\`ere equation
\begin{equation}\label{PCMA}
  \begin{cases}
  \frac{\partial\varphi}{\partial t}&=\log\frac{\big(\omega_t+\sqrt{-1}\partial\bar{\partial}\varphi\big)^n}{\Omega}\\
  \varphi(0)&=0
  \end{cases}
\end{equation}
where $\omega_t=\frac{1}{T}(T-t)\omega_0+\frac{t}{T}\pi^*\omega_Y$ and $\sqrt{-1}\partial\bar{\partial}\log \Omega=\frac{\pi^*\omega_Y-\omega_0}{T}.$ Then we have
\begin{equation}\label{riccip-ub}
|\varphi|\leq C,\qquad\frac{\partial\varphi}{\partial t}\leq C
\end{equation}
where $C$ is a uniform constant depending only on $n,\omega_0,\omega_Y$ (see \cite[Lemma 2.1]{SW13} or \cite[Lemma 3.6]{Zha23}).

Let $u_0(t)=\frac{\partial\varphi}{\partial t}+\frac{\varphi+nt}{T-t}$ be the Ricci potential. Then we have the following coupled equations
\begin{equation}\label{ricci-potential-equation}
  \begin{cases}
  \frac{\partial}{\partial t}u_0&=\Delta u_0+\frac{u_0}{T-t}+\frac{tr_{\omega(t)}\pi^*\omega_Y}{T-t}=-R(t)+\frac{u_0}{T-t}+\frac{n}{T-t}\\
  Ric(\omega(t))&=\frac{\omega(t)-\pi^*\omega_Y}{T-t}-\sqrt{-1}\partial\bar{\partial}u_0.
  \end{cases}
\end{equation}
Since $tr_{\omega(t)}\pi^*\omega_Y\geq0,$ the first equation gives $$(\partial_t-\Delta)u_0\geq\frac{u_0}{T-t}.$$
The maximum principle shows that $t\mapsto(T-t)\inf_Xu_0(\cdot,t)$ is nondecreasing. Together with (\ref{riccip-ub}) (note that $u_0(\cdot,0)=\log\frac{\omega_0^n}{\Omega}$ is bounded) we obtain
\begin{equation}\label{u0-2side}
-\frac{C_*}{T-t}\leq u_0\leq\frac{C_*}{T-t}~~~~\text{on}~~X\times[0,T)
\end{equation}
for a constant $C_*=C_*(n,\omega_0,\omega_Y)<\infty.$ We now define the normalized Ricci potential by
$$v(x,t):=u_0(x,t)+\frac{C_*}{T-t}+1.$$
Since $v$ and $u_0$ differ by a function of $t$ only, we have $$\sqrt{-1}\partial\bar{\partial}v=\sqrt{-1}\partial\bar{\partial}u_0.$$
So, the second equation in (\ref{ricci-potential-equation}) holds with $u_0$ replaced by $v,$ and a direct computation from the first equation in (\ref{ricci-potential-equation}), using $\frac{d}{dt}\frac{C_*}{T-t}=\frac{C_*}{(T-t)^2},$ gives
\begin{equation}\label{ricci-potential-equation1}
\left(\frac{\partial}{\partial t}-\Delta\right)v=\frac{v-1}{T-t}+\frac{tr_{\omega(t)}\pi^*\omega_Y}{T-t},
\end{equation}
while (\ref{u0-2side}) gives
\begin{equation}\label{v-lb-ub}
1\leq v\leq \frac{C}{T-t}.
\end{equation}
We stress that $v$ is smooth in space-time, since no infimum over $X$ is involved in its definition. Using the Ricci-potential \eqref{ricci-potential-equation} and the parabolic Schwarz Lemma, the maximum principle calculation in \cite[Theorem 1.5]{JST23} (see also \cite[(3.17)]{HJST23}), gives
$$\frac{|\Delta w|}{w}+\frac{|\nabla w|^2}{w}\leq\frac{C}{T-t},$$
where $$w=u_0-\inf_Xu_0(\cdot,t)+1.$$
In our setting, the auxiliary positive forms can be chosen as fixed multiples of $\omega_Y$. Thus the same calculation applies to these analytic ingredients.

Since $v\geq w>0$ by (\ref{u0-2side}) and $\nabla v=\nabla w,~\Delta v=\Delta w,$ we obtain the following gradient and Laplacian estimates
\begin{equation}\label{LiYau1}
\frac{|\Delta v|}{v}+\frac{|\nabla v|^2}{v}\leq \frac{C(n,\omega_0,\omega_Y)}{T-t}
\end{equation}
on $X\times[0,T),$ where the gradient and Laplacian are with respect to $\omega(t).$ Furthermore, taking the trace of the second equation in (\ref{ricci-potential-equation}) and using $tr_{\omega(t)}\pi^*\omega_Y\geq0,$ (\ref{LiYau1}) and (\ref{v-lb-ub}), we find
\begin{equation}\label{type2-scalar}
R\leq\frac{n}{T-t}+|\Delta v|\leq\frac{Cv}{T-t}\leq\frac{C}{(T-t)^2}~~~~\text{on}~~X\times[0,T),
\end{equation}
which recovers, in the present setting, the scalar curvature bound of \cite{ZhZ10}. 

\begin{remark}
Let $f(t)=(T-t)\big(\log(T-t)-1\big)$ and $\psi(x,t)=\varphi+(n-k)f(t)$, where $k$ is the dimension of $Y$. Then by Lemma \rm{3.21} in \cite{Zha23}, we have
\begin{equation}\label{ricci-potential-ublb}
\frac{-C}{T-t}\leq\frac{\partial}{\partial t}\psi\leq-(n-k)\log(T-t)+C,
\end{equation}
which gives another derivation of the two sided bound \rm{(\ref{u0-2side})}, namely $$-\frac{C}{T-t}\leq\frac{\partial\varphi}{\partial t}\leq C~~~ \text{and}~~|\varphi|\leq C.$$ \qed
\end{remark}
In the following subsection, we let $T=1.$

\subsection{Li-Yau type inequality}
Inspired by the recent work of Hallgren-Jian-Song-Tian \cite{HJST23}, we prove the following Li-Yau type inequality under the assumption (\ref{omega0-omegay}). It is not needed for the proofs of Theorems \ref{localdistortion}--\ref{hke2}. We include it here because it is the natural substitute. In the present type-II setting, for the Li-Yau type estimates of \cite{BamZh17} for Ricci flows with bounded scalar curvature.
\begin{theorem}\label{LiYauthm}
Assume \rm{(\ref{omega0-omegay})} and let $T=1.$ There exists $C=C(n,\omega_0,\omega_Y)<\infty$ such that the following statement holds. Let $0\leq t_0<t_1<1$ with $1-t_0\leq2(1-t_1).$ Then for any positive solution $u\in C^{\infty}(X\times[t_0,t_1])$ of the heat equation \rm{(\ref{heateqa})}, we have
\begin{equation}\label{liyau-ineq}
-\frac{\Delta u}{u}+\frac{1}{2}\cdot\frac{|\nabla u|^2}{u^2}\leq C\Big(\frac{1}{t-t_0}+\frac{v}{1-t_0}\Big)
\end{equation}
on $X\times(t_0,t_1].$
\end{theorem}
\begin{proof}
Set $t_*:=\frac{1+t_0}{2},$ so that $t_1\leq t_*<1$ by assumption, and consider the type-I rescaled flow
$$\hat{\omega}(\hat{t}):=\frac{1}{1-t_*}\,\omega\big((1-t_*)\hat{t}+t_*\big),~~~~\hat{t}\in[-\tfrac{t_*}{1-t_*},0],$$
which is the rescaling used in \cite[Section 3]{HJST23}. The times $t_0$ and $t_1$ correspond to $\hat{t}_0=-1$ and $\hat{t}_1=\frac{t_1-t_*}{1-t_*}\in(-1,0].$ Since $1-t=(1-t_*)(1-\hat{t}),$ then by (\ref{ricci-potential-equation1}), (\ref{LiYau1}), (\ref{v-lb-ub}) and (\ref{type2-scalar}), we see that the function $v$ satisfies
$$(\partial_{\hat{t}}-\hat{\Delta})v=\frac{v-1}{1-\hat{t}}+\frac{1}{1-\hat{t}}tr_{\hat{\omega}}\Big(\frac{\pi^*\omega_Y}{1-t_*}\Big),$$
$$\frac{|\hat{\Delta}v|}{v}+\frac{|\hat{\nabla}v|^2}{v}\leq\frac{C}{1-\hat{t}}\leq C,~~~~v\geq1,$$
$$R_{\hat{\omega}}\leq\frac{Cv}{1-\hat{t}}\leq Cv,$$
on $X\times[-\frac{t_*}{1-t_*},0],$ where all the quantities are taken with respect to $\hat{\omega}(\hat{t}).$ The parabolic Schwarz lemma \cite{ST07,ST12} gives $tr_{\hat{\omega}}\big(\frac{\pi^*\omega_Y}{1-t_*}\big)\leq C$ and $$(\partial_{\hat{t}}-\hat{\Delta})tr_{\hat{\omega}}\big(\frac{\pi^*\omega_Y}{1-t_*}\big)\leq-C^{-1}\big|\hat{\nabla}tr_{\hat{\omega}}\big(\frac{\pi^*\omega_Y}{1-t_*}\big)\big|^2+C(1-t_*).$$
Writing $a=1-t_*$ and $S=\tr_{\hat\omega}(\pi^*\omega_Y/a)$, the tensor identity is
$$\Ric(\hat\omega)=\frac{\hat\omega-\pi^*\omega_Y/a}{1-\hat t}
-\sqrt{-1}\partial\bar\partial v.$$
The displayed potential equation and $v\geq1$, $S\leq C$ also imply
$$\frac{|\partial_{\hat t}v|}{v}
\leq\frac{|\hat\Delta v|}{v}+\frac{v-1+S}{(1-\hat t)v}\leq C.$$
These bounds, together with the Schwarz lemma, give the properties (3.19)--(3.21) of \cite{HJST23} for the potential $v_i$ there (with $\theta_Y=\omega_Y,$ hence $\rho=0$ and $B_0=0$ in the notation of \cite{HJST23}), with the only difference that the target $(\mathbb{CP}^N,\omega_{FS})$ of \cite{HJST23} is replaced by the compact K\"ahler manifold $(Y,\omega_Y),$ whose bounded bisectional curvature gives the parabolic Schwarz Lemma. The auxiliary positive forms in that proof can be chosen as fixed positive multiples of $\omega_Y$. Hence the proof of \cite[Theorem 3.6]{HJST23} applies to $\hat{\omega}$ using these verified analytic estimates on the time interval $[\hat{t}_0,\hat{t}_1]=[-1,\hat{t}_1]$ and gives, for the positive solution $u$ of $\partial_{\hat{t}}u=\hat{\Delta}u,$
$$-\frac{\hat{\Delta}u}{u}+\frac{1}{2}\frac{|\hat{\nabla}u|^2}{u^2}\leq C\Big(\frac{1}{\hat{t}-\hat{t}_0}+v\Big)~~~~\text{on}~~X\times(-1,\hat{t}_1],$$
where $C$ depends only on $n,\omega_0,\omega_Y,$ because the initial time $\hat{t}_0=-1$ is fixed (in \cite{HJST23} the constant depends on the initial time). Rewriting this in terms of $\omega(t),$ using $$\hat{\Delta}=(1-t_*)\Delta,~~|\hat{\nabla}u|^2=(1-t_*)|\nabla u|^2~~\text{and} ~~\hat{t}-\hat{t}_0=\frac{t-t_0}{1-t_*},$$ and dividing by $1-t_*=\frac{1-t_0}{2},$ we obtain the required estimate (\ref{liyau-ineq}).
\end{proof}

\begin{remark}\label{liyauremark}
Two comments on the form of (\ref{liyau-ineq}) are in order. First, the coefficient of the gradient term is $\frac{1}{2}$. The quantities $\frac{\Delta u}{u}$ and $\frac{|\nabla u|^2}{u^2}$ have the same scaling. For the Euclidean heat kernel $$G(z,\sigma)=(4\pi\sigma)^{-m/2}e^{-|z|^2/4\sigma}$$ one computes
$$-\frac{\Delta G}{G}+a\frac{|\nabla G|^2}{G^2}=\frac{m}{2\sigma}+(a-1)\frac{|z|^2}{4\sigma^2}.$$
By a local blow-up argument at a smooth time slice, no inequality of the form (\ref{liyau-ineq}) can hold with a coefficient $a>1$ in front of the gradient term. Secondly, since $v\leq\frac{C}{1-t}$ by (\ref{v-lb-ub}), the right hand side of \eqref{liyau-ineq} is of the order
$\frac{1}{t-t_0}+\frac{1}{(1-t_0)^2}$, which is consistent with the type-II bound \eqref{type2-scalar}. In particular (\ref{liyau-ineq}) diverges as $t_0\to1,$ and it does not give uniform time-derivative bounds on the heat kernel at scales larger than $1-t_0.$ This is why the cut-off function of Theorem \ref{cutoffftn} is constructed below directly from Bamler-Zhang \cite{BamZh17}. \qed
\end{remark}

\subsection{Distance distortion estimates}\label{3.1}

We derive Theorem \ref{localdistortion} from the distance distortion estimate of Bamler-Zhang \cite[Theorem 1.1]{BamZh17}. Let $(M^m,(g_t)_{t\in[0,T)})$ be a Ricci flow on a compact manifold. Then there is a constant $0<\alpha<1,$ which only depends on $m$ and $\nu[g_0,2T],$ such that the following holds: Suppose that $t_0\in[0,T),$ $0<r_0\leq\sqrt{t_0},$ $x_0,y_0\in M$ with $d_{t_0}(x_0,y_0)\geq r_0$ and $t\in[t_0-\alpha r_0^2,\min\{t_0+\alpha r_0^2,T\}).$ Assume that
\begin{equation}\label{BZscalar}
R\leq r_0^{-2}~~~\text{on}~~M\times[t,t_0]~~\text{or on}~~M\times[t_0,t],
\end{equation}
depending on whether $t\leq t_0$ or $t\geq t_0.$ Then
\begin{equation}\label{BZdistortion}
\alpha d_{t_0}(x_0,y_0)<d_t(x_0,y_0)<\alpha^{-1}d_{t_0}(x_0,y_0).
\end{equation}
In \cite{BamZh17} the upper bound (\ref{BZscalar}) is only required on a subset $U\times[t,t_0]$ (resp. $U\times[t_0,t]$), where $U$ contains a minimizing geodesic between $x_0$ and $y_0$ at all intermediate times. We shall only use the global version stated above. We stress that only an {\it upper} bound on the scalar curvature is required, and no noncollapsing assumption beyond the finiteness of $\nu[g_0,2T].$

\textbf{Proof~of~}Theorem \ref{localdistortion}: Let $1\leq C_{II}<\infty$ be such that $$R\leq C_{II}(T-t)^{-2}~~~ \text{on} ~~X\times[0,T),$$
see \eqref{type2-scalar} and the scalar input \cite{JS26} (or \cite{ZhZ10} in its pullback setting). Put
$$c_0:=\min\Big\{\frac{1}{2\sqrt{C_{II}}},~\frac{1}{\sqrt{2T}}\Big\}.$$
Let $t_0,~r_0,~x_0,~y_0$ be as in the theorem, so that $r_0\leq c_0(T-t_0),$ and let $$I_0:=[t_0-r_0^2,t_0+r_0^2].$$
Since $r_0^2\leq c_0^2(T-t_0)^2\leq\frac{T-t_0}{2}$ by the choice of $c_0,$ we have $$I_0\subset[0,\frac{T+t_0}{2}]~~~ \text{and}~~~ T-t\geq\frac{T-t_0}{2}$$
for all $t\in I_0.$ Consequently
$$R\leq\frac{C_{II}}{(T-t)^2}\leq\frac{4C_{II}}{(T-t_0)^2}\leq\frac{4C_{II}c_0^2}{r_0^2}\leq\frac{1}{r_0^2}~~~~\text{on}~~X\times I_0,$$
that is, (\ref{BZscalar}) holds for all $t\in I_0.$ Moreover $\nu[g_0,2T]>-\infty,$ since $\omega_0$ is a fixed smooth metric on the compact manifold $X.$ We may therefore apply (\ref{BZdistortion}) to the Ricci flow associated to (\ref{KRF}) as explained in Section \ref{preliminary}, at $(x_0,t_0)$ and at the scale $r_0,$ for all $t\in[t_0-\alpha r_0^2,t_0+\alpha r_0^2]\subset I_0.$ This gives the assertion, with a constant $\alpha$ which depends only on $n$ and $\nu[g_0,2T],$ hence only on $n$ and $\omega_0.$

Under the type-I assumption \eqref{typeIscalar}, after increasing $C_0$ to at least one, the same argument applies
with
$$c_0:=\min\{(2C_0)^{-1/2},2^{-1/2}\}.$$
Indeed, if $r_0\leq c_0\sqrt{T-t_0}$, then
$$r_0^2\leq c_0^2(T-t_0)\leq \frac{T-t_0}{2},$$
and hence
$$T-t\geq \frac{T-t_0}{2}
\qquad\text{for all }t\in I_0.$$
Consequently,
$$R\leq\frac{C_0}{T-t}\leq\frac{2C_0}{T-t_0}\leq\frac{2C_0c_0^2}{r_0^2}\leq\frac{1}{r_0^2}~~~~\text{on}~~X\times I_0.$$
This completes the proof of the theorem. \qed

\textbf{Proof~of~}Corollary \ref{globaldistortion}:
Let $\alpha$ be the constant from Theorem \ref{localdistortion}, and choose
$$A\geq\max\{2,2\alpha^{-1/2},2\alpha^{-1}\}.$$
Suppose that the assertion fails for some admissible $t\ne t_0$. By continuity there is a time $t_1$ strictly between $t_0$ and $t$, or equal to $t$, such that
$$d_{t_1}(x_0,y_0)=A\sqrt{r_0^2+|t_1-t_0|}.$$
Put
$$r_1:=\frac A2\sqrt{r_0^2+|t_1-t_0|}.$$
The assumptions imply
$$r_1^2\leq\frac{t_0}{4},\qquad
r_1\leq\frac{c_0}{2}(T-\max\{t_0,t\})\leq c_0(T-t_1).$$
Moreover, $|t-t_0|\leq t_0/A^2$, and hence
$$t_1\geq t_0-|t-t_0|\geq\frac34t_0\geq r_1^2.$$
Thus $r_1\leq\min\{\sqrt{t_1},c_0(T-t_1)\}$, while
$$|t_1-t_0|\leq\alpha r_1^2.$$
We may apply Theorem \ref{localdistortion} at time $t_1$ and scale $r_1$, since $d_{t_1}(x_0,y_0)=2r_1$. It gives
$$r_0=d_{t_0}(x_0,y_0)>\alpha d_{t_1}(x_0,y_0)
=\alpha A\sqrt{r_0^2+|t_1-t_0|}>r_0,$$
a contradiction. The case $t=t_0$ is immediate. \qed

As a consequence of Theorem \ref{localdistortion} we obtain the following two-sided comparison of the distance functions at times $s$ and $t$ at the scale $\sqrt{t-s}.$ It is this form of the distance distortion which will be used repeatedly in Sections \ref{sec5} and \ref{sec6}.

\begin{corollary}\label{distcomparison}
Under the assumption of {\rm Theorem \ref{localdistortion}}, there are constants $0<\alpha<1$ and $c_2>0$ depending only on $n,\omega_0$ such that for all $0\leq s<t<T$ with $t-s\leq c_2(T-t)^2$ and all $x,y\in X$
\begin{align}\label{distcomparison1}
\frac{d_t^2(x,y)}{t-s}\leq\alpha^{-1}+\alpha^{-2}\frac{d_s^2(x,y)}{t-s},~~~~\frac{d_s^2(x,y)}{t-s}\leq\alpha^{-1}+\alpha^{-2}\frac{d_t^2(x,y)}{t-s}.
\end{align}
Under the type-I assumption \eqref{typeIscalar} the same holds for all $0\leq s<t<T$ with $t-s\leq c_2(T-t),$ with constants depending on $C_0$ in {\rm Theorem \ref{hke2}}.
\end{corollary}
\begin{proof}
Let $\alpha_0\in(0,1)$ and $c_0$ be the constants of Theorem \ref{localdistortion}. Choose
$$0<c_2\leq\min\{\alpha_0c_0^2,\alpha_0/(4T)\},\qquad
r_0:=\alpha_0^{-1/2}\sqrt{t-s}.$$
If $t\leq T/2$, the two metrics are uniformly equivalent by Remark \ref{bddgeom}, and the assertion follows after decreasing the final constant $\alpha$. If $t\geq T/2$, then
$$r_0^2\leq\frac{(T-t)^2}{4T}\leq\frac{T}{16},\qquad
s=t-\alpha_0r_0^2\geq\frac{7T}{16}.$$
Consequently
$$r_0\leq\min\{\sqrt{s},c_0(T-t)\}.$$
If $d_t(x,y)<r_0$, then $d_t^2(x,y)/(t-s)<\alpha_0^{-1}$. Otherwise, applying Theorem \ref{localdistortion} at time $t$ gives $d_t(x,y)<\alpha_0^{-1}d_s(x,y)$. These two cases give the first inequality in \eqref{distcomparison1}, with $\alpha_0$ in place of $\alpha$. Applying the same argument at time $s$ gives the second inequality. Decreasing $\alpha$ once more incorporates the uniformly equivalent metrics on $[0,T/2]$.
For the type-I case choose instead
$$0<c_2\leq\min\{\alpha_0c_0^2,\alpha_0/8\}.$$
Then $t-s\leq c_2(T-t)$ implies $r_0\leq c_0\sqrt{T-t}$, and, when $t\geq T/2$, also $r_0^2\leq T/16$ and $s\geq7T/16$. The preceding argument therefore applies using the type-I version of Theorem \ref{localdistortion}.
\end{proof}

\subsection{Cut-off function}

Theorem \ref{cutoffftn} is obtained in the same way from the cut-off function of Bamler-Zhang \cite[Theorem 1.3]{BamZh17}, which reads as follows. Let $(M^m,(g_t)_{t\in[0,T)})$ be a Ricci flow on a compact manifold. Then there is a constant $\rho>0,$ which only depends on $m$ and $\nu[g_0,2T],$ such that the following holds: Let $(x_0,t_0)\in M\times[0,T),$ $0<r_0\leq\sqrt{t_0}$ and $0<\tau\leq\rho^2r_0^2,$ and assume that $$R\leq r_0^{-2}~~~~ \text{on}~~~ ~P(x_0,t_0,r_0,-\tau).$$ Then there is a function $\phi\in C^{\infty}(M\times[t_0-\tau,t_0])$ with the properties $(a)$--$(d)$ of Theorem \ref{cutoffftn}. 

\textbf{Proof~of~}Theorem \ref{cutoffftn}: Let $c_0$ be the constant of Theorem \ref{localdistortion}, and let $$r_0\leq\min\{\sqrt{t_0},c_0(T-t_0)\}$$ and $$0<\tau\leq\rho^2r_0^2\leq r_0^2.$$
As shown in the proof of Theorem \ref{localdistortion}, $$R\leq r_0^{-2}~~~~\text{on}~~~ X\times[t_0-r_0^2,t_0]\supset P(x_0,t_0,r_0,-\tau),$$
and $\nu[g_0,2T]>-\infty.$
Hence the cut-off function of \cite[Theorem 1.3]{BamZh17} exists at $(x_0,t_0)$ and at the scale $r_0,$ with a constant $\rho$ depending only on $n$ and $\nu[g_0,2T],$ hence only on $n,\omega_0.$ The type-I case is identical. \qed

\section{Mean value inequalities for the conjugate heat equation}\label{sec4}

In this section we present some consequences of the existence of the cut-off function from Theorem \ref{cutoffftn}. Motivated by the work of Bamler-Zhang \cite{BamZh17}, we follow their arguments to obtain mean value inequalities for the conjugate heat equation, on the scales on which Theorems \ref{localdistortion} and \ref{cutoffftn} are available. Those mean value inequalities will be used in the subsequent section to deduce a Gaussian upper bound for the heat kernel. Throughout this section $m=2n$ and $c_0$ denotes the constant of Theorem \ref{localdistortion}. Decreasing $c_0$ we may and do assume that $c_0T\leq1.$ More precisely, we obtain the following propositions.

\begin{lemma}\label{jifenLpcheateqa}
Assume the hypotheses of {\rm Theorem \ref{localdistortion}} and let $p>2$. Then there are constants $0<\alpha<1,~A<\infty$ that only depend on $n,~\omega_0$ and $p$ such that the following holds:

Let $(x_0,t_0)\in X\times[0,T)$ and $0<r_0\leq\min\{\sqrt{t_0},c_0(T-t_0)\},$ and let $u\in C^{\infty}(X\times[t_0,t_0+r_0^2])$ be a solution of the conjugate heat equation \eqref{cheateqa}. Then
\begin{equation}\label{Lpineq}
\big(\int_{Q^+(x_0,t_0,\alpha r_0)}|u|^p(x,t)dg_t(x)dt\big)^{\frac{2}{p}}\leq\frac{A}{r_0^{(m+2)(p-2)/p}}\int_{Q^+(x_0,t_0, r_0)}|u|^2(x,t)dg_t(x)dt.
\end{equation}
\end{lemma}
Note that $t_0+r_0^2<T$ since $r_0^2\leq\frac{T-t_0}{2},$ and that (\ref{Lpineq}) is invariant under parabolic rescaling.

\begin{proof}
All the ingredients of the proof below are invariant under the parabolic rescaling $$\omega\mapsto r_0^{-2}\omega,~~~t\mapsto r_0^{-2}t,$$ provided $r_0\leq1$. Theorems \ref{localdistortion} and \ref{cutoffftn} are formulated in a scale invariant way. The Sobolev inequality (\ref{Soboleviq}) is transformed into the same inequality with $R_0+1$ replaced by $r_0^2(R_0+1)\leq R_0+1.$ Since $r_0\leq c_0T\leq1,$ we may therefore assume $r_0=1.$ Moreover, replacing $u$ by $u_{\varepsilon}:=\sqrt{u^2+\varepsilon^2}$ and letting $\varepsilon\to0$ at the end (the flow being smooth on $[t_0,t_0+1],$ the function $u_{\varepsilon}$ satisfies the inequality (\ref{4.3'}) below up to an error term which vanishes as $\varepsilon\to0$), we may assume that $u>0.$

Let $\rho>0$ be the constant in Theorem \ref{cutoffftn}. By the distance distortion estimate of Theorem \ref{localdistortion}, applied at scales comparable to $1,$ we can find uniform constants $0<\gamma\leq\frac{\beta\rho}{2}<\beta\leq\frac{1}{2}$ such that
$$B(x_0,t_1,\beta)\subset B(x_0,t_2,1)~~~\text{and}~~~B(x_0,t_1,\gamma)\subset B(x_0,t_2,\beta\rho)~~~\text{for~~any}~~t_1,t_2\in[t_0,t_0+2\gamma^2].$$
Let $\phi\in C^{\infty}(X\times[t_0,t_0+2\gamma^2])$ be the cut-off function from Theorem \ref{cutoffftn} applied at $x_0$, at scale $\beta$ and at time $t_0+2\gamma^2$ (with $\tau=2\gamma^2\leq\rho^2\beta^2$), that is,
\begin{enumerate}
\item $0\leq\phi<1,~|\nabla\phi|<\beta^{-1},~|\partial_t\phi|+|\Delta\phi|<\beta^{-2}~~\text{on}~X\times[t_0,2\gamma^2+t_0];$
\item $\phi=0~~\text{on}~~(X\setminus B(x_0,t_0+2\gamma^2,\beta))\times[t_0,t_0+2\gamma^2],~~\text{and}~~B(x_0,t_0+2\gamma^2,\beta)\times[t_0,t_0+2\gamma^2]\subset Q^+(x_0,t_0,1);$
\item $\phi>\rho~~\text{on}~~P(x_0,t_0+2\gamma^2,\beta\rho,-2\gamma^2)\supset Q^+(x_0,t_0,\gamma).$
\end{enumerate}
Let $\eta\in C^{\infty}([t_0,2\gamma^2+t_0])$ be a cut-off function with $\eta\equiv1$ on $[t_0,\gamma^2+t_0],~~\eta(2\gamma^2+t_0)=0,$ $0\leq\eta\leq1$ and $|\eta'|<C_1$ for some uniform constant $C_1<\infty$. Then the function
$$\psi(x,t)=\eta(t)\phi(x,t)$$
still satisfies the properties $(1)$, $(2)$ above, with $\beta^{-1},\beta^{-2}$ replaced by a uniform constant $C_0,$ it vanishes at $t=t_0+2\gamma^2,$ and
$$\psi>\rho~~\text{on} ~~Q^+(x_0,t_0,\gamma)\subset B'\times[t_0,t_0+\gamma^2],$$ where $$B':=B(x_0,t_0+2\gamma^2,\beta\rho)$$
(note that $\eta\equiv1$ on the time interval $[t_0,t_0+\gamma^2]$ of $Q^+(x_0,t_0,\gamma)$).

We first note that, for $u>0$ and $p\geq2,$
\begin{equation}\label{4.3'}
\Delta u^{\frac{p}{2}}-\frac{p}{2}Ru^{\frac{p}{2}}+\partial_tu^{\frac{p}{2}}\geq0.
\end{equation}
Multiplying (\ref{4.3'}) by $2u^{\frac{p}{2}}\psi^2$ and integrating at a time $t\in[t_0,2\gamma^2+t_0]$
yields, using $\partial_t(dg_t)=-Rdg_t,$
\begin{align}
\frac{d}{dt}\int_X u^p\psi^2dg_t&\nonumber=2\int_X \partial_tu^{\frac{p}{2}}u^{\frac{p}{2}}\psi^2dg(t)+\int_Xu^p\partial_t(\psi^2)dg_t-\int_Xu^p\psi^2Rdg_t\\ \nonumber
&\geq-2\int\Delta u^{\frac{p}{2}}u^{\frac{p}{2}}\psi^2dg_t+(p-1)\int u^p\psi^2Rdg_t-2C_0\int_Bu^pdg_t\\ \nonumber
&=2\int\nabla u^{\frac{p}{2}}(\nabla(u^{\frac{p}{2}}\psi)\psi+u^{\frac{p}{2}}\psi\nabla\psi)+(p-1)\int u^p\psi^2Rdg_t-2C_0\int_Bu^pdg_t\\ \label{4.4'}
&=2\int|\nabla(\psi u^{\frac{p}{2}})|^2-2\int u^p|\nabla\psi|^2+(p-1)\int u^p\psi^2Rdg_t-2C_0\int_Bu^pdg_t,
\end{align}
where $B=B(x_0,t_0+2\gamma^2,\beta)$ and we used $|\partial_t\psi|\leq C_0$ together with $supp~\psi(\cdot,t)\subset B.$

Set $$f:=\psi u^{\frac{p}{2}}$$
and write $$(p-1)R=(p-1)(R+R_0+1)-(p-1)(R_0+1),$$ where $R_0$ is the constant of Lemma \ref{scalarlbunif}. Put $$\kappa:=\min\{2,p-1\}>0$$ and $$|\nabla\psi|\leq C_0,~\psi\leq1,$$ then the right hand side of (\ref{4.4'}) is bounded from below by
$$\kappa\int_X\big(|\nabla f|^2+(R+R_0+1)f^2\big)dg_t-\big(2C_0+2C_0^2+(p-1)(R_0+1)\big)\int_Bu^pdg_t.$$
We now apply the uniform Sobolev inequality (\ref{Soboleviq}) of Theorem \ref{ZhZh22} to $f,$ with the exponent $$2q:=\frac{2m}{m-2}=\frac{2n}{n-1},~~~ i.e.~~~q=\frac{m}{m-2},$$ and obtain
\begin{equation}\label{4.5'}
\frac{d}{dt}\int_X u^p\psi^2dg_t\geq C_2^{-1}\big(\int_X\psi^{2q}u^{pq}dg_t\big)^{1/q}-C_3\int_Bu^pdg_t.
\end{equation}

We stress that it is important to use here the critical inequality (\ref{Soboleviq}), which contains the scalar curvature term and is invariant under parabolic rescaling. The zero order term $(R_0+1)f^2$ in (\ref{Soboleviq}) only produces the harmless last term in (\ref{4.5'}).

Integrating (\ref{4.5'}) from $t_1\in[t_0,t_0+\gamma^2]$ to $t_0+2\gamma^2$ and observing that the integral on the left-hand side vanishes for $t=t_0+2\gamma^2$ and that $\psi>\rho$ on $B'\times[t_0,t_0+\gamma^2]$ yields
\begin{align*}
&0-\int_Xu^p\psi^2(\cdot,t_1)dg_{t_1}\\
\geq~&~C_2^{-1}\int_{t_1}^{t_0+2\gamma^2}\Big(\int_X\psi^{2q}u^{pq}dg_t\Big)^{1/q}dt-C_3\int_{t_1}^{t_0+2\gamma^2}\int_Bu^pdg_tdt\\
\geq~&~C_2^{-1}\rho^2\int_{t_1}^{t_0+\gamma^2}\Big(\int_{B'}u^{pq}dg_t\Big)^{1/q}dt-C_3\int_{t_0}^{t_0+2\gamma^2}\int_Bu^pdg_tdt.
\end{align*}
Letting $t_1$ vary over $[t_0,t_0+\gamma^2],$ we obtain
$$\sup_{t\in[t_0,t_0+\gamma^2]}\int_{B'}u^p(\cdot,t)dg_t+\int_{t_0}^{t_0+\gamma^2}\Big(\int_{B'}u^{pq}dg_t\Big)^{1/q}dt\leq C_4\int_{t_0}^{t_0+2\gamma^2}\int_Bu^pdg_tdt.$$
By H\"older inequality, we have
\begin{align*}
\int_{t_0}^{t_0+\gamma^2}\int_{B'}u^{p(1+\frac{q-1}{q})}dg_tdt\leq&~\int_{t_0}^{t_0+\gamma^2}\Big(\int_{B'}u^{pq}dg_t\Big)^{1/q}\Big(\int_{B'}u^p\Big)^{\frac{q-1}{q}}dt\\
\leq&~C_4^{1+\frac{q-1}{q}}\Big(\int_{t_0}^{t_0+2\gamma^2}\int_{B}u^pdg_tdt\Big)^{1+\frac{q-1}{q}}.
\end{align*}
As $Q^+(x_0,t_0,\gamma)\subset B'\times[t_0,t_0+\gamma^2]$ and $B\times[t_0,t_0+2\gamma^2]\subset Q^+(x_0,t_0,1),$ it follows, since $1+\frac{q-1}{q}=\frac{m+2}{m},$ that
$$\Big(\int_{Q^+(x_0,t_0,\gamma)}u^{p\frac{m+2}{m}}dg_tdt\Big)^{\frac{m}{p(m+2)}}\leq C_5\Big(\int_{Q^+(x_0,t_0,1)}u^pdg_tdt\Big)^{\frac{1}{p}}.$$

Starting with $p=2$, for which the same energy calculation is valid, and applying this inequality successively at the scales $\gamma^k$ (which is legitimate, since the inequality is scale invariant and $\gamma^k\leq1$), yields the desired inequality for $p=2(\frac{m+2}{m})^k$ and $\alpha=\gamma^k$ and some $A$ which depends on $k$. If $p$ takes other values, we can use interpolation to prove the result. From above, we complete the proof.
\end{proof}

At first glance, it may seem that Moser iteration in the proof of the previous lemma will lead to the mean value inequality. However, since the size of the cube $Q^+(x_0,t_0,\gamma^k)$ shrinks to 0 when $k\to\infty,$ we here follow the same path as the argument of Bamler-Zhang in \cite[Lemma 4.2]{BamZh17} to prove the mean value inequality.
\begin{proposition}\label{supcheateqa}
Under the assumption of {\rm Theorem \ref{localdistortion}}, there are constants $0<\alpha<1,~A<\infty$ that only depend on $n,~\omega_0$ such that the following holds:

Let $(x_0,t_0)\in X\times[0,T)$ and $0<r_0\leq\min\{\sqrt{t_0},c_0(T-t_0)\},$ and let $u\in C^{\infty}(X\times[t_0,t_0+r_0^2])$ be a solution of the conjugate heat equation \rm{(\ref{cheateqa})}. Then
\begin{equation}
\sup_{Q^+(x_0,t_0,\alpha r_0)}|u|^2\leq\frac{A}{r_0^{m+2}}\int_{Q^+(x_0,t_0, r_0)}|u|^2(x,t)dg(t)(x)dt.
\end{equation}
\end{proposition}
\begin{proof}
As in the previous proof, we may assume that $r_0=1.$ Let
$\alpha>0$ be the constant from Lemma \ref{jifenLpcheateqa} for some fixed $p>m+2,$ say $p=m+3,$ and let $\rho>0$ be the constant of Theorem \ref{cutoffftn}. By Theorem \ref{localdistortion} we can find uniform constants $0<\beta<\gamma<\alpha$ with $2\beta^2\leq\rho^2\gamma^2$ such that
\begin{equation}\label{4.8}
Q^+(x_0,t_0,\beta)\subset P(x_0,t_0+2\beta^2,\gamma\rho,-2\beta^2)~~~\text{and}~~~B(x_0,t_0+2\beta^2,\gamma)\times[t_0,t_0+2\beta^2]\subset Q^+(x_0,t_0,\alpha).
\end{equation}
Applying Theorem \ref{cutoffftn} at the point $x_0$, at scale $\gamma$ and at time $t_0+2\beta^2$ (with $\tau=2\beta^2$), we obtain a cutoff function $\phi\in C^{\infty}(X\times[t_0,t_0+2\beta^2])$ with the following properties for some constant $C_0<\infty:$
\begin{enumerate}
\item $0\leq\phi<1,~|\nabla\phi|<C_0,~|\partial_t\phi|+|\Delta\phi|<C_0~~on~X\times[t_0,2\beta^2+t_0];$
\item $\phi=0~~on~~(X\setminus B(x_0,t_0+2\beta^2,\gamma))\times[t_0,t_0+2\beta^2];$
\item $\phi>\rho~~on~~P(x_0,t_0+2\beta^2,\gamma\rho,-2\beta^2).$
\end{enumerate}
Let $\eta\in C^{\infty}([t_0,t_0+2\beta^2])$ be a cutoff function with $\eta\equiv1$ on $[t_0,t_0+\beta^2],~\eta(t_0+2\beta^2)=0,$ $0\leq\eta\leq1$ and $|\eta'|<C_1$ for some uniform constant $C_1<\infty$. Then the function
$$\psi(x,t):=\eta(t)\phi(x,t)$$
still satisfies properties $(1),(2)$ above, possibly with a larger constant $C_0,$ we have $\psi(\cdot,t_0+2\beta^2)\equiv0,$ and $\psi>\rho$ on $Q^+(x_0,t_0,\beta)$ by (\ref{4.8}).

Fix $(x,t)\in Q^+(x_0,t_0,\beta)$ and consider the heat kernel $H(z,s):=H(z,s;x,t)$ of the heat equation, $\partial_sH=\Delta H,$ with pole at $(x,t).$ Using
$$\partial_su=-\Delta u+Ru,~~\partial_sH=\Delta H \text{and} \partial_s(dg_s)=-Rdg_s,$$
we compute
\begin{align*}
\frac{d}{ds}\int_XHu\psi\,dg_s&=\int_X\big(\Delta H\cdot u\psi-H\psi\Delta u+Hu\partial_s\psi\big)dg_s\\
&=\int_XH\big(u\Delta\psi+u\partial_s\psi+2\langle\nabla u,\nabla\psi\rangle\big)dg_s,
\end{align*}
where we integrated by parts in the first term. Integrating from $s=t$ to $s=t_0+2\beta^2,$ and using
$$\psi(\cdot,t_0+2\beta^2)=0~~\text{and} ~~\lim_{s\searrow t}\int_XHu\psi dg_s=(u\psi)(x,t),$$
we obtain the representation formula
$$(u\psi)(x,t)=-\int_t^{t_0+2\beta^2}\int_XH(z,s)\big\{u(\Delta\psi+\partial_s\psi)+2\langle\nabla u,\nabla\psi\rangle\big\}(z,s)dg_sds.$$
Integrating by parts once more in the last term, $$\int_XH\langle\nabla u,\nabla\psi\rangle dg_s=-\int_Xu\big(\langle\nabla_zH,\nabla\psi\rangle+H\Delta\psi\big)dg_s,$$ this becomes
\begin{align*}
(u\psi)(x,t)&~=~\int_t^{t_0+2\beta^2}\int_XH(z,s)u(z,s)\big(\Delta\psi-\partial_s\psi\big)(z,s)dg_sds\\
&~~+2\int_t^{t_0+2\beta^2}\int_Xu(z,s)\langle\nabla_zH(z,s),\nabla\psi(z,s)\rangle dg_sds.
\end{align*}
From properties $(1),(2)$ of the cutoff function $\psi$ and (\ref{4.8}), it follows that
\begin{align*}
|u\psi|(x,t)&~\leq~2C_0\int_t^{t_0+2\beta^2}\int_{B(x_0,s,\alpha)}H(z,s)|u|(z,s)dg_sds\\
&~~~+~2C_0\int_t^{t_0+2\beta^2}\int_{B(x_0,s,\alpha)}|u|(z,s)|\nabla_zH|(z,s)dg_sds\\
&~~~\equiv~2C_0I_1+2C_0I_2.
\end{align*}
Estimate $I_1$. By H\"older's inequality, for $q=\frac{p}{p-1},$
\begin{align*}
I_1~\leq&~\Big(\int_t^{t_0+2\beta^2}\int_{B(x_0,s,\alpha)}H^q(z,s)dg_sds\Big)^{\frac{1}{q}}\Big(\int_t^{t_0+2\beta^2}\int_{B(x_0,s,\alpha)}|u|^p(z,s)dg_sds\Big)^{\frac{1}{p}}\\
~\leq&~\Big(\int_t^{t_0+2\beta^2}\int_{B(x_0,s,\alpha)}H^q(z,s)dg_sds\Big)^{\frac{1}{q}}\cdot A||u||_{L^2(Q^+(x_0,t_0,1))},
\end{align*}
where we used Lemma \ref{jifenLpcheateqa}. By (\ref{ondiagheatub}) again, there exists a universal constant $C_1<\infty$ such that
\begin{equation}\label{4.9}
H(z,s)\leq\frac{C_1}{(s-t)^{\frac{m}{2}}}~~for~~s\in(t,t+1].
\end{equation}
On the other hand, the rescaled scalar curvature is bounded from below by $-r_0^2R_0\geq-R_0$, since $r_0\leq1$. We deduce from
$$\frac{d}{ds}\int_XH(z,s)dg_s=-\int_XR(z,s)H(z,s)dg_s\leq R_0\int_XH(z,s)dg_s$$
that
$$\int_XH(z,s)dg_s\leq e^{R_0(s-t)}\leq e^{R_0}=:C_2~~~\text{for}~~~s\in(t,t+1].$$
So by the interpolation inequality
\begin{align*}
I_1~\leq&~C_3\Big(\int_t^{t_0+2\beta^2}\int_{B(x_0,s,\alpha)}H(z,s)\frac{1}{(s-t)^{\frac{m}{2(p-1)}}}dg_sds\Big)^{\frac{1}{q}}\cdot A||u||_{L^2(Q^+(x_0,t_0,1))}\\
\leq&~C_4\Big(\int_t^{t_0+2\beta^2}\frac{1}{(s-t)^{\frac{m}{2(p-1)}}}ds\Big)^{\frac{1}{q}}\cdot A||u||_{L^2(Q^+(x_0,t_0,1))}.
\end{align*}
Since $p>m+2,$ we have $\frac{m}{2(p-1)}<1.$ This shows that
$$I_1\leq C_5A||u||_{L^2(Q^+(x_0,t_0,1))}.$$
Estimate $I_2.$ From $(\ref{4.9}),$ we know that for all $s\in(t,t+1],$
$$J_s:=\sup_{X\times[\frac{s+t}{2},s]}H(\cdot,\cdot;x,t)\leq\frac{2^{\frac{m}{2}}C_1}{(s-t)^\frac{m}{2}}.$$
Using (\ref{diffhar1}), we obtain that
$$\frac{|\nabla_zH(z,s)|^2}{H^2(z,s)}\leq\frac{C_6}{s-t}\log\frac{J_s}{H(z,s)}.$$
So
$$|\nabla_zH(z,s)|^q\leq H(z,s)\frac{C_7J_s^{q-1}}{(s-t)^{\frac{q}{2}}}\cdot\big(\frac{H(z,s)}{J_s}\big)^{q-1}\cdot\big(\log\frac{J_s}{H(z,s)}\big)^{\frac{q}{2}}.$$
Since $\frac{H(z,s)}{J_s}\leq1,$ the product of the last two factors is bounded from above, which implies
$$|\nabla_zH(z,s)|^q\leq\frac{C_8}{(s-t)^{\frac{((m+1)q-m)}{2}}}H(z,s).$$
Applying this gradient bound and H\"older inequality, we deduce similarly as before
\begin{align*}
I_2&\leq~\Big(\int_t^{t_0+2\beta^2}\int_{B(x_0,s,\alpha)}|u|^p(z,s)dg_sds\Big)^{\frac{1}{p}}\cdot\Big(\int_t^{t_0+2\beta^2}\int_{B(x_0,s,\alpha)}|\nabla_zH|^q(z,s)dg_sds\Big)^{\frac{1}{q}}\\
&\leq~\Big(\int_t^{t_0+2\beta^2}\int_{B(x_0,s,\alpha)}H(z,s)\frac{C_8}{(s-t)^{\frac{((m+1)q-m)}{2}}}dg_sds\Big)^{\frac{1}{q}}\cdot A||u||_{L^2(Q^+(x_0,t_0,1))}\\
&\leq~ C_9\Big(\int_t^{t_0+2\beta^2}\frac{ds}{(s-t)^{\frac{((m+1)q-m)}{2}}}\Big)^{\frac{1}{q}}\cdot A||u||_{L^2(Q^+(x_0,t_0,1))}\\
&\leq~C_{10}A||u||_{L^2(Q^+(x_0,t_0,1))},
\end{align*}
where we used $\frac{((m+1)q-m)}{2}<1$ when $p>m+2.$ Combining the above, since $\psi(x,t)>\rho$ we obtain $$|u|(x,t)\leq\rho^{-1}(2C_0C_5+2C_0C_{10})A||u||_{L^2(Q^+(x_0,t_0,1))}$$
for all $(x,t)\in Q^+(x_0,t_0,\beta),$ which is the assertion with $\beta$ in place of $\alpha.$ This completes the proof.
\end{proof}

\section{Bounds on the heat kernel and its gradient}\label{sec5}

Following the argument of Bamler-Zhang \cite{BamZh17}, we are now ready to prove bounds for the fundamental solution of the heat equation and its gradient. A similar bound is proved in \cite{Bam20,BamZh17,ZhQ12}. An integral upper bound is proven in \cite{HNab14}, which is very important for us.

We first prove an integral-type heat kernel lower bound, which is essential for deriving an integral upper bound for the heat kernel. This type of bound may also be of independent interest and have further applications \cite{Hall23,TZhZ21}. It states that the conjugate heat kernel measure based at \((x,t)\) concentrates at scale \(\sqrt{t-s}\) around the point \(x\) itself.


\begin{lemma}\label{hklb}
Under the assumption of {\rm Theorem \ref{localdistortion}}, there exist uniform constants $\Lambda<\infty$ and $\kappa_0>0,$ depending only on $n,\omega_0,$ such that for any $(x,t)\in X\times(0,T)$ and any $s\in[0,t)$ with
\begin{equation}\label{hklbrange}
\Lambda^2(t-s)\leq\min\{t,(T-t)^2\}
\end{equation}
we have
\begin{equation}\label{hklb1}
\int_{B(x,s,\Lambda\sqrt{t-s})}H(x,t;z,s)dg_s(z)\geq \kappa_0>0.
\end{equation}
Under the type-I assumption \eqref{typeIscalar} the same holds with $(T-t)^2$ replaced by $T-t$ in \eqref{hklbrange}, with constants depending on $C_0$ in {\rm Theorem \ref{hke2}}.
\end{lemma}
\begin{proof}
Let $\rho\in(0,1)$ and $c_0$ be the constants of Theorem \ref{cutoffftn} and let $\alpha\in(0,1),~c_2$ be the constants of Corollary \ref{distcomparison}. Write
$$\Lambda_1:=\max\{\rho^{-1},\sqrt{2\rho^{-1}}\},~~~~\Lambda:=\max\Big\{\sqrt{\alpha^{-1}+\alpha^{-2}\Lambda_1^2},~\frac{\Lambda_1}{c_0},~\frac{1}{\sqrt{c_2}}\Big\}\geq\Lambda_1.$$
$$r_0:=\Lambda_1\sqrt{t-s}~~ \text{and}~~~\tau:=t-s.$$
Then we have
$$r_0\leq\Lambda\sqrt{t-s}\leq\sqrt{t}~~~\text{and}~~~r_0\leq c_0\Lambda\sqrt{t-s}\leq c_0(T-t).$$
Hence $\tau\leq\rho^2r_0^2$ by the choice of $\Lambda_1.$
Applying Theorem \ref{cutoffftn}, at the point $x$ at time $t$ and at scale $r_0,$ there exists a function $\phi\in C^{\infty}(X\times[s,t])$ with
$$0\leq\phi<1;~~\phi>\rho~\text{on}~P(x,t,\rho r_0,-\tau);~~\phi=0~\text{on}~(X\setminus B(x,t,r_0))\times[s,t];$$
$$|\frac{\partial}{\partial l}\phi|+|\Delta\phi|<r_0^{-2}~~\text{everywhere}.$$
In particular $$\left|\Box\phi\right|=\left|\Big(\frac{\partial}{\partial l}-\Delta\Big)\phi\right|<r_0^{-2}~~~\text{on}~~ X\times[s,t].$$

Let $d\nu_{x,t;l}=H(x,t;\cdot,l)dg_l$ be the conjugate heat kernel measure based at $(x,t)$. Then for every $\phi\in C^{\infty}(X\times[s,t])$
\begin{equation}\label{nu-derivative}
\frac{d}{dl}\int_X\phi(\cdot,l)d\nu_{x,t;l}=\int_X\Box\phi~d\nu_{x,t;l},~~~~\lim_{l\nearrow t}\int_X\phi(\cdot,l)d\nu_{x,t;l}=\phi(x,t),
\end{equation}
which is nothing but the duality (\ref{heatkernel}), (\ref{cheatkernel}) between the heat and the conjugate heat operator. Integrating (\ref{nu-derivative}) from $s$ to $t$ and using $\phi(x,t)>\rho$ we obtain
$$\int_X\phi(\cdot,s)d\nu_{x,t;s}=\phi(x,t)-\int_s^t\int_X\Box\phi~d\nu_{x,t;l}dl>\rho-\frac{t-s}{r_0^{2}}=\rho-\frac{1}{\Lambda_1^2}\geq\frac{\rho}{2},$$
by the choice of $\Lambda_1.$ Since $0\leq\phi\leq1$ and $\phi(\cdot,s)$ is supported in $B(x,t,r_0),$ this gives
$$\int_{B(x,t,r_0)}H(x,t;z,s)dg_s(z)\geq\frac{\rho}{2}.$$
It remains to replace the ball $B(x,t,r_0)$ at time $t$ by a ball at time $s.$ This is precisely Corollary \ref{distcomparison}, which applies to the pair $(s,t)$ since $$t-s\leq\Lambda^{-2}(T-t)^2\leq c_2(T-t)^2:$$
for $z\in B(x,t,r_0)$ we have
$$\frac{d_s^2(x,z)}{t-s}\leq\alpha^{-1}+\alpha^{-2}\frac{d_t^2(x,z)}{t-s}\leq\alpha^{-1}+\alpha^{-2}\Lambda_1^2\leq\Lambda^2.$$
Hence
$$B(x,t,r_0)\subset B(x,s,\Lambda\sqrt{t-s}),$$
by the choice of $\Lambda.$ This proves (\ref{hklb1}) with $\kappa_0=\frac{\rho}{2}.$ The type-I case is identical, using the type-I versions of Theorem \ref{cutoffftn} and Corollary \ref{distcomparison}.
\end{proof}
\begin{remark}
Under the additional type-I assumption \rm{(\ref{typeIscalar})}, Lemma \rm{\ref{hklb}} holds for all $0\leq s<t<T$ with $t-s\leq s,$ and it also follows from Hallgren's Gaussian lower bound together with Bamler's volume noncollapsing estimate; see Remark \rm{\ref{remark-alternative-hklb}}.
\end{remark}

\textbf{Proof~of~}Theorem \ref{hke1}: We will divide the proof into the following steps.

{\bf Step 1:} By Perelman's monotonicity and $T<\infty$ we have the uniform lower bound of the $\mu$-entropy (see Theorem \ref{ZhZh22}), and by Theorem \ref{GPSS22} the volume is noncollapsing.

{\bf Step 2:} Following Qi Zhang's argument, we have the on-diagonal upper bound for the heat kernel. Here, we present a more general ultracontractivity property of Ricci flow proved by the authors in \cite{ZhZhl22} (see Theorem \ref{ZhZl21}). Together with the uniform lower scalar curvature bound of Lemma \ref{scalarlbunif}, this gives a constant $C<+\infty$ such that
\begin{align}\label{on-diagub}
H(x,t;y,l)\leq \frac{C}{(t-l)^{m/2}},~~~~l\in[s,t).
\end{align}

{\bf Step 3:} In the spirit of the proof of the heat kernel estimates developed in \cite{BamZh17}, we obtain the Gaussian type upper bound for $H.$ Let $\Lambda$ be the constant of Lemma \ref{hklb}, let $c_0$ be the constant of Theorems \ref{localdistortion}, \ref{cutoffftn} and $c_2$ that of Corollary \ref{distcomparison}. We fix
$$c_1:=\frac{1}{2}\min\Big\{\frac{1}{\Lambda^2},~c_2,~2\Lambda^2c_0^2,~\frac{2}{(1+2\Lambda^2)T}\Big\},$$
and we assume from now on that (\ref{scalerange}) holds. For every $l\in[s,\frac{s+t}{2}],$ we have $$\Lambda^2(t-l)\leq(T-t)^2.$$
Hence Lemma \ref{hklb} applies to the pair $(l,t)$ as soon as $\Lambda^2(t-l)\leq t$. Also we have $$l-s\leq c_2(T-l)^2.$$ Therefore Corollary \ref{distcomparison} applies to the pair $(s,l)$. Define $$r:=\sqrt{(t-s)/(2\Lambda^2)}.$$ Then $$r\leq c_0(T-s).$$
Thus Proposition \ref{supcheateqa} applies at $(y,s)$ at the scale $r$ as soon as $r\leq\sqrt{s}.$ We may moreover assume that
$$2\Lambda^2(t-s)\leq s.$$
Indeed, if $2\Lambda^2(t-s)>s,$ then $$t<(1+2\Lambda^2)(t-s)\leq(1+2\Lambda^2)c_1(T-t)^2,$$ and this forces $t\leq\frac{T}{2}$ (if $t>\frac{T}{2}$ the right hand side is less than $(1+2\Lambda^2)c_1\frac{T^2}{4}\leq\frac{T}{2},$ a contradiction), in which case all the estimates of the theorem follow from Remark \ref{bddgeom}.

Applying Corollary \ref{distcomparison} to the pair of times $(s,l)$ and using $$t-l\geq\frac{t-s}{2}~~~ \text{and}~~~l-s\leq\frac{t-s}{2}$$ we obtain a uniform constant $B<\infty$ such that
\begin{align}\label{l1l2}
\frac{d_{l}^2(x,y)}{t-l}\geq B^{-1}\frac{d_{s}^2(x,y)}{t-s}-B,~~~l\in[s,\tfrac{s+t}{2}].
\end{align}
Next let us recall the following Gaussian concentration inequality for the conjugate heat kernel measure obtained by Hein-Naber \cite[Theorem 1.13 and (1.20)]{HNab14}: for $l\in[s,t)$ and any two Borel sets $A_1,A_2\subset X$,
\begin{align}\label{hngci}
\Big(\int_{A_1}H(x,t;z,l)dg_l(z)\Big)\cdot\Big(\int_{A_2}H(x,t;z,l)dg_l(z)\Big)\leq \exp\Big(-\frac{d^2_l(A_1,A_2)}{8(t-l)}\Big).
\end{align}
We apply (\ref{hngci}) with $A_1=B(x,l,\Lambda\sqrt{t-l})$ and $A_2=B(y,l,\sqrt{t-l}).$ Since $$d_l(A_1,A_2)\geq\big(d_l(x,y)-(\Lambda+1)\sqrt{t-l}\big)_+$$ and $$(a-b)_+^2\geq\frac{1}{2}a^2-b^2 ~~~\text{for}~~ a,b\geq0,$$
then Lemma \ref{hklb} and (\ref{l1l2}) give
\begin{align}\label{hngci1}
\int_{B(y,l,\sqrt{t-l})}H(x,t;z,l)dg_l(z)\leq C_2\exp\big(-\frac{d^2_l(x,y)}{16(t-l)}\big)\leq C_3\exp\big(-\frac{d^2_s(x,y)}{C_3(t-s)}\big).
\end{align}
Together with (\ref{on-diagub}) we get
$$\int_{B(y,l,\sqrt{t-l})}H^2(x,t;z,l)dg_l(z)\leq\frac{C_4}{(t-l)^{m/2}}\exp(-\frac{d^2_s(x,y)}{C_3(t-s)}).$$
Integrating over the time-interval $[s,(s+t)/2],$ we obtain that
$$\int_s^{(s+t)/2}\int_{B(y,l,\sqrt{t-l})}H^2(x,t;z,l)dg_l(z)dl\leq\frac{C_4(t-s)}{(t-s)^{m/2}}\exp(-\frac{d^2_s(x,y)}{C_3(t-s)}).$$
Note that $t-l\geq(t-s)/2$ on the time-interval $[s,(s+t)/2].$ Hence for $$r=\sqrt{(t-s)/(2\Lambda^2)}$$ we have $r\leq\sqrt{s}$ and $$Q^+(y,s,r)\subset\bigcup_{l\in[s,(s+t)/2]}B(y,l,\sqrt{t-l})\times\{l\}.$$
Therefore,
$$\int_{Q^+(y,s,r)}H^2(x,t;z,l)dg_l(z)dl<\frac{C_4}{(t-s)^{m/2-1}}\exp(-\frac{d_s^2(x,y)}{C_3(t-s)}).$$
Combining this estimate with the mean value inequality (Proposition \ref{supcheateqa}), applied to the function
$$(z,l)\mapsto H(x,t;z,l),$$
which solves the conjugate heat equation, we obtain
$$H^2(x,t;y,s)\leq\frac{C_6}{r^{m+2}}\int_{Q^+(y,s,r)}H^2(x,t;z,l)dg_l(z)dl<\frac{C_7}{(t-s)^m}\exp(-\frac{d_s^2(x,y)}{C_8(t-s)}).$$
This gives the upper bound (\ref{gthub}) on $H(x,t;y,s).$

{\bf Step 4:} We prove the upper bound on the gradient of the heat kernel. Fix $(y,s),$ regard $H(x,t)=H(x,t;y,s)$  as a solution to the forward heat equation on the time-interval $[\frac{t+s}{2},t]$. According to the gradient estimate (\ref{diffhar1}), we have
$$\frac{|\nabla_xH(x,t;y,s)|^2}{H^2(x,t;y,s)}\leq\frac{C_2}{t-s}\log\frac{J}{H(x,t;y,s)},$$
where $$J:=\sup_{X\times[\frac{t+s}{2},t]}H(\cdot,\cdot;y,s)\leq\frac{2^{m/2}C}{(t-s)^{m/2}}$$ by (\ref{on-diagub}). Since $\sigma\log\frac{1}{\sigma}\leq e^{-1}$ for $\sigma\in(0,1],$ we deduce, using Step 3,
\begin{align*}
|\nabla_xH(x,t;y,s)|^2&\leq\frac{C_2}{t-s}J\cdot H(x,t;y,s)\cdot\frac{H(x,t;y,s)}{J}\log\frac{J}{H(x,t;y,s)}\\
&\leq\frac{C_3}{t-s}J\cdot H(x,t;y,s)\\
&\leq\frac{C_4}{(t-s)^{2n+1}}\exp(-\frac{d_s^2(x,y)}{C_5(t-s)}).
\end{align*}
This yields the desired bound (\ref{tiduub}) for $|\nabla_xH|.$

{\bf Step 5:} The Gaussian lower bound in (\ref{gthub}) and the gradient bound (\ref{gradientup1}) are proved in Proposition \ref{Gauss-heat-kernel-ub-lb} below, see Steps 4 and 5 of its proof. The proof uses only Perelman's Harnack inequality in the form (\ref{heatke3}), the type-II bound (\ref{type2-scalar}), Corollary \ref{distcomparison} and Bamler's gradient estimate (\ref{graheatuppernash}), and it is independent of the previous steps. This completes the proof of Theorem \ref{hke1}. \qed

\section{Bounds on the heat kernel and its gradient II}\label{sec6}

In this section, we provide an alternative proof of the Gaussian-type upper bound in (\ref{gthub}). This proof uses neither the Hein-Naber inequality nor the mean value inequalities of Section \ref{sec4}. Instead, it depends on the integral lower bound of Lemma \ref{hklb} and on Bamler's estimates in terms of the pointed Nash entropy. We then prove the Gaussian lower bound and the gradient bound of Theorem \ref{hke1}, as well as Theorem \ref{hke2}. First, we recall a conjugate heat kernel estimate obtained in \cite{Hall23} (see also \cite[Theorem 1.4]{BamZh17}, \cite[Theorem 2.1]{CaoZh11}).

\begin{lemma}{\rm(Lemma 3.1 in \cite{Hall23})}\label{type1-hk-ublb}
Let $\big(M^m,(g_t)_{t\in[0,T)}\big)$ be a closed solution of Ricci flow satisfying $$\nu[g_0,2T]\geq-C_0~~ \text{and} ~~|R|(x,t)\leq \frac{C_0}{T-t}~~\forall~(x,t)\in M\times[0,T).$$
Then there exist uniform constants $C_1=C_1(C_0,m),C_2=C_2(C_0,m)$ such that for any $x,y\in M$ and $\frac{3T}{4}\leq s<t<T$ we have
\begin{equation}\label{heatkeublb}
\frac{1}{C_1(t-s)^\frac{m}{2}}\exp\left(-\frac{C_2d_s^2(x,y)}{t-s}\right)\leq H(x,t;y,s)\leq\frac{C_1}{(t-s)^\frac{m}{2}}\exp\left(-\frac{d_s^2(x,y)}{C_2(t-s)}\right).
\end{equation}
\end{lemma}
\begin{proof}
The rescaled flow $$\tilde{g}(r):=\frac{2}{T}g\big(T+\frac{T}{2}r\big),~r\in[-2,0),$$ satisfies $$\nu[\tilde{g}(-2),4]=\nu[g_0,2T]\geq-C_0$$ and $$|\tilde{R}|(x,r)=\frac{T}{2}|R|(x,t)\leq\frac{T}{2}\cdot\frac{C_0}{T-t}=\frac{C_0}{|r|},$$ where $t=T+\frac{T}{2}r.$ Hence \cite[Lemma 3.1]{Hall23} applies to $\tilde{g}$ and gives the two sided Gaussian bound for $$-\frac{1}{2}\leq\tilde{s}<\tilde{t}<0,$$ that is, for $$\frac{3T}{4}\leq s<t<T.$$ Writing $\tilde{s}=2(s-T)/T$ and $\tilde{t}=2(t-T)/T$, the heat kernels satisfy
$$\widetilde H(x,\tilde t;y,\tilde s)=(T/2)^{m/2}H(x,t;y,s),\qquad \widetilde d_{\tilde s}^{\,2}=(2/T)d_s^2.$$
Thus the factors introduced by rescaling cancel, and the constants may be chosen to depend only on $C_0,m$. In \cite[Lemma 3.1]{Hall23}, the estimates are first obtained on intervals on which the scalar curvature is bounded after rescaling. The reproduction formula at the midpoint, an integral lower bound for the conjugate heat kernel, and the distance distortion estimate then extend the lower bound to all $-1/2\leq\tilde s<\tilde t<0$. The upper bound follows from the Hein--Naber concentration inequality and the conjugate mean value inequality. We refer the readers to \cite[Lemma 3.1]{Hall23} for the details. 
\end{proof}

Next, we will use Bamler's estimates in \cite[Theorem 7.1, Theorem 7.2, Theorem 7.5]{Bam20} of upper bounds on the heat kernel and its gradient.
\begin{lemma}{\rm(Theorem 7.1, Theorem 7.2, Theorem 7.5 in \cite{Bam20})}\label{Bamlerheatupb}
Let $\big(M^m,(g(t))_{t\in I}\big)$ be a Ricci flow defined on a compact $m$-dimensional manifold. If $[s,t]\subset I$ and $R\geq R_{min}$ on $M\times[s,t],$ then
\begin{align}\label{heatuppernash}
H(x,t;y,s)\leq\frac{C_0\big(m,R_{min}(t-s)\big)}{(t-s)^\frac{m}{2}}\exp\big(-\mathcal{N}_{x,t}(t-s)\big),
\end{align}
for any $y,y'\in M$
\begin{align}\label{heatuppernash2}
H(x,t;y,s)H(x,t;y',s)\leq\frac{C_0\big(m,R_{min}(t-s)\big)\exp\big(-2\mathcal{N}_{x,t}(t-s)\big)}{(t-s)^m}\exp\Big(-\frac{d_s^2(y,y')}{9(t-s)}\Big),
\end{align}
and
\begin{align}\label{graheatuppernash}
\frac{|\nabla_xH|(x,t;y,s)}{H(x,t;y,s)}\leq\frac{C_0\big(m,R_{min}(t-s)\big)}{(t-s)^\frac{1}{2}}\sqrt{\log\Big(\frac{C_0\exp\big(-\mathcal{N}_{x,t}(t-s)\big)}{(t-s)^\frac{m}{2}H(x,t;y,s)}\Big)}
\end{align}
where $\mathcal{N}_{x,t}$ denotes the pointed Nash entropy based at $(x,t).$
\end{lemma}

The following proposition is essentially a combination of Bamler's heat-kernel estimates (Lemma \ref{Bamlerheatupb}) with the integral lower bound of Lemma \ref{hklb}. The argument is similar to that of Lemma 3.1 in \cite{Hall23} (see also \cite[Theorem 1.4]{BamZh17}, \cite[Theorem 1.2]{Bam18}). The only ingredient beyond Bamler's estimates is the integral lower bound (\ref{hklb1}); namely, the fact that the conjugate heat kernel measure has positive mass near the point $x$ itself at the earlier time $s$.

\begin{proposition}\label{Gauss-heat-kernel-ub-lb}
Under the assumption of {\rm Theorem \ref{localdistortion}}, there are constants $C_1,~C_2$ that only depend on $n,~\omega_0$ such that for all $0\leq s<t<T$ satisfying \eqref{scalerange}
\begin{align}\label{Gauss-HK-ub2}
H(x,t;y,s)&\leq\frac{C_1}{(t-s)^n}\exp\big(-\frac{d_s^2(x,y)}{C_2(t-s)}\big),\\ \label{Gauss-HK-lb2}
H(x,t;y,s)&\geq\frac{1}{C_1(t-s)^n}\exp\big(-\frac{C_2d_s^2(x,y)}{t-s}\big),\\ \label{gradientup3}
\frac{|\nabla_xH(x,t;y,s)|_{\omega(t)}}{H(x,t;y,s)}&\leq\frac{C_2}{(t-s)^{1/2}}\Big(1+\frac{d_s(x,y)}{(t-s)^{1/2}}\Big).
\end{align}
If in addition the type-I scalar curvature bound \eqref{typeIscalar} holds, then \eqref{Gauss-HK-ub2}--\eqref{gradientup3} hold for all $0\leq s<t<T,$ with constants $C_1,C_2$ which depend on $C_0$ in {\rm Theorem \ref{hke2}}.
\end{proposition}

\begin{proof}
The argument is divided into several steps. Let $m=2n$ and let $\Lambda,~c_0,~c_2$ be the constants of Lemma \ref{hklb}, Theorem \ref{localdistortion} and Corollary \ref{distcomparison}. We first treat the case (\ref{scalerange}). As in Step 3 of the proof of Theorem \ref{hke1} we may and do assume that $$2\Lambda^2(t-s)\leq s,$$ since otherwise $t\leq\frac{T}{2}$ and all the asserted inequalities follow from Remark \ref{bddgeom}. In the type-I case, Steps 1--5 below apply verbatim (with the type-I versions of Lemma \ref{hklb} and Corollary \ref{distcomparison}) as long as $$t-s\leq\theta(T-t)$$ and $$2\Lambda^2(t-s)\leq s,$$ where $$\theta:=\min\{c_2,\frac{1}{2\Lambda^2}\},$$ and the remaining cases are treated in Step 6.

{\bf Step 1. Uniform two sided bounds for the pointed Nash entropy.}
The pointed Nash entropy $\mathcal{N}_{x,t}(\tau)$ is nonincreasing in $\tau$ and $\mathcal{N}_{x,t}(0)=0$ (see \cite[Proposition 5.2]{Bam20}, \cite[Proposition 2.10]{HNab14}), so that $$\mathcal{N}_{x,t}(t-s)\leq0.$$ On the other hand, by \cite[Proposition 5.2]{Bam20} we have $$\frac{d}{d\tau}\big(\tau\mathcal{N}_{x,t}(\tau)\big)=\mathcal{W}_{x,t}(\tau)\geq\mu[g_{t-\tau},\tau],$$ and $$\tau\mapsto\mu[g_{t-\tau},\tau]$$ is nonincreasing by Perelman's monotonicity. Hence
\begin{align*}
\mathcal{N}_{x,t}(t-s)=\frac{1}{t-s}\int_0^{t-s}\mathcal{W}_{x,t}(\tau)d\tau\geq~&\frac{1}{t-s}\int_0^{t-s}\mu[g_{t-\tau},\tau]d\tau\\
\geq~&\mu[g_{s},t-s]\geq\mu[g_0,t]\geq\nu[g_0,T],
\end{align*}
and the right hand side is a finite constant depending only on $\omega_0$ (compare Step 1 of the proof of Theorem \ref{ZhZh22}). Hence there is a uniform constant $C_3$ with
\begin{equation}\label{nashentropy-2side}
-C_3\leq\mathcal{N}_{x,t}(t-s)\leq0
\end{equation}
for all $x\in X$ and $0\leq s<t<T.$ In particular (\ref{heatuppernash}) and Lemma \ref{scalarlbunif} give the on-diagonal upper bound
\begin{align}\label{heatkeup}
H(x,t;y,s)\leq\frac{C_1}{(t-s)^\frac{m}{2}}.
\end{align}

{\bf Step 2: A near diagonal lower bound.}
Define
$$\bar{r}:=\Lambda\sqrt{t-s}.$$
Then $$\bar{r}\leq\sqrt{s}.$$
Denote
$$B:=B(x,s,\bar{r}).$$
By Lemma \ref{hklb} (which applies since $\Lambda^2(t-s)\leq\frac{s}{2}\leq t$ and $\Lambda^2(t-s)\leq\Lambda^2c_1(T-t)^2\leq(T-t)^2$), we have
$$\int_{B}H(x,t;z,s)dg_s(z)\geq \kappa_0>0.$$
On the other hand, by Lemma \ref{uppervb} (that is, by $\mathcal{N}_{x,s}(\bar{r}^2)\leq0,$ Lemma \ref{scalarlbunif} and \cite[Theorem 8.1]{Bam20}), we have the volume upper bound
$$|B|_s\leq C_4\bar{r}^{m}=C_4\Lambda^{m}(t-s)^{n}.$$
Consequently there exists $z_1\in B$ with
\begin{equation}\label{neardiagonal-lb}
H(x,t;z_1,s)\geq\frac{\kappa_0}{|B|_s}\geq\frac{c_5}{(t-s)^{n}},~~~~d_s(x,z_1)\leq\Lambda\sqrt{t-s}.
\end{equation}

{\bf Step 3: Gaussian type upper bound.}
We apply (\ref{heatuppernash2}) with $y'\leftarrow z_1$ and use (\ref{nashentropy-2side}), (\ref{neardiagonal-lb}). This yields
$$H(x,t;y,s)\cdot\frac{c_5}{(t-s)^{n}}\leq H(x,t;y,s)H(x,t;z_1,s)\leq\frac{C_6}{(t-s)^{m}}\exp\Big(-\frac{d_s^2(y,z_1)}{9(t-s)}\Big).$$
Since $d_s(y,z_1)\geq d_s(x,y)-\Lambda\sqrt{t-s},$ we have
$$\frac{d_s^2(y,z_1)}{9(t-s)}\geq\frac{d_s^2(x,y)}{18(t-s)}-\frac{\Lambda^2}{9}.$$
Therefore, recalling $m=2n,$ we obtain
$$H(x,t;y,s)\leq\frac{C_1}{(t-s)^{n}}\exp\Big(-\frac{d_s^2(x,y)}{C_2(t-s)}\Big),$$
which is precisely (\ref{Gauss-HK-ub2}).

{\bf Step 4: Gaussian type lower bound.}
Assuming (\ref{scalerange}), we apply the type-II bound (\ref{type2-scalar}). Since $$T-\tau\geq T-t~~~\text{for}~~\tau\leq t$$ and $$\int_s^t\sqrt{t-\tau}d\tau=\frac{2}{3}(t-s)^{3/2},$$ we get
\begin{equation}\label{typeII-Lintegral}
\frac{1}{2\sqrt{t-s}}\int_s^t\sqrt{t-\tau}R(y,\tau)d\tau\leq\frac{C_{II}}{2\sqrt{t-s}(T-t)^2}\int_s^t\sqrt{t-\tau}d\tau=\frac{C_{II}(t-s)}{3(T-t)^2}\leq\frac{C_{II}c_1}{3}.
\end{equation}
Under the type-I assumption (\ref{typeIscalar}) we then have for all $0\leq s<t<T$,
\begin{equation}\label{typeI-Lintegral}
\frac{1}{2\sqrt{t-s}}\int_s^t\sqrt{t-\tau}R(y,\tau)d\tau\leq\frac{C_0}{2\sqrt{t-s}}\int_s^t\frac{d\tau}{\sqrt{t-\tau}}=C_0.
\end{equation}
In both cases Perelman's estimate (\ref{heatkelb2}) gives the on-diagonal lower bound $$H(y,t;y,s)\geq\frac{c}{(t-s)^{n}}.$$
Moreover, (\ref{heatke3}) yields the following off-diagonal lower bound. Exactly as in Lemma \ref{type1-hk-ublb} (see also \cite[Lemma 3.1]{Hall23}), we have
$$H(x,t;y,s)\ge\frac{c_3}{(t-s)^{n}}\exp\big(-2c_2\frac{d_t(x,y)^2}{t-s}\big).$$
Finally, Corollary \ref{distcomparison} applies. Indeed, in the general case,
$$t-s\leq c_1(T-t)^2\leq c_2(T-t)^2,$$
whereas in the type-I case,
$$t-s\leq \theta(T-t)\leq c_2(T-t).$$
It follows that
$$\frac{d_t^2(x,y)}{t-s}\leq\alpha^{-1}+\alpha^{-2}\frac{d^2_s(x,y)}{t-s}.$$
Hence (\ref{Gauss-HK-lb2}) follows. We stress that this step does not use Steps 1--3 nor Lemma \ref{hklb}.

{\bf Step 5: Gradient bound.}
By (\ref{graheatuppernash}), (\ref{nashentropy-2side}) and (\ref{Gauss-HK-lb2}) we obtain
$$\frac{|\nabla_xH|(x,t;y,s)}{H(x,t;y,s)}\leq\frac{C_7}{(t-s)^{\frac{1}{2}}}\sqrt{\log\Big(\frac{C_8(t-s)^{-n}}{H(x,t;y,s)}\Big)}\leq\frac{C_7}{(t-s)^{\frac{1}{2}}}\sqrt{C_9+\frac{C_9d_s^2(x,y)}{t-s}},$$
which gives (\ref{gradientup3}). This completes the proof of Proposition \ref{Gauss-heat-kernel-ub-lb} within (\ref{scalerange}), and, in the type-I case, under the conditions $$t-s\leq\theta(T-t)~~~ \text{and}~~~2\Lambda^2(t-s)\leq s.$$

{\bf Step 6: The remaining cases in the type-I case.}
Assume (\ref{typeIscalar}). If $$t-s\leq\theta(T-t)$$ but $$2\Lambda^2(t-s)>s,$$ then $$t<(1+2\Lambda^2)\theta(T-t).$$
Hence, after decreasing $\theta$ if necessary, we have $t\leq\frac{T}{2}$. All the assertions follow from Remark \ref{bddgeom}. From now on, we assume that
$$t-s>\theta(T-t).$$
If $t\leq\frac{7T}{8},$ then all the assertions follow again from Remark \ref{bddgeom}, with $T'=\frac{7T}{8}.$ If $s\geq\frac{3T}{4},$ then the two-sided bound consisting of (\ref{Gauss-HK-ub2}) and (\ref{Gauss-HK-lb2}) is precisely the content of Lemma \ref{type1-hk-ublb}, while (\ref{gradientup3}) follows from (\ref{Gauss-HK-lb2}) as in Step 5. It remains to consider the case
$$s<\frac{3T}{4}<\frac{7T}{8}<t.$$

The upper bound (\ref{Gauss-HK-ub2}) then follows from the on-diagonal bound (\ref{heatkeup}) and the diameter bound (\ref{diameterbd}), which shows that the Gaussian factor is bounded from below by $\exp(-\frac{8C_1^2}{C_2T}).$

For the lower bound we use the semigroup property of the heat kernel with the intermediate time $b:=\frac{7T}{8}.$ Since the flow is smooth on $[0,b],$ Remark \ref{bddgeom} together with the diameter bound yields
$$H(z,b;y,s)\geq c>0$$
for all $z,y\in X$ and $0\leq s\leq\frac{3T}{4}$ (note $b-s\geq\frac{T}{8}$). Therefore, using $$\int_XH(x,t;z,b)dg_b(z)=1,$$
we have
$$H(x,t;y,s)=\int_XH(x,t;z,b)H(z,b;y,s)dg_b(z)\geq c\int_XH(x,t;z,b)dg_b(z)=c.$$
Since $t-s\geq T/8$ and $\exp(-C_2d_s^2(x,y)/(t-s))\leq1$, we obtain
$$H(x,t;y,s)\geq c\geq\frac{1}{C_1(t-s)^n}\exp\Big(-\frac{C_2d_s^2(x,y)}{t-s}\Big)$$
after taking $C_1\geq c^{-1}(8/T)^n$. Moreover, (\ref{gradientup3}) follows once again from Step 5. This completes the proof of Proposition \ref{Gauss-heat-kernel-ub-lb}.
\end{proof}

\textbf{Proof~of~}Theorem \ref{hke2}: The two sided Gaussian bound (\ref{Gaussuplow2}) and the gradient ratio bound (\ref{gradientup2}) are the type-I case of Proposition \ref{Gauss-heat-kernel-ub-lb}. \qed

\begin{remark}\label{remark-alternative-hklb}
Assume in addition that the type--I bound \eqref{typeIscalar} holds. Then, for all \(0\le l<t<T\) with \(t-l\le l\), the integral lower bound of Lemma \ref{hklb} holds at radius \(\sqrt{t-l}\). Moreover, this bound follows from the Gaussian lower bound \eqref{Gauss-HK-lb2} together with Bamler's volume noncollapsing estimate \cite[Theorem 6.1]{Bam20}, without invoking the cut-off function of Theorem \ref{cutoffftn}. Indeed, let
$$r:=\sqrt{t-l}~~~ \text{and}~~~r':=r/\sqrt{C_0},$$
where we may assume $C_0\geq1.$ For $\tau\in[l-r'^2,l]$ we have
$$T-\tau\geq T-l\geq t-l.$$
Hence
$$R\leq\frac{C_0}{T-\tau}\leq\frac{C_0}{t-l}=r'^{-2}~~~~\text{on}~~X\times[l-r'^2,l],$$
which is the scalar curvature hypothesis of \cite[Theorem 6.1]{Bam20} at the point $(x,l)$ and at the scale $r'.$ Since the pointed Nash entropy is bounded from below by \eqref{nashentropy-2side}, that theorem gives
\begin{equation}\label{volume-noncollapse}
|B(x,l,\sqrt{t-l})|_{l}\geq|B(x,l,r')|_l\geq c\exp\big(\mathcal{N}_{x,l}(r'^2)\big)r'^{2n}\geq c'(t-l)^{n}.
\end{equation}
Combining \eqref{volume-noncollapse} with \eqref{Gauss-HK-lb2}, applied to the pair $(l,t)$ and to the points $z\in B(x,l,\sqrt{t-l}),$ for which $d_l(x,z)^2\leq t-l,$ we have
\begin{equation}\label{heat-kernel-integral-lb}
\int_{B(x,l,\sqrt{t-l})}H(x,t;z,l)dg_{l}(z)\ge c'(t-l)^n\cdot\frac{1}{C_1(t-l)^n}e^{-C_2}=:c''>0.
\end{equation}
We emphasize that the radius must be the varying quantity $\sqrt{t-l}$, and that the volume estimate is applied at each time $l$ separately. 
\end{remark}

\section{Distance Distortion Estimates II}\label{sec7}

In this section, we study the behavior of the distance function along the K\"ahler-Ricci flow under the assumption of Theorem \ref{localdistortion}. We will extend the distance distortion estimates of Section \ref{sec3} to a uniform statement on small time intervals following Bamler-Zhang \cite{BamZh19}. We also discuss what is needed in order to extend the distance function continuously up to the singular time.

\subsection{Lower bound on the distance distortion}
We will follow the argument of Bamler-Zhang \cite{BamZh19} to obtain the lower bound on the distance distortion. At first, we obtain the upper volume bound in the following form.

\begin{lemma}\label{uppervb}
Under the assumption of {\rm Theorem \ref{localdistortion}}, there exists a uniform constant $A$ such that for any $(x,t)\in X\times[0,T)$ and any $r>0$ we have
\begin{equation}\label{6.1}
|B(x,t,r)|_t\leq Ar^{2n}\exp(Ar).
\end{equation}
\end{lemma}
\begin{proof}
Recall first that the pointed Nash entropy satisfies $$\mathcal{N}_{x,t}(\tau)\leq\mathcal{N}_{x,t}(0)=0$$
(cf. \cite[Proposition 5.2]{Bam20}). Hence, by Lemma \ref{scalarlbunif} and Bamler's noninflating estimate in \cite[Theorem 8.1]{Bam20}, we have, whenever $r^2\leq t,$
\begin{align*}
|B(x,t,r)|_t\leq C(n,R_0T)\exp\big(\mathcal{N}_{x,t}(r^2)\big)r^{2n}\leq C_1r^{2n}.
\end{align*}
If $r^2>t$ and $t\leq\frac{T}{2},$ the same bound follows from Remark \ref{bddgeom}, since the metrics $\omega(t),~t\in[0,\frac{T}{2}],$ are uniformly equivalent to $\omega_0$. If $r^2>t$ and $t\geq\frac{T}{2},$ then $r^2\geq\frac{T}{2}$ and  $$|B(x,t,r)|_t\leq\frac{[\omega(t)^n]}{n!}\leq C\leq C(\frac{2}{T})^nr^{2n}.$$
From above, we have for any $x\in X$ and $0\leq r\leq1$
\begin{equation}\label{6.2}
|B(x,t,r)|_t\leq C_1r^{2n}.
\end{equation}

Next we adopt the argument of Lemma 2.1 in Bamler-Zhang \cite{BamZh19} to obtain our desired estimates (\ref{6.1}). Fix some $x\in X$ and let $N<\infty$ be the maximal with the property that we can find points $x_1,\cdots,x_N\in B(x,t,\frac{1}{2})$ such that the balls $B(x_1,t,\frac{1}{8}),\cdots,B(x_N,t,\frac{1}{8})$ are pairwise disjoint. Note that then
$$B(x_1,t,\frac{1}{8}),\cdots,B(x_N,t,\frac{1}{8})\subset B(x,t,1).$$
So, by (\ref{6.2}), we have $N\leq C_*:=(C_0(\frac{1}{8})^{\theta})^{-1}C_1.$ Moreover, by the maximality of $N$, we have
\begin{equation}\label{6.3}
B(x,t,\frac{1}{2})\subset B(x_1,t,\frac{1}{4})\cup\cdots\cup B(x_N,t,\frac{1}{4}).
\end{equation}

We now argue that for all $r\geq\frac{1}{2}$
\begin{equation}\label{6.4}
B(x,t,r+\frac{1}{4})\subset B(x_1,t,r)\cup\cdots\cup B(x_N,t,r).
\end{equation}

Let $y\in B(x,t,r+\frac{1}{4})$ and consider a $g(t)$-minimal geodesic $\gamma:[0,l]\to X$ connect $x$ and $y$ that is parameterized by arclength. Then $l<r+\frac{1}{4}.$ By (\ref{6.3}) we may choose $i\in\{1,\cdots,N\}$ such that $\gamma(\frac{1}{2})\in\overline{B(x_i,t,\frac{1}{4})}.$ Then by triangle inequality
$$d_t(x_i,y)\leq(l-\frac{1}{2})+d_t(\gamma(\frac{1}{2}),x_i)\leq l-\frac{1}{4}<r.$$
So $y\in B(x_i,t,r),$ which confirms (\ref{6.4}).

Let us now prove by induction on $k=1,2,\cdots$ that for any $x\in X$
\begin{equation}\label{6.5}
|B(x,t,\frac{1}{4}k)|_t<C_*^k.
\end{equation}
For $k=1$, the inequality follows from (\ref{6.2}). If the inequality is true for $k$, then we can use (\ref{6.4}) to conclude
\begin{align*}
|B(x,t,\frac{1}{4}(k+1))|_t&\leq|B(x_1,t,\frac{1}{4}k)|_t+\cdots+|B(x_N,t,\frac{1}{4}k)|_t\\
&\leq N\cdot C_*^k\leq C_*^{k+1}.
\end{align*}
So (\ref{6.5}) holds for $k+1$. This confirms (\ref{6.5}).

The assertion of the lemma now follow from (\ref{6.2}) for $r<1.$ For $r\geq1$ choose $k\in\mathbb{N}$ such that $\frac{1}{4}(k-1)\leq r\leq\frac{k}{4}.$ Then, by (\ref{6.5}), we have
$$|B(x,t,r)|_t<|B(x,t,\frac{k}{4})|_t\leq C_*^k=C_*\exp((k-1)\log C_*)\leq C_*\exp(4(\log C_*)r).$$
This proves (\ref{6.1}).
\end{proof}

\begin{lemma}\label{dis2lb}
Under the assumption of {\rm Theorem \ref{localdistortion}}, there exists a uniform constant $C$ such that for all $0\leq s<t<T$ satisfying \eqref{scalerange} and all $x,y\in X$ we have the distance bound
\begin{align}\label{7.6}
d_s(x,y)-C\sqrt{t-s}\leq d_t(x,y).
\end{align}
Under the type-I assumption \eqref{typeIscalar}, \eqref{7.6} holds for all $0\leq s<t<T,$ with $C$ depending on $C_0$ in {\rm Theorem \ref{hke2}}.
\end{lemma}
\begin{proof}
Set $d:=d_s(x,y)$ and let $u\in C^0(X\times[s,t])\cap C^{\infty}(X\times(s,t])$ be a solution to the heat equation
$$\partial_tu=\Delta u,~~u(\cdot,s)=d_s(x,\cdot).$$
Then for any $(z,\tau)\in X\times[s,t]$
$$u(z,\tau)=\int_XH(z,\tau;\cdot,s)u(\cdot,s)dg(s)=\int_XH(z,\tau;\cdot,s)d_s(x,\cdot)dg(s).$$
Applying the Gaussian upper bound of Theorem \ref{hke1} (resp. of Theorem \ref{hke2} in the type-I case) and (\ref{6.1}), we find that
\begin{align*}
u(x,t)&\leq~\int_X\frac{C_1}{(t-s)^{\frac{m}{2}}}\exp\big(-\frac{d_s^2(x,z)}{C_2(t-s)}\big)d_s(x,z)dg(s)\\
&=~\sum_{k=-\infty}^{\infty}\int_{B(x,s,2^k)\setminus B(x,s,2^{k-1})}\frac{C_1}{(t-s)^{\frac{m}{2}}}\exp\big(-\frac{d_s^2(x,z)}{C_2(t-s)}\big)d_s(x,z)dg(s)\\
&\leq~\sum_{k=-\infty}^{\infty}|B(x,s,2^k)|_s\frac{C_1}{(t-s)^{\frac{m}{2}}}\exp\big(-\frac{2^{2k-2}}{C_2(t-s)}\big)\cdot2^k\\
&\leq~\sum_{k=-\infty}^{\infty}AC_1\frac{(2^k)^{m+1}}{(t-s)^{\frac{m}{2}}}\exp\big(-\frac{2^{2k}}{4C_2(t-s)}\big)
=~AC_1\sqrt{t-s}\sum_{k=-\infty}^{\infty}\rho_k^{m+1}e^{-\frac{\rho_k^2}{4C_2}},\\
&\leq~C_3\sqrt{t-s},
\end{align*}
where $\rho_k:=\frac{2^k}{\sqrt{t-s}}$ and $\rho_{k+1}=2\rho_k$. The last sum is a dyadic sum of the integrable function $$\rho\mapsto\rho^{m+1}e^{-\rho^2/4C_2},$$ which is bounded independently of $t-s.$

On the other hand, using $\int_XH(y,t;z,s)dg_s(z)=1,$ we have
\begin{align*}
|d-u(y,t)|&~=~\big|\int_XH(y,t;\cdot,s)(d-d_s(x,\cdot))dg(s)\big|\\
&~\leq~\int_XH(y,t;\cdot,s)|d_s(x,y)-d_s(x,\cdot)|dg(s)\\
&~\leq~\int_XH(y,t;\cdot,s)d_s(y,\cdot)dg(s).
\end{align*}
So similarly,
$$|d-u(y,t)|\leq~ C_4\sqrt{t-s}.$$
It follows that
\begin{equation}\label{6.6}
|u(x,t)-u(y,t)|\geq ~d-(C_3+C_4)\sqrt{t-s}\geq~ d-C_5\sqrt{t-s}.
\end{equation}
By \begin{equation}\label{4.5}
\Delta|\nabla u|^2-\partial_t|\nabla u|^2=|\nabla\nabla u|^2+|\nabla\bar{\nabla}u|^2\geq0,
\end{equation}
we have
\begin{equation}\label{4.6}
\partial_t|\nabla u|\leq\Delta|\nabla u|.
\end{equation}
The initial distance function is $1$-Lipschitz. We may approximate it uniformly by smooth functions with gradient bounded by $1+\varepsilon$, apply the maximum principle to their heat evolutions, and let $\varepsilon\to0$. Consequently $|\nabla u|\leq1$ on $X\times(s,t]$ (see also \cite[Lemma 2.5]{Bam20}). So
$$|u(x,t)-u(y,t)|\leq ~d_t(x,y).$$
Together with (\ref{6.6}) this gives us
$$d_s(x,y)-C_5\sqrt{t-s}\leq ~d_t(x,y).$$
This completes the proof.
\end{proof}

\subsection{Upper bound on the distance distortion}

In this subsection, we consider the converse inequality $$d_t\leq d_s+C\sqrt{t-s}.$$
The argument of Lemma \ref{dis2lb} does not apply directly, since the evolution equation (\ref{4.6}) for $|\nabla u|$ acquires an extra Ricci curvature term when it is applied to the backward heat equation. Hallgren-Jian-Song-Tian \cite[Lemma 3.9, Proposition 3.10]{HJST23} overcome this difficulty on the type-I rescaled flow by an integral gradient estimate. Its weight combines the forward heat kernel with the Ricci potential $v,$ using the Li-Yau inequality of Theorem \ref{LiYauthm} and the volume bound (\ref{6.1}). In the unrescaled flow, however, the situation changes. Both the Li--Yau inequality (\ref{liyau-ineq}) and the estimate
$$|\nabla v|^2\leq\frac{Cv}{1-t}$$
blow up as $t\to1$ (see Remark \ref{liyauremark}). As a result, this argument gives only constants that depend on the initial time $t_0$, and these constants need not remain bounded as $t_0\to1$. Since the constants are allowed to degenerate anyway, we give here short proofs of the resulting local statements, which rely only on the smoothness of the flow before the singular time. We keep the formulation of \cite{HJST23}, since it is precisely this form of the estimate that would be needed, with uniform constants, to obtain the upper inequality \eqref{impdde-upper} up to the singular time; see Remark \ref{remark-holder}.

\begin{lemma}\label{graheatub}
Under the assumption of Theorem {\rm \ref{localdistortion}} and with $T=1,$ for any $t_0\in[0,1)$ there exists $C=C(t_0)<\infty,$ such that the following statement holds.

Let $t_0<t_1<1$ with $1-t_0\leq2(1-t_1),$ and let $u\in C^{\infty}(X\times[t_0,t_1])$ be a solution of the backward heat equation $(\frac{\partial}{\partial t}+\Delta)u=0,$ with $|\nabla u(\cdot,t_1)|\leq1.$ Then
\begin{equation}\label{graheatub1}
|\nabla u|^2(y,t_0)\leq\exp\big(C\sqrt{t_1-t_0}\big)
\end{equation}
for all $y\in X.$ In particular
$$|\nabla u|^2(y,t_0)\leq\exp\big(C\sqrt{t_1-t_0}\cdot(D^2+1)\big)$$
for every $D\geq0$.
\end{lemma}
\begin{proof}
By assumption $t_1\leq\frac{1+t_0}{2}<1,$ so that the flow is smooth on the compact time interval $[t_0,\frac{1+t_0}{2}],$ and there is a constant $K=K(t_0)<\infty$ with $$Ric(\omega(t))\geq-K\omega(t)~~~\forall ~t\in[t_0,\frac{1+t_0}{2}].$$
Set $$U(x,\sigma):=u(x,t_1-\sigma)~~~ \text{and} ~~~\omega_{\sigma}:=\omega(t_1-\sigma)$$ for $\sigma\in[0,t_1-t_0].$ Then
$$\partial_{\sigma}U=\Delta_{\omega_{\sigma}}U~~~ \text{and} ~~~\partial_{\sigma}\omega_{\sigma}=Ric(\omega_{\sigma}),$$
and the Bochner formula gives, with respect to $\omega_{\sigma},$
$$(\partial_{\sigma}-\Delta)|\nabla U|^2=-|\nabla\nabla U|^2-|\nabla\bar{\nabla}U|^2-2Ric(\nabla U,\bar{\nabla}U)\leq2K|\nabla U|^2.$$
Since $|\nabla U|^2(\cdot,0)\leq1,$ the maximum principle gives $|\nabla U|^2(\cdot,\sigma)\leq e^{2K\sigma},$ that is
$$|\nabla u|^2(\cdot,t_0)\leq e^{2K(t_1-t_0)}\leq e^{2K\sqrt{t_1-t_0}},$$
since $t_1-t_0\leq1.$ This completes the proof.
\end{proof}

\begin{lemma}\label{dis2ub}
Under the assumption of {\rm Theorem \ref{localdistortion}}, for every $s\in[0,T)$ there exist constants $\theta=\theta(s)>0$ and $C=C(s)<\infty$ such that for all $t\in(s,\min\{s+\theta,T\})$ and all $x,y\in X$ we have the distance bound
\begin{align}\label{7.11}
d_s(x,y)+C\sqrt{t-s}\geq d_t(x,y).
\end{align}
\end{lemma}
\begin{proof}
Fix $s\in[0,T)$ and let $$\theta=\theta(s):=\frac{T-s}{2}.$$
Then the flow is smooth on $[s,s+\theta]=[s,\frac{T+s}{2}]$. Let $K_s<\infty$ be such that $$Ric(\omega(t))\geq-K_s\omega(t)~~~\forall ~t\in[s,s+\theta].$$ Since $$\partial_t\omega=-Ric(\omega)\leq K_s\omega,$$ we have
$$\omega(t)\leq e^{K_s(t-s)}\omega(s),$$
and hence $$d_t(x,y)\leq e^{K_s(t-s)/2}d_s(x,y)~~~\forall ~t\in[s,s+\theta].$$ Together with the diameter bound (\ref{diameterbd}) and the elementary inequality $e^{a}-1\leq ae^{a}$ this gives
$$d_t(x,y)-d_s(x,y)\leq\big(e^{K_s(t-s)/2}-1\big)diam(X,\omega(s))\leq C_1K_se^{K_s\theta}(t-s)\leq C(s)\sqrt{t-s},$$
where we used $t-s\leq\theta\leq T.$ This completes the proof.
\end{proof}

\begin{remark}\label{remark-holder}
Combining Lemma \ref{dis2lb} with Lemma \ref{dis2ub}, we obtain \eqref{impdde} on every time interval $[s,t]$ with $t-s\le\theta(s)$ for which \eqref{scalerange} holds. Under
\eqref{typeIscalar}, the same conclusion holds on every interval with $t-s\le\theta(s)$. So, the distance function is locally $\frac{1}{2}$-H\"older continuous in time. The constant in the lower bound (\ref{7.6}) is uniform, whereas $C(s)$ is not asserted to remain bounded as $s\to T.$

Under (\ref{typeIscalar}), the inequality (\ref{7.6}) holds for all $s<t.$
For fixed $x,y\in X$ the bounded function $t\mapsto d_t(x,y)$ satisfies
$$\liminf_{t\to T}d_t(x,y)\geq d_{s_j}(x,y)-C\sqrt{T-s_j}$$
for every sequence $s_j\nearrow T.$ Choosing $s_j$ so that
$$d_{s_j}(x,y)\to\limsup_{s\to T}d_s(x,y),$$
we see that the limit $$d_T(x,y):=\lim_{t\to T}d_t(x,y)$$
exists for every pair $(x,y),$ which defines a pseudo-distance on $X.$ A uniform upper bound $$d_t\leq d_s+C\sqrt{t-s}\qquad (0\le s<t<T,\ x,y\in X),$$
with $C$ independent of $s,t,x,y$, would imply the uniform convergence $d_t\to d_T$ on $X\times X$, as well as the continuity of $d_T$ in the topology of $X$. Such a uniform upper bound would therefore give the Gromov--Hausdorff description of the limit as in \cite{Bam18,BamZh19}. The argument of \cite{HJST23} would require a version of Lemma \ref{graheatub} in which the Ricci potential $v$ is replaced by a barrier which stays bounded along the flow, in the sense of \cite[Definition 4.1]{HJST23}. In Section \ref{sec8.3} we obtain uniform convergence by a different argument, using the singular-time conjugate heat kernels. \qed
\end{remark}


\section{The uniqueness of Gromov-Hausdorff limit and Type I blow-up limits}\label{sec8}

In this section, we present the consequences of the recent Type I scalar curvature estimate of Jian--Song \cite{JS26} for the results of the previous sections. Throughout this section $m=2n$, the constants $C_1,C_2$ are those of Theorem \ref{hke2}, and $(X,\omega(t))$ satisfies the assumption of Theorem \ref{localdistortion}, i.e. $\int_X\beta^n>0.$ In particular the diameter bound (\ref{diameterbd}) and the volume lower bound (\ref{noncollapse2}) are available. We write $X^{\circ}:=X\setminus\mathrm{Null}(\alpha)$, where
$$\alpha=[\beta]=[\omega_0]-2\pi Tc_1(X)$$
and $\mathrm{Null}(\alpha)$ is the null locus of $\alpha$ \cite{CT15}.

\subsection{The type-I scalar curvature bound of Jian-Song}\label{sec8.1}
\begin{theorem}{\rm(Theorem 1.1 in \cite{JS26})}\label{JSthm}
Let $(X,\omega(t))$ be a maximal solution of \eqref{KRF} on a compact K\"ahler manifold with $T<\infty.$ Then there is a constant $C_0=C_0(n,\omega_0)$ such that \eqref{typeIscalar} holds, that is
$$|R(x,t)|\leq\frac{C_0}{T-t}~~~~\text{on}~~X\times[0,T).$$
\end{theorem}

In the proof of Theorem \ref{JSthm}, the transcendental base-point-free theorem of Hacon--Xie \cite{HX26} is used in \cite[Theorem 2.1]{JS26} to represent the limiting class $\alpha$ by a smooth semipositive form $\theta=\Phi^*\theta_Y$, obtained by pulling back a K\"ahler form from a compact normal K\"ahler space $Y$. Under the assumption (\ref{omega0-omegay}) of Section \ref{sec3}, such a representative is already given by $\theta=\pi^*\omega_Y$, so this use of \cite{HX26} is unnecessary. The same independence holds when $X$ is projective and $[\omega_0]$ is rational. The rationality and base-point-free theorems give the required pullback representation over a normal projective variety, which may be singular. This latter case need not satisfy (\ref{omega0-omegay}) with a smooth target. In the general compact K\"ahler case, our use of Theorem \ref{JSthm} retains its dependence on \cite{HX26}. We shall use Theorem \ref{JSthm} as a black box.

\begin{corollary}\label{unconditional}
Under the assumption of {\rm Theorem \ref{localdistortion}}, the hypothesis \eqref{typeIscalar} is satisfied with $C_0=C_0(n,\omega_0).$ Consequently the type-I versions of {\rm Theorem \ref{localdistortion}} and {\rm Theorem \ref{cutoffftn}}, of {\rm Corollary \ref{distcomparison}}, of {\rm Lemma \ref{hklb}}, of {\rm Remark \ref{remark-alternative-hklb}} and of {\rm Lemma \ref{dis2lb}}, as well as {\rm Theorem \ref{hke2}}, hold unconditionally, with constants depending only on $n,\omega_0.$ In particular, for all $0\leq s<t<T$ and $x,y\in X,$
\begin{align}\label{8-gauss}
\frac{1}{C_1(t-s)^n}\exp\Big(-\frac{C_2d_s^2(x,y)}{t-s}\Big)\leq H(x,t;y,s)\leq\frac{C_1}{(t-s)^n}\exp\Big(-\frac{d_s^2(x,y)}{C_2(t-s)}\Big),
\end{align}
and
\begin{align}\label{8-dislow}
d_s(x,y)-C\sqrt{t-s}\leq d_t(x,y).
\end{align}
\end{corollary}
\begin{proof}
Note that the lower bound $R\geq-R_0$ of Lemma \ref{scalarlbunif} gives $$|R|\leq\max\{C_0,R_0T\}/(T-t).$$
This is Theorem \ref{JSthm} combined with the type-I statements of Sections \ref{sec3}--\ref{sec7}.
\end{proof}

\subsection{Conjugate heat kernels based at the singular time}\label{sec8.2}

We use the following terminology of Bamler \cite{Bam20,Bam20b,Bam20c}. Let $(X,g_t)$ be a Ricci flow on a compact Riemannian manifold of dimension $m$ which develops a singularity at time $T<\infty$.

A {\it conjugate heat flow} on $X\times[0,T)$ is a family of probability measures
$$d\mu_s=u(\cdot,s)dg_s,~~~s\in[0,T),$$
where $u>0$ is a solution of the conjugate heat equation (\ref{cheateqa}). For $(x,t)\in X\times(0,T)$ the measures $$d\nu_{x,t;s}=H(x,t;\cdot,s)dg_s,~~~s\in[0,t),$$
form a conjugate heat flow (which we extend trivially, as $\delta_x,$ to $s=t$). For a probability measure $\mu$ on $(X,d_s)$ we write
$$\mathrm{Var}_s(\mu):=\int_X\int_Xd_s^2(y,z)\,d\mu(y)d\mu(z),$$
and $d^{(s)}_{W_1}$ for the $W_1$-Wasserstein distance of $(X,d_s).$
Define
$$H_m:=\frac{(m-1)\pi^2}{2}+4.$$
Recall two facts from \cite{Bam20}: for any two conjugate heat flows $\mu^1,\mu^2$ the function
\begin{equation}\label{8-W1mono}
s\longmapsto d^{(s)}_{W_1}(\mu^1_s,\mu^2_s)~~~\text{is nondecreasing on}~[0,T)
\end{equation}
\cite[Lemma 2.7]{Bam20}, and
\begin{equation}\label{8-varmono}
s\longmapsto\mathrm{Var}_s(\mu_s)+H_m s~~~\text{is nondecreasing on}~[0,T),
\end{equation}
\cite[Corollary 3.7]{Bam20}. Following \cite[Definition 2.32]{Bam20c}, a conjugate heat flow $\mu_s$ is said to be {\it based at the singular time} if
$$\lim_{s\nearrow T}\mathrm{Var}_s(\mu_s)=0.$$
By (\ref{8-varmono}), we have
$$\mathrm{Var}_s(\mu_s)\leq H_m(T-s).$$
The set of these flows is denoted by $\mathcal{M}_T,$ and it carries the distance
$$d_{\mathcal{M}_T}(\mu^1,\mu^2):=\lim_{s\nearrow T}d^{(s)}_{W_1}(\mu^1_s,\mu^2_s).$$
The function under the limit is nondecreasing by (\ref{8-W1mono}), and the limit is finite by (\ref{diameterbd}). Note that
$$d^{(s)}_{W_1}(\mu^1_s,\mu^2_s)\leq d_{\mathcal{M}_T}(\mu^1,\mu^2)~~~~\text{for~all}~s.$$
Thus, $d_{\mathcal{M}_T}(\mu^1,\mu^2)=0$ implies $\mu^1=\mu^2$. So, $d_{\mathcal{M}_T}$ is a distance. Lemma 2.35 in \cite{Bam20c} shows that $(\mathcal{M}_T,d_{\mathcal{M}_T})$ can be viewed as the singular time-$T$ slice of the flow $(X,g_t)$.

By \cite[Lemma 9.1]{FL25}, after translating the singular time $T$ to $0$, $(\mathcal{M}_T,d_{\mathcal{M}_T})$ is isometric to the time-zero slice $(Z_0,d_0^Z)$ of the spacetime completion (\cite[Theorem 2.7]{LZ26}).

The following lemma is the bridge between the heat kernel estimates of Theorem \ref{hke2} and the singular time. It shows that, under the type-I scalar bound, the conjugate heat kernel based at the singular time at a point $x$ is unique, and that it is localized around $x$ at every time $s\in[0,T)$, at the scale $\sqrt{T-s}.$


\begin{lemma}{\rm(Localization and uniqueness)}\label{localization}
Under the assumption of {\rm Theorem \ref{localdistortion}}, there are constants $C_3,C_4$ depending only on $n,\omega_0$ such that the following holds.
\begin{itemize}
\item[$(a)$] For every $x\in X$ the limit
$$u^{(x)}(y,s):=\lim_{t\nearrow T}H(x,t;y,s)$$
exists in $C^{\infty}_{loc}(X\times[0,T)),$ and
$$d\mu^{(x)}_s:=u^{(x)}(\cdot,s)dg_s$$
is a conjugate heat flow based at the singular time. More generally $$H(x_j,t_j;\cdot,\cdot)\longrightarrow u^{(x)}~~\text{in}~~C^{\infty}_{loc}(X\times[0,T))$$
whenever $x_j\to x$ in $X$ and $t_j\to T.$
\item[$(b)$] For all $(y,s)\in X\times[0,T),$
\begin{align}\label{8-gausssing}
\frac{1}{C_1(T-s)^n}\exp\Big(-\frac{C_2d_s^2(x,y)}{T-s}\Big)\leq u^{(x)}(y,s)\leq\frac{C_1}{(T-s)^n}\exp\Big(-\frac{d_s^2(x,y)}{C_2(T-s)}\Big).
\end{align}
\item[$(c)$] For all $s\in[0,T),$
\begin{align}\label{8-moment}
d^{(s)}_{W_1}\big(\mu^{(x)}_s,\delta_x\big)=\int_Xd_s(x,y)\,d\mu^{(x)}_s(y)\leq C_3\sqrt{T-s},~~~~\mathrm{Var}_s\big(\mu^{(x)}_s\big)\leq C_4(T-s).
\end{align}
\item[$(d)$] Every $\mu\in\mathcal{M}_T$ is of the form $\mu^{(y)}$ for some $y\in X.$ Moreover $\mu^{(x)}=\mu^{(y)}$ if and only if $d_{\mathcal{M}_T}(\mu^{(x)},\mu^{(y)})=0.$
\end{itemize}
\end{lemma}
\begin{proof}
Let $x_j\to x$ and $t_j\to T.$ By (\ref{8-gauss}), the functions $H(x_j,t_j;\cdot,\cdot)$ are uniformly bounded on $X\times[0,T']$ for every $T'<T$ (for $j$ large) and satisfy the conjugate heat equation (\ref{cheateqa}), whose coefficients are smooth on $X\times[0,T'].$ By parabolic regularity, up to a subsequence, $H(x_j,t_j;\cdot,\cdot)$ converges in $C^{\infty}_{loc}(X\times[0,T))$ to a positive solution $u$ satisfying (\ref{cheateqa}) with $$\int_Xu(\cdot,s)dg_s=1~~~~ \text{for~all}~~ s.$$
Passing to the limit in (\ref{8-gauss}), with $t=t_j,$ and using $d_s(x_j,y)\to d_s(x,y)$ for fixed $s,$ we see that $u$ satisfies (\ref{8-gausssing}).

Next let $u$ be any such limit and $d\mu_s=u(\cdot,s)dg_s.$ By the upper bound in (\ref{8-gausssing}) and the volume bound (\ref{6.1}) of Lemma \ref{uppervb}, the same computation in the proof of Lemma \ref{dis2lb} gives, for $k=1,2,$
\begin{align*}
\int_Xd_s^k(x,y)\,d\mu_s(y)&\leq\sum_{j=-\infty}^{\infty}|B(x,s,2^j)|_s\frac{C_1}{(T-s)^n}\exp\Big(-\frac{2^{2j-2}}{C_2(T-s)}\Big)2^{jk}\\
&\leq C(T-s)^{k/2}\sum_{j=-\infty}^{\infty}\rho_j^{m+k}e^{-\rho_j^2/4C_2}\leq C(T-s)^{k/2},
\end{align*}
where $\rho_j:=\frac{2^j}{\sqrt{T-s}}$. The case $k=1$ is the first inequality in (\ref{8-moment}), and the case $k=2$ gives $$\mathrm{Var}_s(\mu_s)\leq4\int_Xd_s^2(x,y)d\mu_s(y)\leq C_4(T-s).$$
Hence $\mu\in\mathcal{M}_T.$ This proves $(b),(c)$ for every subsequential limit.

Now let $u,u'$ be two subsequential limits (possibly along different sequences $(x_j,t_j),(x'_j,t'_j)$ with the same limit point $x$), with associated flows $\mu,\mu'.$ By (\ref{8-moment}) and the triangle inequality for $d^{(s)}_{W_1},$ we have
$$d^{(s)}_{W_1}(\mu_s,\mu'_s)\leq d^{(s)}_{W_1}(\mu_s,\delta_x)+d^{(s)}_{W_1}(\delta_x,\mu'_s)\leq2C_3\sqrt{T-s}.$$
By (\ref{8-W1mono}), we have that $d^{(s)}_{W_1}(\mu_s,\mu'_s)$ is nondecreasing in $s.$ Hence for every fixed $s_0\in[0,T)$ and all $s\in[s_0,T)$, we obtain $$d^{(s_0)}_{W_1}(\mu_{s_0},\mu'_{s_0})\leq2C_3\sqrt{T-s}\to0.$$
Letting \(s\to T\), we obtain \(\mu_{s_0}=\mu'_{s_0}\). Therefore all subsequential limits coincide. This proves $(a)$.

For $(d),$ let $\mu\in\mathcal{M}_T$ and choose $s_i\nearrow T.$ Let \(z_i\in X\) be a minimizer of
$$z\mapsto \int_X d_{s_i}^2(z,y)\,d\mu_{s_i}(y).$$
Averaging over \(z\) with respect to \(\mu_{s_i}\), we obtain
$$\int_Xd_{s_i}^2(z_i,y)d\mu_{s_i}(y)\leq\mathrm{Var}_{s_i}(\mu_{s_i})\to0.$$
After passing to a subsequence, we may assume that $z_i\to y\in X.$ For $s\leq s_i$ the reproduction formula gives
$$\mu_s=\int_X\nu_{w,s_i;s}\,d\mu_{s_i}(w).$$
By the convexity of the \(W_1\)-distance under convex combinations of
measures, together with this representation,
\[
d^{(s)}_{W_1}\big(\mu_s,\nu_{z_i,s_i;s}\big)
\le
\int_X d^{(s)}_{W_1}\big(\nu_{w,s_i;s},\nu_{z_i,s_i;s}\big)
\,d\mu_{s_i}(w).
\]
By \eqref{8-W1mono}, the map
\[
s\mapsto d^{(s)}_{W_1}(\nu_{w,s_i},\nu_{z_i,s_i})
\]
is nondecreasing. At \(s=s_i\) its value equals \(d_{s_i}(w,z_i)\).
Hence the right-hand side is bounded by
\[
\int_X d_{s_i}(w,z_i)\,d\mu_{s_i}(w).
\]
By the Cauchy--Schwarz inequality and the minimizing property of
\(z_i\),
\[
\int_X d_{s_i}(w,z_i)\,d\mu_{s_i}(w)
\le
\Big(\int_X d_{s_i}^2(w,z_i)\,d\mu_{s_i}(w)\Big)^{1/2}
\le
\mathrm{Var}_{s_i}(\mu_{s_i})^{1/2}.
\]
Combining the preceding inequalities gives
\[
d^{(s)}_{W_1}\big(\mu_s,\nu_{z_i,s_i;s}\big)
\le
\mathrm{Var}_{s_i}(\mu_{s_i})^{1/2}.
\]
Letting \(i\to\infty\) at fixed \(s\), and invoking assertion (a) with
\(x_j=z_i\) and \(t_j=s_i\), we obtain
\[
\mu_s=\mu^{(y)}_s.
\]
Thus \(\mu=\mu^{(y)}\), which establishes the representation assertion
in (d) and completes the proof.
\end{proof}

\subsection{The singular time slice}\label{sec8.3}

In this section, we will show that the pointwise limit \(d_T=\lim_{t\to T}d_t\) exists, that the convergence is uniform, and that the Gromov--Hausdorff limit of $(X,d_t)$ is unique. The following theorem strengthens Remark \ref{remark-holder}. The uniqueness part is the first assertion of Li-Zhang \cite[Theorem 1.1]{LZ26}. Our proof below is different, and it also gives the uniform convergence of the distance functions.

\begin{theorem}{\rm(The singular time slice)}\label{thmB}
Under the assumption of {\rm Theorem \ref{localdistortion}}, the following hold.
\begin{itemize}
\item[$(a)$] For all $x,y\in X$ the limit $$d_T(x,y):=\lim_{t\to T}d_t(x,y)$$ exists, and
$$\sup_{x,y\in X}|d_t(x,y)-d_T(x,y)|\longrightarrow0~~~~\text{as}~t\to T.$$
\item[$(b)$] Let $x\sim y$ if and only if $d_T(x,y)=0.$ Then $(X/\!\sim,d_T)$ is a compact metric space, the quotient map
    $$\Phi_T:X\longrightarrow X/\!\sim$$
    is continuous, and
$$(X,d_t)\longrightarrow(X/\!\sim,d_T)~~~~\text{in the Gromov-Hausdorff topology as}~t\to T.$$
In particular the Gromov-Hausdorff limit of $(X,d_t)$ as $t\to T$ is unique.
\item[$(c)$] The map $x\mapsto\mu^{(x)}$ of {\rm Lemma \ref{localization}} induces an isometry $$(X/\!\sim,d_T)\to(\mathcal{M}_T,d_{\mathcal{M}_T})$$ onto Bamler's asymptotic boundary, which is isometric to the time-zero slice $(Z_0,d_0^Z)$ of Fang-Li {\rm\cite[Lemma 9.1]{FL25}}.
\end{itemize}
\end{theorem}
\begin{proof}
For $x,y\in X$ and $t\in[0,T)$ set
$$F_t(x,y):=d^{(t)}_{W_1}\big(\mu^{(x)}_t,\mu^{(y)}_t\big).$$
By (\ref{8-moment}) and the triangle inequality, we have $$|F_t(x,y)-d_t(x,y)|=|d^{(t)}_{W_1}(\mu^{(x)}_t,\mu^{(y)}_t)-d^{(t)}_{W_1}(\delta_x,\delta_y)|\leq2C_3\sqrt{T-t},$$
that is
\begin{equation}\label{8-Fd}
|F_t(x,y)-d_t(x,y)|\leq2C_3\sqrt{T-t}~~~~\text{for all}~x,y\in X,~t\in[0,T).
\end{equation}
By (\ref{8-W1mono}), we obtain that $F_t(x,y)$ is nondecreasing in $t$. Moreover, by \eqref{diameterbd}, it is bounded above. Hence $$\lim_{t\nearrow T}F_t(x,y)=\sup_{t\in[0,T)}F_t(x,y)$$ exists. In view of \eqref{8-Fd}, the same holds for
$$d_T(x,y)=\lim_{t\nearrow T}d_t(x,y).$$
Moreover $$d_T(x,y)=\sup_{t\in[0,T)}F_t(x,y)=d_{\mathcal{M}_T}(\mu^{(x)},\mu^{(y)}).$$
This gives the existence part of $(a)$ (a second proof of Remark \ref{remark-holder}). Together with Lemma \ref{localization} $(d)$, it proves assertion $(c)$. Note that $F_t\leq d_T$, and the triangle inequality for $d^{(t)}_{W_1}$ yields
\begin{equation}\label{8-Lip}
|F_t(x,y)-F_t(x',y)|\leq F_t(x,x')\leq d_T(x,x').
\end{equation}
It follows that $F_t$ is $1$-Lipschitz in each variable with respect to the pseudo-distance $d_T.$

{\bf Step 1: total boundedness.} Let $x^1,\dots,x^N\in X$ satisfy $d_T(x^k,x^l)\geq\varepsilon$ for $k\neq l,$ where $0<\varepsilon\leq1.$ Since there are finitely many pairs, we can choose $t<T$ such that
$$F_t(x^k,x^l)\geq\varepsilon/2~~ ~\forall~k\neq l$$ and
$$2C_3\sqrt{T-t}\leq\varepsilon/8.$$
By (\ref{8-Fd}), we have
$$d_t(x^k,x^l)\geq3\varepsilon/8.$$
Hence the balls $B(x^k,t,3\varepsilon/16)$ are pairwise disjoint.
By (\ref{noncollapse2}) and the volume lower bound in Step 1 of the proof of Theorem \ref{ZhZh22}, each of them has volume at least $c(3\varepsilon/16)^{\theta},$ whereas $|X|_t\leq C.$ Hence $$N\leq N(\varepsilon):=C\varepsilon^{-\theta}.$$
Thus $(X/\!\sim,d_T)$ is totally bounded. Let $(\bar{Z},\bar{d})$ be its completion, which is a compact metric space.

{\bf Step 2: uniform convergence.}
By \eqref{8-Lip}, \(F_t\) descends to the quotient \(X/\!\sim\) and extends continuously to \(\bar{Z}\times\bar{Z}\). We still denote this extension by \(F_t\). It is nondecreasing in \(t\) and satisfies \(F_t\leq\bar{d}\). Now fix \(z,z'\in\bar{Z}\) and \(\epsilon>0\). Choose \(x,x'\in X\) such that
$$\bar{d}(x,z)<\epsilon,~~\bar{d}(x',z')<\epsilon.$$
Then by (\ref{8-Lip})
$$\liminf_{t\to T}F_t(z,z')\geq\lim_{t\to T}F_t(x,x')-2\epsilon=d_T(x,x')-2\epsilon\geq\bar{d}(z,z')-4\epsilon.$$
Hence $F_t\nearrow\bar{d}$ pointwise on the compact space $\bar{Z}\times\bar{Z}.$ Since all functions involved are continuous, Dini's theorem yields
$$\eta(t):=\sup_{\bar{Z}\times\bar{Z}}(\bar{d}-F_t)\longrightarrow0~~~~\text{as}~t\to T.$$
Combined with (\ref{8-Fd}) this yields
\begin{equation}\label{8-unif}
\sup_{x,y\in X}|d_t(x,y)-d_T(x,y)|\leq2C_3\sqrt{T-t}+\eta(t)\longrightarrow0,
\end{equation}
which is $(a).$

{\bf Step 3: compactness of the quotient and Gromov-Hausdorff convergence.} Each $d_t$ is continuous on $X\times X$ in the manifold topology. By (\ref{8-unif}), $d_T$ is their uniform limit and is therefore continuous. In particular, if $x_k\to x$ in $X$, then $d_T(x_k,x)\to0$. Thus the quotient map
$$\Phi_T:X\longrightarrow(X/\!\sim,d_T)$$
is continuous. Its image is compact, so $(X/\!\sim,d_T)$ is complete and agrees with $(\bar Z,\bar d)$.

By (\ref{8-unif}), the surjective map
$$(X,d_t)\longrightarrow(X/\!\sim,d_T)$$
changes distances by at most
$$\epsilon(t):=2C_3\sqrt{T-t}+\eta(t).$$
Its graph is a correspondence of distortion at most $\epsilon(t)$, and hence
$$d_{GH}((X,d_t),(X/\!\sim,d_T))\leq\tfrac12\epsilon(t)\longrightarrow0.$$
This proves $(b)$ and completes the proof.
\end{proof}

\begin{remark}\label{remarkB}
The type-I bound is used only via \eqref{8-moment}, that is, through the localization of the singular-time conjugate heat kernels. Without it, one does not know that every point of $X$ lies at distance $O(\sqrt{T-t})$ from the $H_m$-center of some $\mu\in\mathcal{M}_T$. In \cite{LZ26} this is replaced by the construction of regular points in every ball by means of the curvature radius estimates of Fang-Li \cite{FL25}. Within the proof of Theorem \ref{thmB}, Step 1 is the only place where the noncollapsing estimate \eqref{noncollapse2} is used. \qed
\end{remark}

\subsection{The regular part and the K\"ahler space induced by the limiting class}\label{sec8.4}
By \cite[Theorem 4.1]{CT15}, the metrics $\omega(t)$ converge, as $t\to T$, in $C^{\infty}_{loc}(X^{\circ})$ to a K\"ahler metric $\omega_T$ on $X^{\circ}=X\setminus\mathrm{Null}(\alpha)$, which extends to a positive current on $X$ with $[\omega_T]=\alpha$. We denote by $d_{\omega_T}$ the intrinsic length distance of $(X^{\circ},\omega_T)$, in the normalization of Section \ref{preliminary}, and by
$$\Phi_T:X\longrightarrow Q:=(X/\!\sim,d_T)$$
the quotient map of Theorem \ref{thmB}. The intrinsic distance $d_{\omega_T}$ uses only curves in $X^{\circ}$ and may be incomplete.

\begin{proposition}{\rm(The regular part)}\label{propY}
Under the assumption of {\rm Theorem \ref{localdistortion}}, the following hold.
\begin{itemize}
\item[$(a)$] $\Phi_T|_{X^{\circ}}$ is injective, $\Phi_T(X^{\circ})$ is open in $Q$, and
$$\Phi_T:(X^{\circ},d_{\omega_T})\longrightarrow Q$$
is a $1$-Lipschitz local isometry. In particular the Gromov-Hausdorff limit of $(X,d_t)$ contains the K\"ahler manifold $(X^{\circ},\omega_T)$ as an open subset, and $\Phi_T(X^{\circ})$ is the regular part of the limit in the spacetime sense of {\rm \cite{FL25,LZ26}}. $\sim$ identifies only points of $\mathrm{Null}(\alpha).$
\item[$(b)$] Assume in addition \eqref{omega0-omegay}, with $Y$ connected. Then $k=n$ and $\pi$ is surjective and generically finite. The null locus is the union of the positive-dimensional irreducible components of the fibres of $\pi$. Moreover, $\pi$ descends to a surjective Lipschitz map
$$\bar\pi:Q\longrightarrow(Y,d_{\omega_Y}),\qquad \bar\pi\circ\Phi_T=\pi.$$
This map is a homeomorphism if and only if
\begin{equation}\label{8-fibre}
\sup_{q\in Y}\diam_{d_t}\big(\pi^{-1}(q)\big)\longrightarrow0\qquad\text{as }t\to T.
\end{equation}
\item[$(c)$] If, in addition to $(b)$, $\pi$ is a bimeromorphic contraction onto $Y$, then $\mathrm{Null}(\alpha)=\mathrm{Exc}(\pi)$ and
$$\pi:X^{\circ}\longrightarrow Y\setminus\pi\big(\mathrm{Exc}(\pi)\big)$$
is a biholomorphism. Consequently, $\bar\pi|_{\Phi_T(X^{\circ})}$ is biholomorphic for the complex structure inherited from $X^{\circ}$.
\end{itemize}
\end{proposition}
\begin{proof}
$(a)$ Put $Q^{\circ}:=\Phi_T(X^{\circ})$, and let $g_T=2g_{\omega_T}$ be the limiting Riemannian metric on $X^{\circ}$. We prove the local metric assertions by the geodesic comparison argument of \cite[Lemma 3.9]{Guo19}.

\noindent$\bullet$ {\it $\Phi_T|_{X^{\circ}}$ is injective and $Q^{\circ}$ is open.}
Fix $x\in X^{\circ}$ and choose a neighborhood $U\Subset X^{\circ}$ of $x$. By the smooth convergence $g_t\to g_T$ on $\overline U$, there are $c>0$ and $t_0<T$ such that $g_t\geq cg_0$ on $\overline U$ for $t\geq t_0$. Any curve joining $x$ to $y\in X\setminus U$ must meet $\partial U$. Its initial segment therefore has length at least
$$\rho_U:=\sqrt c\,d_0(x,X\setminus U)>0.$$
Hence $d_t(x,y)\geq\rho_U$, and letting $t\nearrow T$ gives $d_T(x,y)\geq\rho_U$. For any $y\neq x$, we may choose $U$ excluding $y$, so $d_T(x,y)>0$. In particular, no point of $X^{\circ}$ is identified with another point of $X$.

Moreover,
$$B_Q(\Phi_T(x),\rho_U)\subset\Phi_T(U),$$
which proves that $Q^{\circ}$ is open. Choosing $U$ inside any prescribed open neighborhood of $x$ in $X^{\circ}$ also shows that $\Phi_T|_{X^{\circ}}$ is an open map. Since $\Phi_T$ is continuous by Theorem \ref{thmB}, its restriction is a homeomorphism from $X^{\circ}$ onto $Q^{\circ}$.

\noindent$\bullet$ {\it $\Phi_T|_{X^{\circ}}$ is a local isometry.}
Fix $x\in X^{\circ}$. Since $g_T$ is smooth and positive definite, we can choose $r>0$ such that
$$U:=B_{\omega_T}(x,4r)\Subset X^{\circ}$$
is geodesically convex. Put $V:=B_{\omega_T}(x,r)$. Smooth convergence on $\overline U$ gives $\varepsilon_t\to0$ such that, for late $t$,
$$(1-\varepsilon_t)g_T\leq g_t\leq(1+\varepsilon_t)g_T
\quad\hbox{on }\overline U,\qquad 0\leq\varepsilon_t<\frac1{10}.$$
For $p,q\in V$, let $\sigma\subset U$ be a minimizing $g_T$-geodesic joining them. Then
$$d_t(p,q)\leq L_{g_t}(\sigma)
\leq\sqrt{1+\varepsilon_t}\,d_{\omega_T}(p,q)
\leq2r\sqrt{1+\varepsilon_t}.$$
On the other hand, any curve starting at $p\in V$ and leaving $U$ has, up to its first exit, $g_T$-length at least $3r$. Its $g_t$-length is therefore at least
$$3r\sqrt{1-\varepsilon_t}>2r\sqrt{1+\varepsilon_t}.$$
It follows that a minimizing $g_t$-geodesic $\gamma_t$ from $p$ to $q$ stays in $U$. Consequently,
$$d_t(p,q)=L_{g_t}(\gamma_t)
\geq\sqrt{1-\varepsilon_t}\,L_{g_T}(\gamma_t)
\geq\sqrt{1-\varepsilon_t}\,d_{\omega_T}(p,q).$$
Combining the two inequalities yields
\begin{equation}\label{8-localcompare}
\sqrt{1-\varepsilon_t}\,d_{\omega_T}(p,q)
\leq d_t(p,q)\leq
\sqrt{1+\varepsilon_t}\,d_{\omega_T}(p,q),
\qquad p,q\in V.
\end{equation}
Letting $t\nearrow T$, we obtain
$$d_T(p,q)=d_{\omega_T}(p,q),\qquad p,q\in V.$$
Since $\Phi_T(V)$ is open in $Q$, the restriction $\Phi_T|_V$ is an isometry onto an open neighborhood of $\Phi_T(x)$. This proves the local isometry assertion.

For any $p,q\in X^{\circ}$ and any piecewise smooth curve $\gamma\subset X^{\circ}$ joining them, smooth convergence on its compact image gives
$$d_T(p,q)=\lim_{t\nearrow T}d_t(p,q)
\leq\lim_{t\nearrow T}L_{g_t}(\gamma)=L_{g_T}(\gamma).$$
Taking the infimum over $\gamma$ proves $d_T(p,q)\leq d_{\omega_T}(p,q)$. Thus $\Phi_T$ is globally $1$-Lipschitz.

\noindent$\bullet$ {\it $Q^{\circ}$ is the regular part of $Q$.}
Let $I:Q\to(Z_0,d_0^Z)$ be the canonical isometry of Theorem \ref{thmB}$(c)$, and let $Q^{\mathrm{reg}}:=I^{-1}(Z_0\cap\mathcal R)$, where $\mathcal R$ is the regular spacetime of the completion. For $x\in X^{\circ}$, the flow extends smoothly to $T$ on a neighborhood of $x$. The worldline $(x,t)$ therefore has a regular terminal endpoint. Its backward kernels are $\mu^{(x)}$ by Lemma \ref{localization}$(a)$, so this endpoint is $I(\Phi_T(x))$; see \cite[Lemma 2.2]{LZ26}. Hence $Q^{\circ}\subset Q^{\mathrm{reg}}$.

Conversely, let $z\in Q^{\mathrm{reg}}$. By the regular-worldline argument of \cite[Proposition 2.3]{LZ26}, $I(z)$ is the endpoint of $(x,t)$ for a fixed $x\in X$ near $T$, and there are a fixed neighborhood $U_x$ and a constant $c>0$ such that $\omega(t)\geq c\omega_0$ on $U_x$ for late $t$. Lemma \ref{localization}$(a)$ identifies the backward kernels of this endpoint with $\mu^{(x)}$. Therefore $z=\Phi_T(x)$.

If $x\in\mathrm{Null}(\alpha)$, there is a positive-dimensional irreducible analytic subvariety $W$ through $x$, of dimension $l$, such that $\int_W\alpha^l=0$. However,
$$\int_W\omega(t)^l
\geq c^l\int_{W\cap U_x}\omega_0^l>0,$$
uniformly for late $t$, whereas
$$\int_W\omega(t)^l\longrightarrow\int_W\alpha^l=0.$$
This is a contradiction. Thus $x\in X^{\circ}$, and $Q^{\mathrm{reg}}\subset Q^{\circ}$.

Finally, $X^{\circ}$ is dense in $X$, and $\Phi_T:X\to Q$ is continuous and surjective. Hence $Q^{\circ}$ is dense in $Q$. The injectivity proved above shows that $\sim$ identifies only points of $\mathrm{Null}(\alpha)$. This completes the proof of $(a)$.

$(b)$ Since $\alpha=[\pi^*\omega_Y]$ and
\[
\alpha^n=\big([\omega_0]-2\pi Tc_1(X)\big)^n=\int_X\beta^n>0,
\]
we must have $k=n$. Indeed, if $k<n$, then
$(\pi^*\omega_Y)^n=0$ in cohomology, contradicting $\alpha^n>0$.
Thus $\pi$ is a holomorphic map between compact complex manifolds of the
same dimension. Since $\alpha^n=(\pi^*\omega_Y)^n>0$, the map $\pi$ is
dominant; as $Y$ is connected, it is surjective. Moreover, $\pi$ is
generically finite.

We next identify the null locus. Let $W\subset X$ be an irreducible
analytic subvariety of dimension $d$. Then
\[
\int_W\alpha^d=\int_W(\pi^*\omega_Y)^d.
\]
If $\dim\pi(W)=d$, then $\pi|_W:W\to\pi(W)$ is generically finite and
\[
\int_W(\pi^*\omega_Y)^d
=\deg(\pi|_W)\int_{\pi(W)}\omega_Y^d>0.
\]
If $\dim\pi(W)<d$, then the pullback of the top-degree form
$\omega_Y^d$ vanishes generically, so the integral is zero. Hence
\begin{equation}\label{8-rankcriterion}
\int_W\alpha^d=0
\quad\Longleftrightarrow\quad
\dim\pi(W)<d.
\end{equation}
By the fibre dimension theorem, the latter condition means that every
point of $W$ lies on a positive-dimensional irreducible component of its
fibre $\pi^{-1}(\pi(w))$. Conversely, if $F$ is a positive-dimensional
irreducible component of a fibre of $\pi$, then $\pi(F)$ is a point, so
$\int_F\alpha^{\dim F}=0$. Therefore the null locus is exactly the union
of all positive-dimensional irreducible components of the fibres of
$\pi$.

By the parabolic Schwarz lemma \cite{ST07,ST12}, there is a constant
$C<\infty$ such that
\[
\pi^*\omega_Y\le C\omega(t)\qquad\text{on }X\times[0,T).
\]
Hence, for all $x,y\in X$ and $t\in[0,T)$,
\[
d_{\omega_Y}\big(\pi(x),\pi(y)\big)\le C^{1/2}d_t(x,y).
\]
Letting $t\to T$ and using the uniform convergence $d_t\to d_T$ from
Theorem \ref{thmB}(a), we obtain
\[
d_{\omega_Y}\big(\pi(x),\pi(y)\big)\le C^{1/2}d_T(x,y).
\]
Thus $d_T(x,y)=0$ implies $\pi(x)=\pi(y)$, so $\pi$ descends to a
surjective Lipschitz map
\[
\bar\pi:Q\longrightarrow(Y,d_{\omega_Y}),\qquad
\bar\pi\circ\Phi_T=\pi.
\]

It remains to prove the equivalence with \eqref{8-fibre}. Since $Q$ is
compact and $Y$ is Hausdorff, the continuous surjection $\bar\pi$ is a
homeomorphism if and only if it is injective. Injectivity is equivalent
to
\[
\diam_{d_T}\big(\pi^{-1}(q)\big)=0
\qquad\text{for every }q\in Y.
\]
Set
\[
\delta(t):=\sup_{x,y\in X}|d_t(x,y)-d_T(x,y)|.
\]
By Theorem \ref{thmB}(a), $\delta(t)\to0$ as $t\to T$. For every $q\in Y$
we have
\[
\left|\diam_{d_t}\big(\pi^{-1}(q)\big)
-\diam_{d_T}\big(\pi^{-1}(q)\big)\right|\le\delta(t).
\]
Therefore, if $\bar\pi$ is injective, then
$\diam_{d_T}(\pi^{-1}(q))=0$ for all $q$, and hence
\[
\sup_{q\in Y}\diam_{d_t}\big(\pi^{-1}(q)\big)\le\delta(t)\to0.
\]
Conversely, if
\[
\sup_{q\in Y}\diam_{d_t}\big(\pi^{-1}(q)\big)\to0,
\]
then for each fixed $q$,
\[
\diam_{d_T}\big(\pi^{-1}(q)\big)
\le\diam_{d_t}\big(\pi^{-1}(q)\big)+\delta(t)\to0,
\]
so $\diam_{d_T}(\pi^{-1}(q))=0$. Thus $\bar\pi$ is injective, hence a
homeomorphism. This proves $(b)$.

$(c)$ A proper bimeromorphic map onto a normal target has connected fibres. Where its fibres are zero-dimensional it is locally finite and birational, hence an isomorphism. Thus its exceptional locus is exactly the locus of positive-dimensional fibre components, which equals $\mathrm{Null}(\alpha)$ by \eqref{8-rankcriterion}. Removing the exceptional locus and its image gives the stated biholomorphism. Composing with the inverse of $\Phi_T|_{X^{\circ}}$ proves the assertion about $\bar\pi$.
\end{proof}

\begin{remark}\label{fibreremark}
The contraction assumption in Proposition \ref{propY}$(c)$ cannot be inferred from bigness. For instance, let
$$b:X=\mathrm{Bl}_{[1:1:1]}\mathbb P^2\to\mathbb P^2$$
be the blow-up, with exceptional curve $E$, and put
$$H=b^*(2\pi c_1(\mathcal O(1)))~~~ \text{and}~~~e=2\pi c_1(\mathcal O_X(E)).$$
The K\"ahler-Ricci flow with $[\omega_0]=5H-e$ has maximal time $T=1$ and limiting class $\alpha=2H$, with $\mathrm{Null}(\alpha)=E$. The map $\pi=q_2\circ b$, where $$q_2([z_0:z_1:z_2])=[z_0^2:z_1^2:z_2^2],$$
satisfies $\alpha=\pi^*[\omega_Y]$ for a Fubini--Study form in $2\pi c_1(\mathcal O(1))$. Nevertheless, $\pi$ is generically four-to-one and ramifies outside $E$, and its fibre over $[1:1:1]$ consists of $E$ and three isolated regular points.

For a bimeromorphic contraction, Proposition \ref{propY} reduces the identification of the Gromov--Hausdorff limit with the target, equipped with the transported limit distance, to the fibre collapse \eqref{8-fibre}. The latter is known for the contractions considered in \cite[Lemma 3.2]{SW13}, where the ambient fibre diameter is bounded by $C(T-t)^{1/3}$; see also \cite[Lemma 3.10]{Guo19}. The stronger Type I fibre diameter bound
$$\diam_{d_t}\pi^{-1}(q)\leq C\sqrt{T-t}$$
is the second half of \cite[Conjecture 1.1]{JST23}. It holds for the solutions with Calabi symmetry considered in \cite[Section 10]{JST23}, while further diameter estimates are announced in \cite[Section 1]{JS26}. Our heat kernel estimates do not establish this bound. Lemma \ref{localization} localizes a conjugate heat kernel measure around its base point, not around a fibre which is independently of fibre collapse. The quotient $Q$ is canonical, its regular part is $\Phi_T(X^{\circ})$, and its singular part has Hausdorff dimension at most $2n-4$ by \cite[Theorem 1.2]{LZ26}. \qed
\end{remark}

\subsection{Type I blow-up limits}\label{sec8.5}
Let $x\in X,$ $t_i\nearrow T$ and $\tau_i:=T-t_i.$ The {\it Type I blow-ups} of the flow at $x$ are the rescaled flows
\begin{equation}\label{8-rescaled}
g_i(s):=\tau_i^{-1}g(T+\tau_is),~~~~s\in[-T/\tau_i,0),
\end{equation}
so that $g_i(-1)=\tau_i^{-1}g(t_i).$ We are interested in the pointed limits of the time slices $(X,\tau_i^{-1/2}d_{t_i},x)$ at $s=-1.$ By Theorem \ref{JSthm}, the rescaled scalar curvature satisfies
\begin{equation}\label{8-rescaledR}
|R_{g_i(s)}|\leq\frac{C_0'}{|s|}
\qquad\text{on }X\times[-T/\tau_i,0),
\end{equation}
where \(C_0'=\max\{C_0,R_0T\}\). Hence it is uniformly bounded on every compact subinterval of \((-\infty,0)\). Meanwhile, the Nash entropy of the flows (\ref{8-rescaled}) is bounded below by \(\nu[g_0,T]\) (see (\ref{nashentropy-2side})). We equip the rescaled flows (\ref{8-rescaled}) with the conjugate heat flow \(\mu^{(x)}_{T+\tau_i s}\) from Lemma \ref{localization}, based at the singular time \(s=0\) of the rescaled flows. The main result of this section is as follows.

\begin{theorem}{\rm (Type I blow-up limits)}\label{thmA}
Under the assumption of {\rm Theorem \ref{localdistortion}}, let $x\in X$ and $t_i\nearrow T.$ After passing to a subsequence, the following hold.
\begin{itemize}
\item[$(1)$] The rescaled flows \eqref{8-rescaled}, with the conjugate heat flows $\mu^{(x)}_{T+\tau_is},$ converge in the $\mathbb{F}$-sense of Bamler {\rm\cite{Bam20b}} on compact subintervals of $(-\infty,0)$ to a metric flow $(\mathcal{X},(\mu^{\infty}_s)_{s<0})$ which is a {\it metric soliton}, with constant Nash entropy
    $$\mathcal{N}_{\mu^{\infty}}(\tau)\equiv\Theta(x):=\lim_{\tau\to0}\mathcal{N}_{\mu^{(x)}}(\tau)\in[\nu[g_0,T],0].$$
\item[$(2)$] The time $(-1)$ slice $(Z,d_Z)$ of $\mathcal{X}$ is a singular space in the sense of Bamler {\rm\cite{Bam20c}} of real dimension $2n$. Its singular set has Minkowski dimension at most $2n-4,$ and its regular part $\mathcal{R}_Z$ carries a smooth (in general incomplete) shrinking K\"ahler-Ricci soliton $(g_Z,J_Z,f_Z),$ $$Ric(g_Z)+\nabla^2f_Z=\frac{1}{2}g_Z,$$
    where $J_Z\nabla f_Z$ is holomorphic and Killing, and generates a complete flow.
\item[$(3)$] In the pointed Gromov--Hausdorff topology,
    $$(X,\tau_i^{-1/2}d_{t_i},x)\to(Z,d_Z,z_{\infty}),$$
    and the convergence is smooth on $\mathcal{R}_Z$. Moreover, $z_{\infty}$ lies within distance \(C_3+\sqrt{H_m}\) of an \(H_m\)-center of \(\mu^{\infty}_{-1}\). Equivalently, the type I blow-up limit of the flow at \(x\) is the slice at time \((-1)\) of a tangent flow of the flow at the point \(\mu^{(x)}\in\mathcal{M}_T\).
\item[$(4)$] \(Z\) carries the structure of a normal complex analytic variety with log terminal singularities, and its analytic singular set is the metric singular set \(Z\setminus\mathcal{R}_Z\). The form \(\omega_Z\) extends to a K\"ahler current on \(Z\) with locally bounded potentials. At every point, the tangent cone of \(Z\) is unique, and it is a Ricci-flat K\"ahler cone on an affine variety with log terminal singularities. Outside a subset of complex codimension three, \(Z\) is a smooth orbifold. Moreover, \(Z\) is simply connected, and it admits a polarized Fano fibration in the sense of {\rm \cite{SZ24,HZ26}}.
\item[$(5)$] $\Theta(x)$ does not depend on the subsequence, and $\Theta(x)=0$ if and only if $Z$ is the Gaussian shrinker on $\mathbb{C}^n,$ if and only if $x\in X^{\circ}.$
\end{itemize}
\end{theorem}
\begin{proof}
$(1)$ By Lemma \ref{localization}, we have $\mu^{(x)}\in\mathcal{M}_T$. Moreover, the rescaled flows (\ref{8-rescaled}), with base flows $\mu^{(x)}_{T+\tau_is}$, coincide precisely with the parabolic rescalings considered in \cite[Theorem 2.37]{Bam20c}, where they are denoted by
$$\big(M,(\tau_i^{-1}g_{\tau_it+T})_{t},(\nu_{x,T;\tau_it+T})_{t}\big).$$
Recall that the pointed Nash entropy
$$\tau\mapsto\mathcal{N}_{\mu^{(x)}}(\tau)$$
is nonincreasing. Hence
$$\lim_{\tau\to0}\mathcal{N}_{\mu^{(x)}}(\tau)$$
exists. By Bamler \cite[Theorem 2.18, Theorem 2.37]{Bam20c}, for any $\tau_i\to0$, there exists a subsequence of the rescaled flows which $\mathbb{F}$-converges on compact time intervals to a metric soliton with Nash entropy $\Theta(x).$ The lower bound
$$\Theta(x)\geq\nu[g_0,T]$$
is precisely (\ref{nashentropy-2side}).

$(2)$ A metric soliton obtained as a noncollapsed $\mathbb{F}$-limit of smooth Ricci flows is described by \cite[Theorem 2.18]{Bam20c}. Its time slices are the rescalings $|s|^{1/2}d_Z$ of a fixed singular space $(Z,d_Z)$, whose singular set has Minkowski dimension at most $2n-4$ by \cite[Theorem 2.18 and Theorem 2.16(d)]{Bam20c}. On the regular part, the limit of the conjugate heat kernel measures is
$$(4\pi)^{-n}e^{-f_Z}dg_Z,$$
with
$$Ric+\nabla^2 f_Z=\frac12 g_Z.$$
As the rescaled flows (\ref{8-rescaled}) are K\"ahler with parallel complex structures, their complex structures converge smoothly on the regular part, together with the metrics. Hence the limit is K\"ahler. The holomorphicity and Killing property of $J_Z\nabla f_Z$, together with the completeness of its flow on $\mathcal{R}_Z$, are established in Hallgren-Jian \cite[Theorem 1.3]{HJ23}.

$(3)$ We use the fixed-base-point dilation compactness theorem \cite[Theorem 1.2]{Bam18}, in the form recalled in \cite[Theorem 1.1 and Section 5]{Hall23}. The scalar type-I and finite-scale entropy bounds give, after a further subsequence, pointed Gromov--Hausdorff convergence of $(X,\tau_i^{-1/2}d_{t_i},x)$ to a singular shrinking space, with smooth convergence on its regular part. The convergence of the rescaled conjugate heat kernels in \cite[Section 5]{Hall23}, together with their uniqueness from Lemma \ref{localization}, identifies this measured shrinking limit with the model $(Z,d_Z)$ of part $(1)$. 

If $p_i$ is an $H_m$-center of $\mu^{(x)}$ at $t_i$, then
$$d_{t_i}(x,p_i)\leq\int d_{t_i}(x,y)d\mu^{(x)}_{t_i}(y)
+\int d_{t_i}(y,p_i)d\mu^{(x)}_{t_i}(y)
\leq(C_3+\sqrt{H_m})\sqrt{\tau_i}.$$
After rescaling, the points $p_i$ remain at bounded distance from $x$. Passing to a further subsequence, they converge to a point $p_\infty\in Z$. The center inequality and lower semicontinuity give
$$\int_Z d_Z(p_\infty,y)^2d\mu^\infty_{-1}(y)\leq H_m.$$
Thus $p_\infty$ is an $H_m$-center of $\mu^\infty_{-1}$, and
$$d_Z(z_\infty,p_\infty)\leq C_3+\sqrt{H_m}.$$
Smooth convergence on the regular part is included in the cited dilation compactness and heat-kernel convergence results.

$(4)$ To apply \cite[Setting (I), Section 2.2]{HZ26}, we approximate by flows defined through a regular terminal time. Choose $\varepsilon_i\downarrow0$, set $t_i'=T-\varepsilon_i\tau_i$, and choose an $H_m$-center $p_i$ of $\mu^{(x)}$ at $t_i'$. Concentration and \eqref{8-W1mono} give, for every $l\leq t_i'$,
$$d^{(l)}_{W_1}\big(\mu^{(x)}_l,\nu_{p_i,t_i';l}\big)
\leq\sqrt{H_m\varepsilon_i\tau_i}.$$
Translate the regular pole time to zero and define
$$\widetilde g_i(u):=\tau_i^{-1}g(t_i'+\tau_i u)=g_i(u-\varepsilon_i),
\qquad -t_i'/\tau_i\leq u\leq0,$$
with genuine conjugate heat kernels $\nu_{p_i,t_i';t_i'+\tau_i u}$. Their rescaled $W_1$-distance from $\mu^{(x)}_{t_i'+\tau_i u}$ is at most $\sqrt{H_m\varepsilon_i}$. Vanishing time translations and this uniform approximation on compact negative-time intervals give the same $\mathbb F$-limit, by Bamler's theory \cite{Bam20b,Bam20c}.

The genuine terminal kernels have a uniform Nash-entropy lower bound at scale one in the rescaled flows, directly from Perelman's finite-scale bound. Their ancient domains exhaust $(-\infty,0]$, and part $(1)$ identifies their limit as a metric soliton. Thus the precise hypotheses of setting (I) are satisfied. Theorems A, B, and C of \cite{HZ26}, version 2, give the asserted analytic, log-terminal, tangent-cone, orbifold, and polarized Fano fibration structure, as well as simple connectedness.

$(5)$ By Lemma \ref{localization}, for every \(x\in X\) the conjugate heat flow
\(\mu^{(x)}\in\mathcal M_T\) is uniquely determined. By
\cite[Theorem 2.37]{Bam20c}, the pointed Nash entropy
\[
\tau\longmapsto \mathcal N_{\mu^{(x)}}(\tau)
\]
is nonincreasing in \(\tau\), and by \eqref{nashentropy-2side} it is bounded
below by \(\nu[g_0,T]\). Hence the limit
\[
\Theta(x):=\lim_{\tau\to0}\mathcal N_{\mu^{(x)}}(\tau)
\]
exists and is independent of the choice of subsequence.

To compare with the heat kernel density in \cite{Hall23}, write
$d\mu^{(x)}_{T-\tau}=(4\pi\tau)^{-n}e^{-f_\tau}dg_{T-\tau}$ and set
$$\mathcal W_x(\tau):=\mathcal W(g(T-\tau),f_\tau,\tau).$$
Perelman's entropy identity gives
$$\frac{d}{d\tau}\big(\tau\mathcal N_{\mu^{(x)}}(\tau)\big)
=\mathcal W_x(\tau).$$
Passing to the limit from smooth terminal kernels in Perelman's entropy bounds gives
$$\nu[g_0,T]\leq\mathcal W_x(\tau)\leq0,\qquad 0<\tau<T.$$
The Nash entropy is bounded near zero, so integration from zero yields
$$\mathcal N_{\mu^{(x)}}(\tau)
=\frac{1}{\tau}\int_0^\tau\mathcal W_x(a)\,da.$$
Since $\mathcal W_x$ is monotone and bounded, its limit at zero is therefore
$\Theta(x)$. By the uniqueness in Lemma \ref{localization}, the infimum
over singular-time kernels in the definition of the heat kernel density
in \cite[Section 3]{Hall23} is taken over the single kernel $\mu^{(x)}$.
Thus \cite[Theorem 1.1]{Hall23} identifies \(\Theta(x)\) with the shrinker
entropy of every Type-I blow-up limit at \(x\). Consequently all Type-I blow-up limits at \(x\)
have the same entropy, namely \(\Theta(x)\). By the rigidity of
Perelman's entropy, see \cite[Theorem 1.1]{Hall23} and
\cite[Theorem 2.37]{Bam20c}, we have
\begin{equation}\label{eq:ThetaGaussian}
\Theta(x)=0
\quad\Longleftrightarrow\quad
Z \text{ is the Gaussian shrinker on } \mathbb C^n,
\end{equation}
and in this case the curvature of \(\omega(t)\) is bounded near \(x\) up
to time \(T\).

It remains to identify the points at which the curvature is bounded up to
the singular time. By \cite[Theorem 4.1]{CT15}, the metrics \(\omega(t)\)
converge in \(C^\infty_{\mathrm{loc}}(X^\circ)\) to the K\"ahler metric
\(\omega_T\) on
\[
X^\circ=X\setminus\operatorname{Null}(\alpha).
\]
Hence the curvature is bounded near \(x\) up to time \(T\) whenever
\(x\in X^\circ\). Conversely, if the curvature is bounded near \(x\) up
to time \(T\), then parabolic regularity shows that the flow extends
smoothly to \(t=T\) in a neighborhood of \(x\), so the metric has a uniform positive lower bound on a neighborhood of $x$. If $x$ belonged to the null locus, an analytic subvariety $W$ through $x$ with $\int_W\alpha^{\dim W}=0$ would contradict the uniform positive volume of $W$ on that neighborhood, exactly as in the proof of Proposition \ref{propY}$(a)$. Hence $x\in X^\circ$; see also \cite[Proposition 2.3]{LZ26}. Thus
\begin{equation}\label{eq:curvX0}
x\in X^\circ
\quad\Longleftrightarrow\quad
\text{the curvature is bounded near }x\text{ up to time }T.
\end{equation}
Conversely, if $x\in X^\circ$, smooth extension near $x$ implies that the pointed parabolic blow-ups converge to the flat flow on $T_xX\cong\mathbb C^n$, so $Z$ is the Gaussian shrinker. Combining this observation with \eqref{eq:ThetaGaussian} and \eqref{eq:curvX0} yields
\[
\Theta(x)=0
\quad\Longleftrightarrow\quad
Z \text{ is the Gaussian shrinker on } \mathbb C^n
\quad\Longleftrightarrow\quad
x\in X^\circ .
\]
\end{proof}

\begin{remark}\label{remarkA}
$(1)$ For tangent flows at a fixed base point, Theorem \ref{thmA} establishes the algebraic part of \cite[Conjecture 2.2]{JST23}: the time slices admit a polarized Fano fibration in the sense of Sun--Zhang \cite{SZ24}. The singular counterpart of this structure was proved by Hallgren--Zhang \cite{HZ26}. In the tangent-flow case, the soliton property is a
consequence of the monotonicity of the Nash entropy \cite[Theorem 2.37]{Bam20c}. For blow-ups along moving base points, the soliton property requires an additional entropy-pinching condition; see Proposition \ref{propC} below.

$(2)$ The complex analytic structure across the singular set of $Z$ is not contained in the results of \cite{Bam20c,HJ23,SZ24} alone. Indeed, \cite{Bam20c} describes the metric structure of the singular space, \cite{HJ23} establishes the K\"ahler and holomorphic-Killing properties of the soliton on the regular part, and \cite{SZ24} treats smooth K\"ahler--Ricci shrinkers. The required ingredient is supplied by \cite[Theorems A, B, C]{HZ26}, which show that $Z$ is a normal complex analytic variety with log terminal singularities, that its analytic singular set coincides with the metric singular set, and that $\omega_Z$ extends to a K\"ahler current with locally bounded potentials. The proof uses H\"ormander's $L^2$-estimate on the regular part of the shrinker and then extends the resulting holomorphic and plurisubharmonic objects across the singular set, following the strategy of Liu--Sz\'ekelyhidi \cite{LG22,LG19}. In this argument the conjugate heat kernel measure of the shrinker plays the role that heat kernel estimates play in the smooth
compact setting. In the approach announced in \cite[Section 1]{JS26}, this ingredient is to be supplied instead by the local partial $C^0$-estimates along the flow developed in \cite[Section 8]{JST23} and \cite{HJST23}.

$(3)$ The shrinker $Z$ obtained in Theorem \ref{thmA} may a priori depend on the subsequence and on the choice of the sequence $t_i\nearrow T$. Only the limiting Nash entropy $\Theta(x)$ is independent of these choices; see Theorem \ref{thmA}(5). Uniqueness of the tangent flow is known when the flow admits a holomorphic approximation by an asymptotically conical K\"ahler--Ricci shrinker, in which case the tangent flow is unique and coincides with that shrinker; see \cite[Theorem A]{CHL26}. In general, uniqueness is part of the conjectural picture in
\cite[Conjecture 6.8 and Remark 6.9]{SZ24}. The related
\cite[Conjecture 1.8]{CHL26} characterizes holomorphic approximation by a
shrinker in terms of the associated Fano fibration germs.
\end{remark}

\subsection{Blow-ups along moving base points}\label{sec8.6}
Let now $(x_i,t_i)\in X\times[0,T)$ be arbitrary with $t_i\nearrow T,$ $\tau_i:=T-t_i,$ and consider the rescaled flows (\ref{8-rescaled}) with the base flows $\mu^{(i)}:=\mu^{(x_i)}$ of Lemma \ref{localization}. By \cite[Theorem 2.14]{Bam20c}, a subsequence of the rescaled flows \eqref{8-rescaled} $\mathbb{F}$-converges on compact subintervals of $(-\infty,0)$ to a metric flow $(\mathcal{X}^{\infty},(\mu^{\infty}_s)_{s<0}).$

\begin{proposition}{\rm(Moving base points)}\label{propC}
Under the assumption of Theorem {\rm\ref{localdistortion}}, the following hold for the chosen limit.
\begin{itemize}
\item[$(a)$] The limit is a noncollapsed ancient metric flow whose regular spacetime carries a smooth K\"ahler-Ricci flow with $|R|\leq C_0'/|s|$ for $s<0$.
\item[$(b)$] For the chosen subsequence, the measured metric-flow pair $(\mathcal{X}^{\infty},\mu^{\infty})$ is a metric soliton if and only if
\begin{equation}\label{8-pinch}
\mathcal{N}_{\mu^{(i)}}(\tau_i/K)-\mathcal{N}_{\mu^{(i)}}(K\tau_i)\longrightarrow0~~~~\text{for every}~K>1.
\end{equation}
In this case the conclusions $(2)$ and $(4)$ of {\rm Theorem \ref{thmA}} hold for its time slices.
\item[$(c)$] Suppose a future completion is specified and the hypotheses of {\rm \cite[Theorems 2.6, 2.11, and 2.14]{Bam20c}} hold at the point under consideration. Then each tangent flow is a metric soliton with the conclusions of parts $(2)$ and $(4)$ of {\rm Theorem \ref{thmA}}. These conclusions also hold for tangent flows at infinity, using the entropy lower bound at all finite scales and {\rm \cite[Theorem 2.40]{Bam20c}}.
\end{itemize}
\end{proposition}
\begin{proof}
$(a)$ follows from the regular-spacetime structure and smooth convergence in \cite{Bam20b,Bam20c}; the complex structures are parallel and pass to the regular limit. The scalar estimate follows from \eqref{8-rescaledR}. For $(b)$, entropy convergence in \cite[Theorem 2.14]{Bam20c} gives
$$\mathcal N_{\mu^\infty}(a)=\lim_i\mathcal N_{\mu^{(i)}}(\tau_i a),\qquad a>0.$$
If \eqref{8-pinch} holds, any two positive scales lie in $[1/K,K]$ for some $K>1$. Monotonicity bounds their limiting entropy difference by the drop in \eqref{8-pinch}, so the difference is zero. The limiting entropy is constant, and \cite[Theorem 2.18]{Bam20c} makes the measured limit a metric soliton. Conversely, constancy and entropy convergence give \eqref{8-pinch}. The genuine-kernel approximation in Theorem \ref{thmA}$(4)$ applies with $x_i$ in place of $x$ and yields the analytic conclusions.

For $(c)$, use the tangent-flow theorems for general noncollapsed metric-flow limits cited in the statement, after specifying the future completion and checking their pointwise hypotheses. The original compact-flow boundary theorem, \cite[Theorem 2.37]{Bam20c}, is not substituted for these general statements. At infinity, the entropy lower bound at every fixed finite scale follows from the initial finite-scale entropy bound before rescaling; \cite[Theorem 2.40]{Bam20c} applies. The same genuine-kernel approximation puts the resulting solitons into the setting of \cite{HZ26}.
\end{proof}

\begin{remark}\label{openremark}
The statement that every Type I blow-up limit along an arbitrary sequence $(x_i,t_i)$ is the self-similar flow of a K\"ahler-Ricci soliton does not follow from the above. Condition \rm{(\ref{8-pinch})} concerns the pair (flow, base measure). For illustration, suppose a compact flow admits a tangent flow given by the Feldman--Ilmanen--Knopf shrinker, and consider a moving sequence whose convergence maps identify the limiting conjugate heat kernel measure with a boundary measure of that shrinker not based at the vertex. The underlying metric flow may then be self-similar while the selected conjugate heat kernel measure has nonconstant Nash entropy. In particular, prescribing the rescaled distance of the base points from the exceptional divisor does not by itself determine this identification.

The weaker statement that the limit flow is self-similar would follow if one could always find $y_i$ with
$$d_{t_i}(x_i,y_i)\leq C\sqrt{\tau_i}$$
such that (\ref{8-pinch}) holds for $\mu^{(y_i)}$. The counting of bad dyadic scales in \cite[Section 7]{JS26} bounds the number of bad scales of a given point, but it does not produce a good point at a prescribed scale.

The question of whether the blow-up limits based at Ricci vertices are tangent flows is raised in \cite[Section 1]{HJST23} and \cite[Conjecture 2.2, Remark 2.3]{JST23}. For a limit that is not self-similar, one available route to a normal complex analytic structure is the partial \(C^0\)-estimate of Jian-Song-Tian \cite[Section 8]{JST23} and Hallgren \cite{HJST23}, since Hallgren-Zhang \cite{HZ26} uses the shrinker structure in an essential way. \qed
\end{remark}

\begin{remark}\label{comparisonremark}
We summarize the relation of this section to \cite{JS26} and \cite{LZ26}. The forthcoming part II of Jian-Song \cite{JS26} announces diameter estimates, the identification of the Gromov-Hausdorff limit with the K\"ahler space induced by the limiting class, partial $C^0$-estimates, and the soliton and normal analytic structure of type I blow-up limits, by means of the global Ricci potential of \cite{JS26} and the techniques of \cite{JST23,HJST23}. Li--Zhang \cite{LZ26} prove the uniqueness of the singular-time Gromov--Hausdorff limit, identify its regular part with \(X^{\circ}\), and establish the codimension-four estimate for its singular set. Their proof uses the spacetime completion of Fang--Li and the quantitative curvature radius estimates of \cite{FL25}. Their later sections also study rescaled geometry and Ricci shrinkers. In the present section, the uniqueness of the limit (Theorem \ref{thmB}) is instead obtained from Theorem \ref{hke2} and the monotonicity of the $W_1$-distance, with the uniform convergence of the distance functions as a by-product. Proposition \ref{propY} identifies the regular part of the limit by combining local metric convergence with the terminal-time spacetime argument of Li--Zhang \cite{LZ26}; the identification is then completed by the null-volume argument of Collins--Tosatti \cite{CT15}. This reduces the identification of the limit with $(Y,\bar d)$ to the fibre-collapsing condition \eqref{8-fibre}. Once Lemma \ref{localization} is available, the soliton and normal analytic structure of the blow-up limits at a fixed base point (Theorem \ref{thmA}) follow from \cite{Bam20c,JST23,HJ23,HZ26}.
\end{remark}



\end{document}